\documentclass[11pt,twoside]{amsart}

\usepackage{latexsym}
\usepackage{amssymb}
\usepackage{amsfonts}
\usepackage{amstext}
\usepackage{multicol}
\usepackage{mathtools}
\usepackage{amscd,amssymb,amsfonts,amsmath,amstext,amsthm,verbatim,mathdots}
\usepackage{comment,array,blkarray,multirow,tikz,graphicx,ifthen}
\usepackage [margin = 1.3in] {geometry}
\usepackage{youngtab}
\usepackage{ytableau}
\usepackage{cancel}
\usepackage{hyperref}
\numberwithin{equation}{section}

\newtheorem{thm}{Theorem}[section]
\newtheorem{lem}[thm]{Lemma}
\newtheorem{cor}[thm]{Corollary}
\newtheorem{prop}[thm]{Proposition}

\theoremstyle{definition}
\newtheorem{defn}[thm]{Definition}
\newtheorem{conv}[thm]{Convention}
\newtheorem{ex}[thm]{Example}
\newtheorem{remark}[thm]{Remark}
\newtheorem{defn_prop}[thm]{Definition/Proposition}

\newenvironment{sproof}{\proof}{\endproof}

\newcommand{\bb}[1]{\mathbb{#1}}
\newcommand{\bbZ}{\bb{Z}}
\newcommand{\bbC}{\bb{C}}

\newcommand{\bbR}{\bb{R}}
\newcommand{\bbX}{\bb{X}}
\newcommand{\tw}[1]{\widetilde{#1}}
\newcommand{\wh}[1]{\widehat{#1}}

\newcommand{\SL}{{\rm SL}}
\newcommand{\GL}{{\rm GL}}
\newcommand{\Gr}{{\rm Gr}}
\newcommand{\Id}{{\rm Id}}

\DeclareMathOperator{\sh}{sh}
\DeclareMathOperator{\fl}{fl}
\DeclareMathOperator{\wt}{wt}
\DeclareMathOperator{\sgn}{sgn}
\DeclareMathOperator{\Trop}{Trop}

\DeclareMathOperator{\PR}{PR}

\DeclareMathOperator{\inv}{inv}

\newcommand{\ov}[1]{\overline{#1}}
\newcommand{\mb}[1]{\mathbf{#1}}
\newcommand{\ds}{\displaystyle}

\newcommand{\Cx}{\mathbb{C}^\times}
\newcommand{\Y}[2]{\Gr(#1,#2) \times \Cx}
\newcommand{\X}[1]{\bbX_{#1}}

\newcommand{\Zn}{\mathbb{Z}/n\mathbb{Z}}

\newcommand{\lp}{\lambda} 

\newcommand{\bT}[1]{\mathbb{T}_{#1}}
\newcommand{\tT}[1]{\tw{\mathbb{T}}_{#1}}

\newcommand{\ve}{\varepsilon}
\newcommand{\vp}{\varphi}
\newcommand{\tep}{\tw{\varepsilon}}
\newcommand{\tph}{\tw{\varphi}}
\newcommand{\te}{\tw{e}}
\newcommand{\tg}{\tw{\gamma}}

\title{The geometric $R$-matrix for affine crystals of type $A$}
\author{Gabriel Frieden}
\date{\today}
\thanks{The author was supported in part by NSF grants DMS-1464693 and DMS-0943832.}
\address{Department of Mathematics, University of Michigan, Ann Arbor, MI 48109-1043, USA}
\email{gfrieden@umich.edu}

\begin{document}
\maketitle

\begin{abstract}
In [Frieden, \texttt{arXiv:1706.02844}], we constructed a geometric crystal on the variety $\X{k} := \Y{k}{n}$ which tropicalizes to the affine crystal structure on rectangular tableaux with $n-k$ rows. In this sequel, we define and study the {\em geometric $R$-matrix}, a birational map $R : \X{k_1} \times \X{k_2} \rightarrow \X{k_2} \times \X{k_1}$ which tropicalizes to the combinatorial $R$-matrix on pairs of rectangular tableaux. We show that $R$ is an isomorphism of geometric crystals, and that it satisfies the Yang--Baxter relation. In the case where both tableaux have one row, we recover a birational action of the symmetric group that has appeared in the literature in a number of contexts. We also define a rational function $E : \X{k_1} \times \X{k_2} \rightarrow \bbC$ which tropicalizes to the coenergy function from affine crystal theory.

Most of the properties of the geometric $R$-matrix follow from the fact that it gives the unique solution to a certain equation of matrices in the loop group $\GL_n(\bbC(\lp))$.
\end{abstract}

\tableofcontents

\ytableausetup{centertableaux}

\section{Introduction}
\label{sec_intro}

In the early 1990s, Kashiwara introduced the theory of crystal bases \cite{Kash90, Kash91}. This groundbreaking work provides a combinatorial model for the representation theory of semisimple (and more generally, Kac--Moody) Lie algebras, allowing many aspects of the representation theory to be studied from a purely combinatorial point of view. In type $A$, crystal bases can be realized as a collection of combinatorial maps on semistandard Young tableaux, and many previously studied combinatorial tableau algorithms turned out to be special cases of crystal theory. For example, the Robinson--Schensted--Knuth correspondence is the crystal version of the decomposition of the $\GL_n(\bbC)$-representation $(\bbC^n)^{\otimes d}$ into its irreducible components \cite{ShimDummies}; Lascoux and Sch\"utzenberger's symmetric group action on tableaux is a special case of the Weyl group action on any crystal \cite{BumSch}; Sch\"utzenberger's promotion map, restricted to rectangular tableaux, is the crystal-theoretic manifestation of the rotation of the affine type $A$ Dynkin diagram \cite{Shim}.

Tableau algorithms are traditionally described as a sequence of local modifications to a tableau, such as bumping an entry from one row to the next, or sliding an entry into an adjacent box. These combinatorial descriptions are quite beautiful, but for some purposes, one might want a formula that describes the result of the algorithm in terms of a natural set of coordinates on tableaux, such as the number of $j$'s in the $i^{th}$ row, or the closely related Gelfand--Tsetlin patterns. Kirillov and Berenstein discovered that the Bender--Knuth involutions, which are the building blocks for algorithms such as promotion and evacuation, act on a Gelfand--Tsetlin pattern by simple piecewise-linear transformations \cite{KirBer}. This discovery sparked a search for piecewise-linear formulas for other combinatorial algorithms.

The goal of this paper is to find a piecewise-linear formula for the affine type $A$ combinatorial $R$-matrix, a map on pairs of rectangular tableaux which arises from affine crystal theory.

\subsection{Affine crystals and the combinatorial $R$-matrix}
\label{sec_intro_KR}

Quantum affine algebras admit a class of finite-dimensional, non-highest weight representations called Kirillov--Reshetikhin (KR) modules. The crystal bases of these representations, which we call KR crystals, have received a lot of attention for several reasons. Kang et al. showed that the crystal bases of highest-weight modules for quantum affine algebras can be built out of infinite tensor products of KR crystals, and they used this construction to compute the one-point functions of certain solvable lattice models coming from statistical mechanics \cite{KKMMNN1}. KR crystals have also played a central role in the study of a cellular automaton called the box-ball system and its generalizations \cite{TakSat, HHIKTT}.

Unlike the tensor product of representations of Lie algebras and finite groups, the tensor product of representations of quantum algebras---and thus of crystals, which arise as the ``$q \rightarrow 0$ limit'' of such representations---is not commutative. In the case of KR crystals, however, there is a unique crystal isomorphism
\[
\tw{R} : B_1 \otimes B_2 \rightarrow B_2 \otimes B_1.
\]
This isomorphism is the combinatorial $R$-matrix, and it plays an essential role in both of the applications mentioned in the preceding paragraph. For example, the states of the box-ball system can be represented as elements of a tensor product of KR crystals, and the time evolution is given by applying a sequence of combinatorial $R$-matrices.

In (untwisted) affine type $A$, Kirillov--Reshetikhin modules correspond to partitions of rectangular shape $(L^k)$, and their crystal bases, which we denote by $B^{k,L}$, are modeled by semistandard Young tableaux of shape $(L^k)$. If one ignores the affine crystal operators $\tw{e}_0, \tw{f}_0$, then $B^{k,L}$ is the crystal associated to the irreducible $\mathfrak{sl}_n$-module of highest weight $(L^k)$. Shimozono showed that the affine crystal operators are obtained by conjugating the crystal operators $\tw{e}_1, \tw{f}_1$ by Sch\"utzenberger's promotion map \cite{Shim}. He also gave a combinatorial description of the action of the combinatorial $R$-matrix on pairs of rectangular tableaux, which we now explain.

Let $*$ denote the associative product on the set of semistandard Young tableaux introduced by Lascoux and Sch\"utzenberger (see \S \ref{sec_combinatorial_R} for a definition). If $T \in B^{k_1, L_1}$ and $U \in B^{k_2, L_2}$, then there are unique tableaux $U' \in B^{k_2,L_2}$ and $T' \in B^{k_1,L_1}$ such that $T*U = U'*T'$, and the combinatorial $R$-matrix acts as the map $\tw{R} : T \otimes U \mapsto U' \otimes T'$.

\begin{ex}
\label{ex_comb_R}
Let
\[
T = \ytableaushort{1133334} \in B^{1,7} \quad\quad \text{ and } \quad\quad U = \ytableaushort{11123, 22444} \in B^{2,5}.
\]
The product $T * U$ can be computed by using Schensted's row bumping algorithm to insert the entries of $U$ into $T$, starting with the bottom row, and moving from left to right; the result is
\[
T*U = \ytableaushort{1111123444, 22334, 33} \, .
\]
The reader may verify that the tableaux
\[
U' = \ytableaushort{11122, 33334} \quad\quad \text{and} \quad\quad T' = \ytableaushort{1123444}
\]
satisfy $U'*T' = T*U$, so $\tw{R}(T \otimes U) = U' \otimes T'$.
\end{ex}

There is a combinatorial procedure for pulling $T*U$ apart into $U'$ and $T'$, so the whole process is algorithmic. It is nevertheless natural to ask if the map $\tw{R}$ can be computed in one step, without first passing through the product $T * U$. In the case where $T$ and $U$ are both one-row tableaux, there is an elegant piecewise-linear formula for $\tw{R}$.

\begin{prop}[Hatayama et al. {\cite[Prop. 4.1]{HHIKTT}}]
\label{prop_comb_one_row}
Suppose $T$ and $U$ are one-row tableaux, with entries at most $n$, and suppose $\tw{R}(T \otimes U) = U' \otimes T'$. Let $a_j, b_j$ be the numbers of $j$'s in $T$ and $U$, respectively. Define
\[
b'_j = b_j + \tw{\kappa}_{j+1} - \tw{\kappa}_j, \quad\quad a'_j = a_j + \tw{\kappa}_j - \tw{\kappa}_{j+1},
\]
\[
\text{ where } \quad\quad \tw{\kappa}_j = \min_{0 \leq r \leq n-1} (b_j + b_{j+1} + \cdots + b_{j+r-1} + a_{j+r+1} + a_{j+r+2} + \cdots + a_{j+n-1}),
\]
and all subscripts are interpreted modulo $n$. Then $b'_j, a'_j$ are the numbers of $j$'s in $U'$ and $T'$, respectively.
\end{prop}

We will generalize Proposition \ref{prop_comb_one_row} to a piecewise-linear description of the action of the combinatorial $R$-matrix on pairs of arbitrary rectangular tableaux.

\subsection{Geometric lifting}
\label{sec_intro_lifting}

How does one find---and work with---piecewise-linear formulas for complicated combinatorial operations? A very useful method is to use tropicalization and geometric lifting. Tropicalization is the procedure which turns a positive rational function (i.e., a function consisting of the operations $+, \cdot, \div$, but not $-$; such functions are often called ``subtraction-free'' in the literature) into a piecewise-linear function by making the substitutions
\[
(+, \cdot, \div) \mapsto (\min, +, -).
\]
A geometric (or rational) lift of a piecewise-linear function $\tw{h}$ is any positive rational function $h$ which tropicalizes to $\tw{h}$. Rational functions are often easier to work with than piecewise-linear functions, since one may bring to bear algebraic and geometric techniques. Furthermore, identities proved in the lifted setting can be ``pushed down,'' via tropicalization, to results about the piecewise-linear functions and the corresponding combinatorial maps.

For example, the formula for $\tw{R}$ in Proposition \ref{prop_comb_one_row} turns out to be the tropicalization of a rational map which solves a certain matrix equation. Given $x = (x_1, \ldots, x_n) \in (\Cx)^n$, define
\begin{equation}
\label{eq_g_one_row}
g(x) = \left(
\begin{array}{ccccccc}
x_1 & &&&& \lp \\
1 & x_2 & \\
 & 1 & x_3 & \\
 & & & \ddots \\
 & & & & x_{n-1} \\
 & & & & 1 & x_n
\end{array}
\right).
\end{equation}
Here $\lp$ is an indeterminate, and we view $g(x)$ as an element of the loop group $\GL_n(\bbC(\lp))$.

\begin{prop}[Yamada {\cite{Yam}, Lam--Pylyavskyy \cite[Thm. 6.2]{LPwhirl}}]
\label{prop_geom_one_row}
If $x,y \in (\Cx)^n$ are sufficiently generic, then the matrix equation
\begin{equation}
\label{eq_g_eq_one_row}
g(x)g(y) = g(y')g(x')
\end{equation}
has two solutions: the trivial solution $y'_j = x_j, x'_j = y_j$, and the solution
\begin{equation}
\label{eq_geom_one_row}
y'_j = y_j\dfrac{\kappa_{j+1}}{\kappa_j}, \quad x'_j = x_j\dfrac{\kappa_j}{\kappa_{j+1}}, \quad \text{ where } \quad\quad
\kappa_j = \sum_{r=0}^{n-1} y_j \cdots y_{j+r-1} x_{j+r+1} \cdots x_{j+n-1},
\end{equation}
and subscripts are interpreted modulo $n$. The solution given by \eqref{eq_geom_one_row} is the unique solution to \eqref{eq_g_eq_one_row} which satisfies the additional constraint
\begin{equation}
\label{eq_prod_x_y}
\prod x_j = \prod x_j' \quad\quad \text{ and } \quad\quad \prod y_j = \prod y_j'.
\end{equation}
\end{prop}

Note that the piecewise-linear map $\tw{R}$ in Proposition \ref{prop_comb_one_row} is the tropicalization of the rational map $R : (x,y) \mapsto (y',x')$, where $y',x'$ are defined by \eqref{eq_geom_one_row}\footnote{In the tropicalization, we replace the ``rational variables'' $x_j$ and $y_j$, which can be thought of as generic nonzero complex numbers, or indeterminates, with the ``combinatorial variables'' $a_j$ and $b_j$, which take on integer values.} (note also that \eqref{eq_prod_x_y} tropicalizes to the condition $\sum a_j = \sum a'_j, \sum b_j = \sum b_j'$, which says that the tableaux $T$ and $T'$ (resp., $U$ and $U'$) have the same length). Thus, the map $R$ is a geometric lift of the combinatorial $R$-matrix on pairs of one-row tableaux.

It is striking that the solution to a matrix equation also describes a combinatorial procedure for swapping pairs of tableaux. In fact, this example is just one instance of a larger phenomenon. Since the appearance of Kirillov--Berenstein's work on the Bender--Knuth involutions, a number of other combinatorial algorithms have been lifted to rational maps, including the Robinson--Schensted--Knuth correspondence \cite{Kir, NouYam, DanKosh} and rowmotion on posets \cite{EinPropp}.

One of the crowning achievements of the geometric lifting program is Berenstein and Kazhdan's theory of geometric crystals, which provides a framework for lifting the entire combinatorial structure of crystal bases \cite{BKI, BKII, BKshort}. Roughly speaking, a geometric crystal is a complex algebraic variety $X$, together with rational actions $e_i : \Cx \times X \rightarrow X$, which are called geometric crystal operators. The geometric crystal operators are required to satisfy rational lifts of the piecewise-linear relations satisfied by (combinatorial) crystal operators. In many cases, the geometric crystal operators are positive, and they tropicalize to piecewise-linear formulas for the combinatorial crystal operators $\tw{e}_i$ on a corresponding combinatorial crystal $B_X$; when this happens, we say that $X$ tropicalizes to $B_X$. For each reductive group $G$, Berenstein and Kazhdan \cite{BKII} constructed a geometric crystal on the flag variety\footnote{The geometric crystal is actually constructed on $G/B \times T$, where $T$ is a maximal torus.} of $G$ which lifts the crystals associated to all the irreducible representations of $G^\vee$, the Langlands dual group. These geometric crystals provide a new method for constructing and studying crystals; in addition, they have proved useful beyond combinatorics, with applications to quantum cohomology and mirror symmetry, Brownian motion on Lie groups, and the local Langlands conjectures \cite{LamTemp, Chh, BravKazh, BKshort}.

Nakashima \cite{Nak} extended the definition of geometric crystals to the setting of Kac--Moody (and in particular, affine) Lie algebras. There has been a concerted effort to construct geometric lifts of Kirillov--Reshetikhin crystals, and to find compatible lifts of the associated combinatorial $R$-matrices. In the case of the one-row affine type $A$ crystals mentioned above, it is straightforward to define a corresponding geometric crystal, and the rational map from Proposition \ref{prop_geom_one_row} turns out to be an isomorphism of geometric crystals (see the Introduction of \cite{KOTY}). Kuniba--Okado--Takagi--Yamada and Kashiwara--Nakashima--Okado have constructed a geometric crystal and compatible geometric $R$-matrix for the analogue of one-row KR crystals in types $D_n^{(1)}, B_n^{(1)}, D_{n+1}^{(2)}, A_{2n-1}^{(2)}, A_{2n}^{(2)}$ \cite{KOTY, KNO2}.

Since geometric crystals are the natural setting in which to lift the combinatorial $R$-matrix, the first step in our project was to construct a geometric crystal which tropicalizes to the disjoint union of the KR crystals $B^{k,L}, L \geq 0$. This was carried out in \cite{F1}. The main ideas used in the construction of this geometric crystal are also central to the construction of the geometric $R$-matrix; we give an overview of these ideas in \S \ref{sec_intro_cyclic} and \S \ref{sec_intro_unip}.

\subsection{Cyclic symmetry and the Grassmannian}
\label{sec_intro_cyclic}

We saw above that the geometric $R$-matrix in the one-row case is the solution to a matrix equation. The same is true in the general case, and in fact, the full geometric crystal structure is determined from the appropriate generalization of the matrix $g(x)$ in \eqref{eq_g_one_row}. Before describing this matrix, we introduce coordinates on semistandard rectangular tableaux with $k$ rows (and entries at most $n$). The entries in the $i^{th}$ row of such a tableau must lie in the interval $\{i, i+1, \ldots, i+n-k\}$. If we fix the row length $L$, then the $i^{th}$ row is determined by the $n-k$ integers $b_{ii}, b_{i,i+1}, \ldots, b_{i,i+n-k-1}$, where $b_{ij}$ is the number of $j$'s in the $i^{th}$ row. Thus, a $k$-row rectangular tableau is determined by $k(n-k)$ integers $b_{ij}$, plus the row length $L$. (These integers must satisfy certain inequalities, such as non-negativity, but we ignore the inequalities in this discussion; see \S \ref{sec_k_rect} for full details.)

To lift the combinatorial $R$-matrix in the one-row case, we replaced the integer coordinates $b_1, \ldots, b_n$ with rational coordinates $x_1, \ldots, x_n$. In the $k$-row case, we replace $b_{ij}$ with $x_{ij}$, and the row length $L$ with the rational coordinate $t$. It turns out that the coordinates $(x_{ij},t)$ are not well-suited to defining the generalization of the matrix \eqref{eq_g_one_row}. In the $k = 1$ case, the coordinates $x_1, \ldots, x_n$ have a simple cyclic symmetry of order $n$, which is reflected in the matrix \eqref{eq_g_one_row}.\footnote{To see this symmetry in the matrix, one must ``unfold'' $g(x)$ into an infinite periodic matrix which repeats the sequence $x_1, \ldots, x_n$ along the main diagonal, and has an infinite diagonal of 1's just below the main diagonal. See \S \ref{sec_unfold} for the precise definition of ``unfolding.''} For $k > 1$, the coordinates $(x_{ij}, t)$ do not have any obvious cyclic symmetry. There is, however, a ``hidden'' cyclic symmetry coming from Sch\"utzenberger promotion, which has order $n$ on rectangular tableaux with entries at most $n$. The key to defining the analogue of \eqref{eq_g_one_row} is to use an alternative set of coordinates which makes the action of promotion transparent. This alternative set of coordinates comes from the Grassmannian.

Let $\Gr(n-k,n)$ denote the Grassmannian of $(n-k)$-dimensional subspaces in $\bbC^n$. In \cite{F1}, borrowing a construction from the work of Lusztig and Berenstein--Fomin--Zelevinsky on total positivity \cite{Lusz, BFZ}, we defined a birational isomorphism from the $k(n-k)$ rational coordinates $x_{ij}$ to a subspace $N \in \Gr(n-k,n)$. Write $\X{n-k}$ for the variety $\Gr(n-k,n) \times \Cx$, and let $\Theta_{n-k} : \bbC^{k(n-k)+1} \rightarrow \X{n-k}$ denote the birational map given by
\[
(x_{ij}, t) \mapsto (N,t) =: N|t
\]
(see \S \ref{sec_GT_param} for details). The Grassmannian has a natural cyclic symmetry induced by rotating a basis of the underlying $n$-dimensional vector space. By ``twisting'' this symmetry by the parameter $t$, we defined a map $\PR : N|t \mapsto N'|t$, and we showed that the composition $\Theta_{n-k}^{-1} \circ \PR \circ \, \Theta_{n-k}$ tropicalizes to a piecewise-linear formula for promotion on $k$-row rectangular tableaux. Since $\Theta_{n-k}$ is a birational isomorphism (and there is a simple formula for its inverse), one may do computations in terms of Pl\"{u}cker coordinates on the Grassmannian, and then translate back to the coordinates $(x_{ij}, t)$ at the end.

The analogue of \eqref{eq_g_one_row} in the general $k$-row case is a matrix in $\GL_n(\bbC(\lp))$ filled with ratios of Pl\"{u}cker coordinates of an $(n-k)$-dimensional subspace. When $n = 5$ and $k = 3$, the matrix looks like this:
\begin{equation}
\label{eq_ex_2_5}
g(N|t) = \left(
\begin{array}{ccccc}
\dfrac{P_{15}}{P_{45}} & 0 & \lp & \lp \dfrac{P_{13}}{P_{23}} & \lp \dfrac{P_{14}}{P_{34}} \smallskip \\
\dfrac{P_{25}}{P_{45}} & \dfrac{P_{12}}{P_{15}} & 0 & \lp & \lp \dfrac{P_{24}}{P_{34}} \smallskip \\
\dfrac{P_{35}}{P_{45}} & \dfrac{P_{13}}{P_{15}} & t \dfrac{P_{23}}{P_{12}} & 0 & \lp \smallskip \\
1 & \dfrac{P_{14}}{P_{15}} & t \dfrac{P_{24}}{P_{12}} & t \dfrac{P_{34}}{P_{23}} & 0 \smallskip \\
0 & 1 & t \dfrac{P_{25}}{P_{12}} & t \dfrac{P_{35}}{P_{23}} & t \dfrac{P_{45}}{P_{34}}
\end{array}
\right).
\end{equation}
Here $P_I$ is the $I^{th}$ Pl\"{u}cker coordinate of the two-dimensional subspace $N$. See Definition \ref{defn_g} for the general definition of $g(N|t)$; note that the one-row case corresponds to $\X{n-1,n}$.

Suppose $M|s \in \X{k_1}$ and $N|t \in \X{k_2}$. As in the one-row case, we seek a solution to the matrix equation
\begin{equation}
\label{eq_g_k_row}
g(M|s)g(N|t) = g(N'|t)g(M'|s),
\end{equation}
where $N' \in \Gr(k_2,n)$ and $M' \in \Gr(k_1,n)$. Using properties of the Grassmannian and linear algebra, we show that for sufficiently generic $M, N, s, t$, there is a unique candidate for the solution to \eqref{eq_g_k_row} (Lemma \ref{lem_recover_N_1}, Corollary \ref{cor_N_i_unique}). We define the geometric $R$-matrix to be the map
\[
R : (M|s, N|t) \mapsto (N'|t, M'|s)
\]
given by this unique candidate. The main technical results of this paper are
\begin{itemize}
\item Theorem \ref{thm_hard}, which states that $R$ does in fact give a solution to \eqref{eq_g_k_row};
\item Theorem \ref{thm_R_posit}, which states that $R$ is positive, in the sense that the map
\begin{equation*}
(\Theta_{k_2}^{-1} \times \Theta_{k_1}^{-1}) \circ R \circ (\Theta_{k_1} \times \Theta_{k_2}) : ((x_{ij}, s), (y_{ij}, t)) \mapsto ((y'_{ij}, t), (x'_{ij}, s))
\end{equation*}
is given by positive rational functions in $x_{ij}, y_{ij}, s,$ and $t$.
\end{itemize}
The latter result shows that the geometric $R$-matrix can be tropicalized, and the former result is the key to showing that $R$ commutes with the geometric crystal operators.

\subsection{Unipotent crystals and the loop group}
\label{sec_intro_unip}

Why should the geometric $R$-matrix satisfy a matrix equation? One explanation comes from the notion of unipotent crystals (for an alternative explanation, based on the combinatorial description of $\tw{R}$, see Remark \ref{rmk_tab_mat_tab_prod}). Let $G$ be a reductive group, $B^-$ a fixed Borel subgroup, and $U$ the unipotent radical of the opposite Borel. In the case $G = \GL_n(\bbC)$, one can take $B^-$ to be the lower triangular matrices and $U$ the upper uni-triangular matrices. Berenstein and Kazhdan \cite{BKI} gave $B^-$ a geometric crystal structure in which the geometric crystal operator $e_i$ is given by simultaneous left and right multiplication by certain elements of the one parameter subgroup in $U$ corresponding to the $i^{th}$ simple root. They defined a unipotent crystal to be a pair $(X,g)$, where $X$ is a variety which carries a rational action of $U$, and $g : X \rightarrow B^-$ is a rational map which is ``compatible'' with the $U$-action (see \S \ref{sec_unip} for details). A unipotent crystal $(X,g)$ induces a geometric crystal on $X$, in such a way that $g$ intertwines the geometric crystal operators on $X$ and $B^-$ (i.e., $g e_i = e_i g$). Furthermore, if $(X,g)$ and $(Y,g)$ are unipotent crystals, then $(X \times Y, g)$ is a unipotent crystal, where
\begin{equation}
\label{eq_g_prod_defn}
g(x,y) := g(x)g(y).
\end{equation}
This unipotent crystal induces a geometric crystal on the product $X \times Y$, and if $X$ and $Y$ tropicalize to crystals $B_X, B_Y$, then $X \times Y$ tropicalizes to the tensor product $B_X \otimes B_Y$.

In our affine type $A$ setting, the appropriate analogue of the reductive group $G$ is the loop group $\GL_n(\bbC(\lp))$, which consists of invertible $n \times n$ matrices over the field of rational functions in an indeterminate $\lp$. We take $B^-$ to be a certain ``lower triangular'' submonoid of the loop group, and $U$ an ``upper uni-triangular'' subgroup (this triangularity refers to the ``unfolded'' version of the matrices; see \S \ref{sec_defn_unip}). Berenstein and Kazhdan's theory of unipotent crystals extends essentially unchanged to this setting.

We showed in \cite{F1} that the map $g : \X{k} \rightarrow B^-$ discussed above makes $\X{k}$ into an affine type $A$ unipotent crystal. This explains why the geometric $R$-matrix ought to provide a solution to the matrix equation \eqref{eq_g_k_row}. Indeed, the geometric $R$-matrix is supposed to be a map $R : \X{k_1} \times \X{k_2} \rightarrow \X{k_2} \times \X{k_1}$ which commutes with the geometric crystal operators. Equation \eqref{eq_g_k_row} says that $g \circ R = g$; if this is satisfied, then since $g$ commutes with the geometric crystal operators, we have
\begin{equation}
\label{eq_g_e_R}
g \circ e_iR = g \circ Re_i.
\end{equation}
By the uniqueness of the solution to \eqref{eq_g_k_row}, \eqref{eq_g_e_R} implies that $R$ commutes with $e_i$.

We remark that our approach to the geometric $R$-matrix is similar in some important respects to that of \cite{KOTY, KNO2}. In particular, the matrix $g(N|t)$ is the analogue of the ``matrix realization'' (or ``$M$-matrix'') in those works, and our use of the uniqueness of the solution to \eqref{eq_g_k_row} to prove properties of the geometric $R$-matrix closely resembles several proofs in \cite[\S 4]{KOTY}.

\subsection{Applications and future directions}
\label{sec_intro_apps}

The deepest property of the combinatorial $R$-matrix is that it satisfies the Yang--Baxter (or braid) relation. In \S \ref{sec_R_props}, we show that the Yang--Baxter relation for the geometric $R$-matrix follows easily from the fact that the geometric $R$-matrix gives the unique solution to the matrix equation \eqref{eq_g_k_row}. By tropicalizing, we obtain a new proof of the Yang--Baxter relation for the combinatorial $R$-matrix.

The one-row geometric $R$-matrix of Proposition \ref{prop_geom_one_row} has appeared in a number of different contexts \cite{BravKazh, Et, Yam, LPwhirl, LPnet}. Because it is an involutive solution to the Yang--Baxter equation, it induces a birational action of the symmetric group $S_m$ on the field of rational functions in $mn$ variables. Lam and Pylyavskyy called the polynomial invariants of this action loop symmetric functions, and they showed that these invariants have many properties analogous to those of symmetric functions \cite{LPwhirl, LamICCM}. We expect that the more general geometric $R$-matrix constructed here will have applications to loop symmetric functions.

In fact, our original motivation for lifting the combinatorial $R$-matrix comes from a conjectural connection between loop symmetric functions and the above-mentioned box-ball system. Lam--Pylyavskyy--Sakamoto \cite{LPS} conjectured that the tropicalization of certain loop symmetric functions gives a formula for the scattering of an initial configuration of balls into solitons. They were able to prove the first case of their conjecture using the one-row geometric $R$-matrix. To extend their method to prove the full conjecture, one needs a lift of the combinatorial $R$-matrix in the case where one of the tableaux has more than one row. We are optimistic that our general geometric $R$-matrix can be used to prove the conjecture in full generality.

Another interesting feature of affine crystal theory is the (co)energy function on tensor products of Kirillov--Reshetikhin crystals. Lam and Pylyavskyy showed that a certain ``stretched staircase'' loop Schur function tropicalizes to the coenergy function on tensor products of arbitrarily many one-row tableaux \cite{LPenergy}. As an application of our setup, we show in \S \ref{sec_coenergy} that a minor of the matrix $g(M|s)g(N|t)$ tropicalizes to the coenergy function on tensor products of two arbitrary rectangular tableaux. It would be interesting to find a lift of the coenergy function on tensor products of more than two rectangular tableaux which simultaneously generalizes this minor and the ``stretched staircase'' loop Schur function.

We hope that our methods can be extended to lift Kirillov--Reshetikhin crystals and their combinatorial $R$-matrices in other affine types, beyond the analogue of the one-row case. One potential difficulty is that most KR crystals outside of type $A_{n-1}^{(1)}$ are reducible as classical crystals. We speculate that this will make it necessary to use ``isotropic partial flag varieties,'' rather than just ``isotropic Grassmannians,'' in the other types.

\subsection{Organization}
\label{sec_intro_org}

In \S \ref{sec_comb_R}, we provide background on the tensor product of crystals, the combinatorial $R$-matrix, the coenergy function, and Gelfand--Tsetlin patterns. In \S \ref{sec_geom_unip}, we review the definitions of geometric and unipotent crystals and their products, and we recall from \cite{F1} the geometric and unipotent crystal structure on the variety $\X{k}$. We also recall three symmetries of $\X{k}$ ($\PR$, $S$, and $D$) which play a crucial role in this paper. In \S \ref{sec_trop}, we review the parametrization of $\X{k}$ by the rational coordinates $(x_{ij}, t)$ mentioned above. We define notions of positive varieties and positive rational maps (following \cite{BKII}), and we show that for the variety $\X{k}$, the general definition of positivity is equivalent to a more intuitive notion of ``Pl\"{u}cker positivity.'' We give a formal definition of tropicalization, and we state the precise sense in which products of the geometric crystals $\X{k}$ tropicalize to products of the combinatorial crystals $\bigsqcup_{L \geq 0} B^{n-k,L}$ (Theorem \ref{thm_recover_crystals}).

In \S \ref{sec_geom_R}, we define the geometric $R$-matrix. We state Theorem \ref{thm_hard}, which says that $R$ gives a solution to the matrix equation \eqref{eq_g_k_row}, and Theorem \ref{thm_R_posit}, which says that $R$ is positive. We use these theorems to prove that $R$ is an isomorphism of geometric crystals, that $R$ commutes with the symmetries $\PR$, $S$, and $D$, that $R$ satisfies the Yang--Baxter relation, and that $R$ tropicalizes to $\tw{R}$ (Theorems \ref{thm_master}, \ref{thm_recover_comb_R}). In \S \ref{sec_coenergy}, we define a notion of geometric coenergy function, and we exhibit such a function on $\X{k_1} \times \X{k_2}$. In \S \ref{sec_one}, we consider the case where the first tableau has one row. We obtain an explicit formula for the geometric $R$-matrix in this case, and we show that the formula reduces to that of Proposition \ref{prop_geom_one_row} when the second tableau also has one row. We also verify our results in a small example. We prove Theorems \ref{thm_R_posit} and \ref{thm_hard} in \S 8 and \S 9, respectively.

We include an Appendix which reviews the representation of matrices by planar networks and the Lindstr\"{o}m Lemma.

\subsection{Notation}
\label{sec_intro_notation}

Throughout this paper, we fix an integer $n \geq 2$. For integers $i$ and $j$, we write
\[
[i,j] = \{m \in \mathbb{Z} \, | \, i \leq m \leq j\}.
\]
We often abbreviate $[1,j]$ to $[j]$. We write ${[n] \choose k}$ for the set of $k$-element subsets (or {\em $k$-subsets}) of $[n]$, and $|J|$ for the cardinality of a set $J$.

Given a matrix $X$ and two subsets $I,J$, we write $X_{I,J}$ to denote the submatrix using the rows in $I$ and the columns in $J$. If $|I| = |J|$, we write
\[
\Delta_{I,J}(X) = \det(X_{I,J}).
\]

We use the term \emph{upper} (resp., \emph{lower}) \emph{uni-triangular} to refer to matrices with zeroes below (resp., above) the main diagonal, and 1's on the main diagonal. We write $\Cx$ for the multiplicative group of nonzero complex numbers. Almost all the maps between algebraic varieties appearing in this paper are rational, so we write them with solid arrows (e.g., $h : X \rightarrow Y$), rather than dotted arrows.

\subsection*{Acknowledgments}

I owe a great deal of thanks to my advisor, Thomas Lam, for introducing me to total positivity, the combinatorial $R$-matrix, and the idea of geometric lifting, and for many helpful conversations and insightful suggestions, without which this project would not have been possible. My thanks also to Thomas Lam and Sergey Fomin for helpful comments on an earlier version of this paper.

\section{Combinatorial background}
\label{sec_comb_R}

\subsection{Affine crystal structure on rectangular tableaux}
\label{sec_affine_cryst}

A {\em type $A_{n-1}^{(1)}$ crystal} consists of an underlying set $B$, together with
\begin{itemize}
\item a \emph{weight map} $\tw{\gamma} : B \rightarrow (\bbZ_{\geq 0})^n$;
\item for each $i \in \Zn$, functions $\tep_i, \tph_i : B \rightarrow \bbZ_{\geq 0}$;
\item for each $i \in \Zn$, a \emph{crystal operator} $\te_i : B \rightarrow B \sqcup \{0\}$.
\end{itemize}
We say that $\te_i(b)$ is \emph{defined} (resp., \emph{undefined}) if $\te_i(b) \in B$ (resp., $\te_i(b) = 0$). The maps $\tg, \tep_i, \tph_i, \te_i$ must satisfy various axioms (see, e.g., \cite{BumSch}), but we will not use the axioms in this paper. An {\em isomorphism} of crystals $A$ and $B$ is a bijection $\tw{\phi} : A \rightarrow B$ such that $\tw{\rho} \circ \tw{\phi} = \tw{\rho}$ for $\tw{\rho} = \tg, \tep_i, \tph_i$, and $\te_i \circ \tw{\phi} = \tw{\phi} \circ \te_i$, where $\tw{\phi}(0) := 0$.

Let $\lambda$ be a partition with at most $n$ parts. A {\em semistandard Young tableau (SSYT) of shape $\lambda$} is a filling of the Young diagram of $\lambda$ with entries in $[n]$, such that the rows are weakly increasing, and the columns are strictly increasing. We will often refer to these objects simply as \emph{tableaux}. We write $B(\lp)$ to denote the set of SSYTs of shape $\lp$.

For $k \in [n-1]$ and $L \geq 0$, define $B^{k,L} := B(L^k)$, the set of SSYTs (with entries in $[n]$) whose shape is the $k \times L$ rectangle. By convention, $B^{k,0}$ consists of a single ``empty tableau.'' Each $B^{k,L}$ admits a type $A_{n-1}^{(1)}$ crystal structure, corresponding to the crystal basis of a {\em Kirillov--Reshetikhin module}, a finite-dimensional representation of $U_q'(\wh{\mathfrak{sl}}_n)$. We refer to the crystals $B^{k,L}$ as {\em Kirillov--Reshetikhin (KR) crystals}. The map $\tw{\gamma}$ assigns to each tableau its {\em weight} (or {\em content}), i.e., the vector $(a_1, \ldots, a_n)$, where $a_i$ is the number of $i$'s in the tableau. Combinatorial descriptions of the maps $\tep_i, \tph_i, \te_i$ on rectangular tableaux can be found in \cite{Shim,F1}.

\subsection{The combinatorial $R$-matrix}
\label{sec_combinatorial_R}

\begin{defn}
Given two type $A_{n-1}^{(1)}$ crystals $A,B$, their tensor product $A \otimes B$ is defined as follows. The underlying set is the Cartesian product $A \times B$, whose elements we denote by $a \otimes b$. The crystal structure is defined by\footnote{We use the tensor product convention that is compatible with tableau combinatorics and with the product of geometric crystals, as in \cite{Shim, BKII}; Kashiwara's original convention interchanges the roles of $a$ and $b$.}
\begin{itemize}
\item $\tg(a \otimes b) = \tg(a) + \tg(b)$;
\item $\tep_i(a \otimes b) = \tep_i(b) + \max(0, \tep_i(a) - \tph_i(b))$;
\item $\tph_i(a \otimes b) = \tph_i(a) + \max(0, \tph_i(b) - \tep_i(a))$;
\item $\te_i(a \otimes b) = \begin{cases}
\te_i(a) \otimes b & \text{ if } \tep_i(a) > \tph_i(b) \\
a \otimes \te_i(b) & \text{ if } \tep_i(a) \leq \tph_i(b).
\end{cases}
$
\end{itemize}
In the definition of $\te_i$, we take $0 \otimes b = a \otimes 0 = 0$.
\end{defn}

The tensor product of crystals is associative, but not commutative. In the case of the Kirillov--Reshetikhin crystals $B^{k,L}$, however, there is a unique crystal isomorphism $\tw{R} : B^{k_1, L_1} \otimes B^{k_2, L_2} \rightarrow B^{k_2,L_2} \otimes B^{k_1,L_1}$, called the {\em combinatorial $R$-matrix}. The existence and uniqueness of this isomorphism is proved using quantum groups (see \cite[Thm 3.19]{Shim}). We now describe how this map acts on tableaux, following Shimozono \cite{Shim}.

Lascoux and Sch\"utzenberger introduced an associative {\em tableau product} on the set of semistandard tableaux \cite{LasSch}. Given two tableaux $T$ and $U$, the product $T * U$ may be defined as the rectification of the skew-tableau obtained by placing $U$ to the northeast of $T$, as shown here:
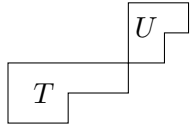
\begin{center}
\begin{tikzpicture}[scale=0.8]
\draw (2,-1) -- (2,-1.5) -- (1,-1.5) -- (1,-2) -- (0,-2) -- (0,-1) -- (2,-1);
\draw (2,-1) -- (2,0) -- (3,0) -- (3,-.5) -- (2.6,-.5) -- (2.6,-1) -- (2,-1);
\draw (0.6,-1.5) node {$T$};
\draw (2.3,-.4) node {$U$};
\end{tikzpicture}
\end{center}
The rectification can be computed using Sch\"{u}tzenberger's jeu-de-taquin slides or Schensted's row insertion (we refer the reader to \cite{Ful} for details).

\begin{prop}
\label{prop_comb_R}
Suppose $(T,U) \in B^{k_1,L_1} \otimes B^{k_2,L_2}$.
\begin{enumerate}
\item There is a unique pair $(U',T') \in B^{k_2,L_2} \times B^{k_1,L_1}$ such that $T*U = U'*T'$.
\item The combinatorial $R$-matrix is given by $\tw{R}(T,U) = (U',T')$.
\end{enumerate}
\end{prop}

\begin{sproof}
The Littlewood-Richardson coefficient $c_{\lambda \mu}^{\nu}$ is equal to the number of pairs $(T,U) \in B(\lambda) \times B(\mu)$ such that $T*U = V$, where $V$ is a fixed element of $B(\nu)$ (see \cite[\S 5.1, Cor. 2]{Ful}). If $\lambda$ and $\mu$ are rectangles, then the product of Schur functions $s_\lambda s_\mu$ is multiplicity-free (see \cite{StemSchur}). Thus, there is exactly one pair $(U',T') \in B^{k_2,L_2} \times B^{k_1,L_1}$ such that $U' * T' = T * U$. This proves (1).

By \cite[Lem. 3.8]{Shim}, if $\lp$ and $\mu$ are arbitrary partitions and $\tw{h} : B(\lp) \times B(\mu) \rightarrow B(\mu) \times B(\lp)$ is a bijection which commutes with the classical crystal operators $\tw{e}_1, \ldots, \tw{e}_{n-1}$, then $\tw{h}(T,U) = (U',T')$ implies $T * U = U' * T'$. Thus, (2) follows from (1) and the existence of $\tw{R}$.
\end{sproof}

A sample computation of the map $\tw{R}$ appeared in Example \ref{ex_comb_R}.

\begin{prop}
\label{prop_comb_YB}
Let $A,B,C$ be three Kirillov--Reshetikhin crystals of type $A_{n-1}^{(1)}$.
\begin{enumerate}
\item The map $\tw{R}^2 : A \otimes B \rightarrow A \otimes B$ is the identity.
\item The combinatorial $R$-matrix satisfies the Yang--Baxter relation. That is, if $\tw{R}_1 : A \otimes B \otimes C \rightarrow B \otimes A \otimes C$ is the map which applies $\tw{R}$ to the first two factors and does nothing to the third factor, and $\tw{R}_2 : A \otimes B \otimes C \rightarrow A \otimes C \otimes B$ is the map which applies $\tw{R}$ to the last two factors and does nothing to the first factor, then
\[
\tw{R}_1 \circ \tw{R}_2 \circ \tw{R}_1 = \tw{R}_2 \circ \tw{R}_1 \circ \tw{R}_2
\]
as maps from $A \otimes B \otimes C \rightarrow C \otimes B \otimes A$.
\end{enumerate}
\end{prop}

The first statement follows immediately from the description of $\tw{R}$ in Proposition \ref{prop_comb_R}. There are several proofs of the Yang--Baxter relation. For instance, this relation is a consequence of Akasaka--Kashiwara's result that every tensor product $B^{k_1,L_1} \otimes \cdots \otimes B^{k_d,L_d}$ is connected (as an affine crystal), which is proved by analyzing the poles of the quantum $R$-matrix \cite{AkaKash}. Shimozono gave a combinatorial proof of the Yang--Baxter relation using a generalization of Lascoux and Sch\"utzenberger's cyclage poset \cite[Thm. 8(A3)]{ShimYB}. In \S \ref{sec_trop_R}, we give a new proof using the geometric $R$-matrix.

\subsection{The coenergy function}
\label{sec_comb_coenergy}

Another important element of affine crystal theory is the coenergy function.

\begin{defn}
\label{defn_comb_coenergy}
Let $A$ and $B$ be Kirillov--Reshetikhin crystals. A function $\tw{H} : A \otimes B \rightarrow \mathbb{Z}$ is a {\em coenergy function} if $\tw{H} \circ \tw{e}_i = \tw{e}_i$ for $i = 1, \ldots, n-1$, and $\tw{H}$ interacts with $\tw{e}_0$ as follows: if $a \otimes b \in A \otimes B$ and $\tw{R}(a \otimes b) = b' \otimes a'$, then
\begin{equation}
\label{eq_coenergy}
\tw{H}(\tw{e}_0(a \otimes b)) = \tw{H}(a \otimes b) + \begin{cases}
1 & \text{ if } \tep_0(a) > \tph_0(b) \text{ and } \tep_0(b') > \tph_0(a') \\
-1 & \text{ if } \tep_0(a) \leq \tph_0(b) \text{ and } \tep_0(b') \leq \tph_0(a')\\
0 & \text{ otherwise.}
\end{cases}
\end{equation}
\end{defn}

\begin{remark}
A function $\tw{H}$ is a coenergy function if and only if $-\tw{H}$ is an {\em energy function}, in the sense of \cite{KKMMNN1,Shim}. We have chosen to work with coenergy instead of energy because the coenergy function $\tw{E}$ defined below naturally arises as the tropicalization of a certain rational function on our geometric crystals.
\end{remark}

Given $T \in B^{k_1,L_1}$ and $U \in B^{k_2,L_2}$, define $\tw{E}(T \otimes U)$ to be the number of boxes in the tableau $T*U$ which are not in the first $\max(k_1, k_2)$ rows. It is clear from the nature of Schensted insertion that if $T_0$ is the classical highest weight element of $B^{k_1,L_1}$ (that is, the tableau whose $i^{th}$ row is filled with the number $i$), then
\begin{equation}
\label{eq_E_highest_wt}
\tw{E}(T_0 \otimes U) = 0 \text{ for all } U \in B^{k_2,L_2}.
\end{equation}

\begin{ex}
\label{ex_comb_coenergy}
Let $T$ and $U$ be the tableaux in Example \ref{ex_comb_R}. There are two boxes outside the first $\max(1,2)$ rows of $T * U$, so $\tw{E}(T \otimes U) = 2$.
\end{ex}

\begin{prop}
\label{prop_coenergy}
\
\begin{enumerate}
\item Up to a global additive constant, there is a unique (co)energy function on $B^{k_1,L_1} \otimes B^{k_2,L_2}$.
\item $\tw{E}$ is a coenergy function on $B^{k_1,L_1} \otimes B^{k_2,L_2}$.
\end{enumerate}
\end{prop}

\begin{sproof}
For part (1), see \cite[\S 4]{KKMMNN1} and \cite[\S 3.6]{Shim}. For (2), define $\tw{F}(T \otimes U)$ to be the number of boxes in $T*U$ that are not in the first $\max(L_1, L_2)$ columns. By \cite[Prop. 4.5 and (2.4)]{Shim}, $\tw{F}$ is an energy function. It is straightforward to show, using basic properties of jeu-de-taquin and Schensted insertion, that
\[
\tw{E}(T \otimes U) + \tw{F}(T \otimes U) = \min(k_1,k_2)\min(L_1,L_2),
\]
so $\tw{E}$ is a coenergy function.
\end{sproof}

\subsection{Gelfand--Tsetlin patterns and $k$-rectangles}
\label{sec_k_rect}

A {\em Gelfand--Tsetlin (GT) pattern} is a triangular array of nonnegative integers $(A_{ij})_{1 \leq i \leq j \leq n}$ satisfying the inequalities
\begin{equation}
\label{eq_GT_def}
A_{i,j+1} \geq A_{ij} \geq A_{i+1,j+1}
\end{equation}
for $1 \leq i \leq j \leq n-1$. There is a simple bijection between Gelfand--Tsetlin patterns and semistandard tableaux. Given a Gelfand--Tsetlin pattern $(A_{ij})$, the associated tableau $T$ is described as follows: the number of $j$'s in the $i^{th}$ row of $T$ is $A_{ij} - A_{i,j-1}$ (we set $A_{i,i-1} = 0$). Equivalently, the $j^{th}$ row $(A_{1j}, \ldots, A_{jj})$ of the pattern is the shape of $T_{\leq j}$, the part of $T$ obtained by removing numbers larger than $j$. Here is an example of the bijection, with $n=5$:

\begin{equation}
\label{eq_GT_ex}
\begin{array}{ccccccccccc}
&&&& 2 & \\
&&& 4 && 2\\
&&6 && 3 && 1 \\
&6 && 6 && 1 && 0 \\
6 && 6 && 6 && 0 && 0
\end{array}
\quad\quad
\longleftrightarrow \quad\quad
\ytableaushort{{1}{1}{2}{2}{3}{3}, {2}{2}{3}{4}{4}{4}, {3}{5}{5}{5}{5}{5}} \; .
\end{equation}

For $k \in [n-1]$, set
\begin{equation}
\label{eq_R_k}
R_k = \{(i,j) \: | \: 1 \leq i \leq k, \: i \leq j \leq i+n-k-1\},
\end{equation}
and define $\tT{k} = \mathbb{Z}^{R_k} \times \mathbb{Z} \cong \mathbb{Z}^{k(n-k)+1}$. We denote a point of $\tT{k}$ by $(B_{ij},L)$, where $(i,j)$ runs over $R_k$. Given $(B_{ij}, L) \in \tT{k}$, define a triangular array $(A_{ij})_{1 \leq i \leq j \leq n}$ by
\[
A_{ij} = \begin{cases}
B_{ij} & \text{ if } (i,j) \in R_k \\
L & \text{ if } j > i+n-k-1 \\
0 & \text{ if } j < i.
\end{cases}
\]

\begin{defn}
\label{defn_k_rect}
Define $B^k$ to be the set of $(B_{ij},L) \in \tT{k}$ such that $(A_{ij})$ is a Gelfand--Tsetlin pattern. We call an element of $B^k$ a {\em $k$-rectangle}, and we say that $(A_{ij})$ is the {\em associated Gelfand--Tsetlin pattern}.
\end{defn}

It is easy to see that the bijection between GT patterns and SSYTs restricts to a bijection between $k$-rectangles and rectangular tableaux with $k$ rows, with the coordinate $L$ giving the number of columns in the tableau. Thus, we identify
\[
B^k = \bigsqcup_{L=0}^\infty B^{k,L}.
\]
We view the integers $(B_{ij},L)$ as coordinates on the set of $k$-row rectangular SSYTs. Sometimes it will be more convenient to work with an alternative set of coordinates. For $1 \leq i \leq k$ and $i \leq j \leq i+n-k$, define
\begin{equation}
\label{eq_row_coords}
b_{ij} = B_{ij} - B_{i,j-1},
\end{equation}
where we use the convention that $B_{i,i-1} = 0$ and $B_{i,i+n-k} = L$ for all $i$. Note that if $(B_{ij}, L) \in B^k$, then $b_{ij}$ is the number of $j$'s in the $i^{th}$ row of the corresponding rectangular tableau; hence, these coordinates satisfy $\sum_j b_{ij} = L$ for $i = 1, \ldots, k$.

\section{Geometric and unipotent crystals}
\label{sec_geom_unip}

\subsection{Geometric crystals}
\label{sec_geom}

We recall the notions of type $A_{n-1}^{(1)}$ decorated geometric crystals and their product, following \cite{BKII, Nak}.

\begin{defn}
\label{defn_geom_precryst}
A {\em geometric pre-crystal (of type $A_{n-1}^{(1)}$)} consists of an irreducible complex algebraic (ind-)variety $X$, together with
\begin{itemize}
\item a rational map $\gamma : X \rightarrow (\Cx)^n$
\item for each $i \in \Zn$, rational functions $\vp_i, \ve_i : X \rightarrow \Cx$ which are not identically zero.\footnote{In \cite{BKII}, some of the $\vp_i$ and $\ve_i$ are allowed to be zero, but we will not need this more general setting.}
\item for each $i \in \Zn$, a rational unital\footnote{This means that $e_i(1,x)$ is defined (and thus equal to $x$) for all $x \in X$.} action $e_i: \Cx \times X \rightarrow X$. We will usually denote the image $e_i(c,x)$ by $e_i^c(x)$.
\end{itemize}
We call $e_i$ a \emph{geometric crystal operator}, and we usually denote its action by $e_i^c(x)$ instead of $e_i(c,x)$. These rational maps must satisfy the following identities (whenever both sides are defined):
\begin{itemize}
\item For $x \in X$ and $c \in \Cx$,
\[
\gamma(e_i^c(x)) = \alpha_i^\vee(c)\gamma(x), \quad \vp_i(e_i^c(x)) = c^{-1}\vp_i(x), \quad \ve_i(e_i^c(x)) = c\ve_i(x),
\]
where $\alpha_i^\vee(c) = (1, \ldots, c, c^{-1}, \ldots, 1)$, with $c$ in the $i^{th}$ component and $c^{-1}$ in the $(i+1)^{th}$ component (mod $n$).
\item For $x \in X$,
\[
\dfrac{\ve_i(x)}{\vp_i(x)} = \alpha_i(\gamma(x)),
\]
where $\alpha_i(z_1, \ldots, z_n) = \dfrac{z_i}{z_{i+1}}$, with subscripts interpreted mod $n$.
\end{itemize}
\end{defn}

\begin{defn}
\label{defn_geom_cryst}
A {\em geometric crystal} is a geometric pre-crystal which satisfies the following \emph{geometric Serre relations}:

If $n \geq 3$, then for each pair $i,j \in \mathbb{Z}/n\mathbb{Z}$, and $c_1, c_2 \in \Cx$, the actions $e_i, e_j$ satisfy
\begin{equation*}
\begin{array}{ll}
e_i^{c_1}e_j^{c_2} = e_j^{c_2}e_i^{c_1} & \text{if } | i - j | > 1 \\
e_i^{c_1}e_j^{c_1c_2}e_i^{c_2} = e_j^{c_2}e_i^{c_1c_2}e_j^{c_1} & \text{if } | i - j | = 1.
\end{array}
\end{equation*}
If $n = 2$, there is no Serre relation for $e_0$ and $e_1$, so a geometric pre-crystal of type $A_1^{(1)}$ is automatically a geometric crystal.
\end{defn}

\begin{defn}
\label{defn_morphism}
A {\em morphism} of geometric (pre-)crystals $X$ and $Y$ is a rational map $h : X \rightarrow Y$ such that $e_i h = h e_i$, and $\rho h = \rho$ for $\rho = \gamma, \vp_i, \ve_i$.
\end{defn}

\begin{defn}
A {\em decorated geometric (pre-)crystal} is a geometric (pre-)crystal $X$ equipped with a rational function $f : X \rightarrow \bbC$ such that
\begin{equation}
\label{eq_f_prop}
f(e_i^c(x)) = f(x) + \frac{c-1}{\vp_i(x)} + \frac{c^{-1} - 1}{\ve_i(x)}
\end{equation}
for $x \in X$ and $i \in \Zn$.
The function $f$ is called a {\em decoration}.
\end{defn}

The purpose of the decoration is to ``cut out'' a combinatorial crystal from the lattice $\bb{Z}^{\dim X}$ after tropicalizing (see Theorem \ref{thm_recover_crystals} for an example).

\begin{defn_prop}[Berenstein--Kazhdan {\cite[Lem. 2.34]{BKII}}]
\label{defn_prop_prod_geom}
Suppose $X$ and $Y$ are decorated geometric pre-crystals. Define the following rational maps on $(x,y) \in X \times Y$:
\[
\gamma(x,y) = \gamma(x) \gamma(y)
\]
\[
\vp_i(x,y) = \dfrac{\vp_i(x)(\ve_i(x) + \vp_i(y))}{\ve_i(x)} \quad\quad \ve_i(x,y) = \dfrac{\ve_i(y)(\ve_i(x) + \vp_i(y))}{\vp_i(y)}
\]
\[
e_i^c(x,y) = (e_i^{c_1}(x), e_i^{c_2}(y)) \quad \text{ where } \quad c_1 = \dfrac{c\ve_i(x) + \vp_i(y)}{\ve_i(x) + \vp_i(y)}, \quad c_2 = \dfrac{\ve_i(x) + \vp_i(y)}{\ve_i(x) + c^{-1}\vp_i(y)}
\]
\[
f(x,y) = f(x) + f(y).
\]
These maps make $X \times Y$ into a decorated geometric pre-crystal, which we call the {\em product} of $X$ and $Y$. This product is associative.
\end{defn_prop}

If $X$ and $Y$ are geometric crystals, their product is not necessarily a geometric crystal (\cite[Rem. 2.21]{BKII}). To get around this problem, Berenstein and Kazhdan introduced unipotent crystals, and showed that if $X$ and $Y$ are induced from unipotent crystals, then their product is a geometric crystal.


\subsection{Unipotent crystals}
\label{sec_unip}

The definition of geometric pre-crystal given in the previous section makes sense for any reductive group $G$; one simply replaces the torus $(\Cx)^n$ by a maximal torus in $G$, and $\alpha_i, \alpha_i^\vee$ with the corresponding simple characters and cocharacters. Given a geometric pre-crystal, it is in general quite difficult to verify the geometric Serre relations, and, as mentioned above, the fact that the Serre relations hold for $X$ and $Y$ does not guarantee that they hold for $X \times Y$. Berenstein and Kazhdan invented unipotent crystals to get around these difficulties \cite{BKI}. The intuitive idea behind a unipotent crystal is that a geometric pre-crystal which ``comes from'' $G$ itself will automatically satisfy the Serre relations, and will automatically behave nicely under products.

Nakashima extended the notions of geometric and unipotent crystals to the Kac--Moody setting \cite{Nak}. The (minimal) Kac-Moody group which corresponds to the affine Lie algebra $\wh{\mathfrak{sl}}_n$ is closely related to $\SL_n(\bbC[\lp, \lp^{-1}])$, the group of $n \times n$ matrices of determinant 1 with entries in the Laurent polynomial ring $\bbC[\lp, \lp^{-1}]$. For our purposes, however, we have found it necessary to allow determinants other than 1, so we work with the bigger group $\GL_n(\bbC(\lp))$, which consists of $n \times n$ matrices with entries in the field of rational functions in the indeterminate $\lp$, and nonzero determinant. We call $\GL_n(\bbC(\lp))$ the \emph{loop group}, and $\lp$ the \emph{loop parameter}.

Before giving the definition of unipotent crystals, we pause to discuss a correspondence between $n \times n$ matrices with entries in $\bbC((\lp))$ and ``infinite periodic'' matrices with entries in $\bbC$. This construction generalizes the correspondence between formal Laurent series and Toeplitz matrices (which is the $n=1$ case), and plays a central role in Lam and Pylyavskyy's study of total positivity in the loop group \cite{LPwhirl}.

\subsubsection{Unfolding}
\label{sec_unfold}

Let $\mathbb{C}((\lp))$ be the field of formal Laurent series in the indeterminate $\lp$, that is, expressions of the form
\[
\sum_{m = m_0}^{\infty} a_m \lp^m
\]
where $m_0$ is an integer, and each $a_m$ is in $\mathbb{C}$. Let $M_n[\mathbb{C}((\lp))]$ denote the ring of $n \times n$ matrices with entries in this field.

An {\em $n$-periodic matrix} (over $\mathbb{C}$) is a $\mathbb{Z} \times \mathbb{Z}$ array of complex numbers $(X_{ij})_{(i,j) \in \mathbb{Z}}$ such that $X_{ij} = 0$ if $j-i$ is sufficiently large, and $X_{ij} = X_{i+n,j+n}$ for all $i,j$. Say that the entries $X_{ij}$ with $i-j = k$ lie on the {\em $k^{th}$ diagonal} of $X$, or that $k$ indexes this diagonal. Thus, the main diagonal of $X$ is indexed by 0, and higher numbers index lower diagonals. We add these matrices entry-wise, and multiply them using the usual matrix product: if $X = (X_{ij})$ and $Y = (Y_{ij})$, then
\[
(XY)_{ij} = \sum_{k \in \mathbb{Z}} X_{ik}Y_{kj}.
\]
The hypothesis that $X_{ij} = 0$ for $j-i$ sufficiently large ensures that each of these sums is finite, and it is clear that the product of two $n$-periodic matrices is $n$-periodic. Denote the ring of $n$-periodic matrices by $M_n^{\infty}(\mathbb{C})$.

Given a matrix $A = (A_{ij}) \in M_n[\mathbb{C}((\lp))]$, where $A_{ij} = \sum a_m^{i,j} \lp^m$, define an $n$-periodic matrix $X = (X_{ij})$ by
\[
X_{r n + i, s n + j} = a^{i,j}_{r - s}
\]
for $r,s \in \bbZ$ and $i,j \in [n]$. For example, if $n=2$ and
\[
A = \left(
\begin{array}{cc}
2\lp^{-1} + 3 + 4\lp + 5\lp^2 & \lp^{-1} + 7 + 8\lp \\
-3\lp^{-1} + 1 + \lp^2 & -2\lp^{-1} + 5 + 6\lp
\end{array}
\right)
\]
then
\[
X = \left(
\begin{array}{c|cc|cc|cc|c}
\ddots &&&&&&& \iddots \\ \hline
&3 & 7 & 2 & 1 & 0 & 0 \\
&1 & 5 & -3 & -2 & 0 & 0 \\ \hline
&4 & 8 & 3 & 7 & 2 & 1 \\
&0 & 6 & 1 & 5 & -3 & -2 \\ \hline
&5 & 0 & 4 & 8 & 3 & 7 \\
&1 & 0 & 0 & 6 & 1 & 5 \\ \hline
\iddots &&&&&&& \ddots
\end{array}
\right)
\]
where the row (resp., column) indexed by 1 is the upper-most row (resp., left-most column) whose entries are shown. The vertical and horizontal lines partition the matrix into $2 \times 2$ blocks whose entries are the $m^{th}$ coefficients of the entries of $A$, for some $m$.

It is straightforward to check that the map $A \mapsto X$ is an isomorphism of rings. We will refer to the $n \times n$ matrix $A$ as a {\em folded matrix}, and the $n$-periodic matrix $X$ as an {\em unfolded matrix}. We call $X$ the {\em unfolding} of $A$, and $A$ the {\em folding} of $X$. When it is important to distinguish between folded and unfolded matrices, we will try to use letters near the beginning of the alphabet for folded matrices, and letters near the end of the alphabet for unfolded matrices.

\subsubsection{Definition of unipotent crystals}
\label{sec_defn_unip}

Every rational function in $\lp$ has a Laurent series expansion, so $\GL_n(\bbC(\lp))$ is a subset of $M_n[\mathbb{C}((\lp))]$, and we may talk about the unfoldings of its elements.

In what follows, we will work with the submonoid $G \subset \GL_n(\bbC(\lp))$ consisting of matrices whose entries are Laurent polynomials in $\lp$, and whose determinant is a nonzero Laurent polynomial in $\lp$. The purpose of restricting to this monoid is that it is an ind-variety, so we may talk about rational maps to and from this space. (For our purposes, an ind-variety is simply an infinite-dimensional object that admits rational maps. We refer the reader to \cite{Kumar} for more information about ind-varieties.)

Let $B^- \subset G$ be the submonoid of matrices whose unfolding is lower triangular with nonzero entries on the main diagonal. In terms of folded matrices, this means that all entries are (ordinary) polynomials in $\lp$, with the entries on the main diagonal having nonzero constant term, and the entries above the main diagonal having no constant term. $B^-$ is naturally an ind-variety, where the $m^{th}$ piece consists of unfolded matrices which are supported on diagonals $0, \ldots, m$.

For $a \in \bbC$, define the folded matrices
\[
\wh{x}_i(a) = Id + aE_{i,i+1} \quad \text{ for } i \in [n-1], \quad\quad\quad \text{ and } \quad \wh{x}_0(a) = Id + a\lp^{-1}E_{n1},
\]
where $Id$ is the $n \times n$ identity matrix, and $E_{ij}$ is an $n \times n$ matrix unit. For $i \in \bbZ$, set $\wh{x}_i(a) = \wh{x}_{\ov{i}}(a)$, where $\ov{i}$ is the residue of $i$ mod $n$ (in $\{0, \ldots, n-1\}$). Let $U \subset G$ be the subgroup generated by the elements $\wh{x}_i(a)$. Note that the unfolding of each element of $U$ is upper uni-triangular.

The usual definition of unipotent crystals (\cite{BKI, Nak}) is based on rational actions of $U$. We work here with a slightly weaker notion, which was used in \cite{F1}.

\begin{defn}
\label{defn_pseudo}
Let $V$ be a complex algebraic (ind-)variety, and let $\alpha : U \times V \rightarrow V$ be a partially-defined map. Let $u.v := \alpha(u,v)$. We say that $\alpha$ is a {\em pseudo-rational $U$-action} if it satisfies the following properties:
\begin{enumerate}
\item $1.v = v$ for all $v \in V$;
\item If $u.v$ and $u'.(u.v)$ are defined, then $(u'u).v =u'.(u.v)$;
\item For each $i \in \Zn$, the partially defined map $\bb{C} \times V \rightarrow V$ given by $(a,v) \mapsto \wh{x}_i(a).v$ is rational.
\end{enumerate}
\end{defn}

\begin{defn}
\label{defn_B-_action}
Define $\alpha_{B^-} : U \times B^- \rightarrow B^-$ by $u.b = b'$ if $ub = b'u'$, with $b' \in B^-, u' \in U$. If $ub$ does not have such a factorization, then $u.b$ is undefined.
\end{defn}

Note that if $b_1u_1 = b_2u_2$, then $b_2^{-1}b_1 = u_2u_1^{-1}$ is both lower triangular and upper uni-triangular (as an unfolded matrix), so it must be the identity matrix, and thus $b_1 = b_2$ and $u_1 = u_2$. This shows that $\alpha_{B^-}$ is well-defined (as a partial map). Observe that if $X \in B^-$ is an unfolded matrix and $i \in \bbZ$, then
\begin{equation*}
\wh{x}_i(a) \cdot X \cdot \wh{x}_i\left(\frac{-aX_{i+1,i+1}}{X_{ii} + aX_{i+1,i}}\right) \in B^-,
\end{equation*}
so we have
\begin{equation}
\label{eq_pseudo_action}
\wh{x}_i(a).X = \wh{x}_i(a) \cdot X \cdot \wh{x}_i(\tau_i(a,X)) \quad\quad \text{ where } \quad\quad \tau_i(a,X) = \frac{-aX_{i+1,i+1}}{X_{ii} + aX_{i+1,i}}.
\end{equation}
This shows that $\alpha_{B^-}$ satisfies property (3) of Definition \ref{defn_pseudo}. It is clear that the first two properties are satisfied as well, so $\alpha_{B^-}$ is a pseudo-rational $U$-action.

\begin{defn}
\label{defn_U_variety}
A {\em $U$-variety} is an irreducible complex algebraic (ind-)variety $X$ together with a pseudo-rational $U$-action $\alpha : U \times X \rightarrow X$. A \emph{morphism of $U$-varieties} is a rational map which commutes with the $U$-actions (when they are defined).
\end{defn}

For example, the ind-variety $B^-$ with the pseudo-rational $U$-action $\alpha_{B^-}$ is a $U$-variety.

\begin{defn}
\label{defn_unip_cryst}
A {\em unipotent crystal (of type $A_{n-1}^{(1)}$)} is a pair $(V,g)$, where $V$ is a $U$-variety, and $g : V \rightarrow B^-$ is a morphism of $U$-varieties, such that for each $i \in [n]$ (equivalently, each $i \in \bbZ$), the rational function $v \mapsto g(v)_{i+1,i}$ is not identically zero (here $g(v)$ is viewed as an unfolded matrix).
\end{defn}

Note that the pair $(B^-,\Id)$ is a unipotent crystal.

The following result, which is essentially due to Berenstein and Kazhdan \cite[Thm. 3.8]{BKI} shows how to obtain a geometric crystal from a unipotent crystal. A sketch of the proof appears in \cite[\S 6.2]{F1}.

\begin{thm}
\label{thm_induces}
Let $(V,g)$ be a unipotent crystal. Suppose $v \in V$, and let $X = g(v)$ be an unfolded matrix. Define
\[
\gamma(v) = (X_{11}, \ldots, X_{nn}), \quad\quad \vp_i(v) = \dfrac{X_{i+1,i}}{X_{ii}}, \quad\quad \ve_i(v) = \dfrac{X_{i+1,i}}{X_{i+1,i+1}}, \quad\quad e_i^c(v) = \wh{x}_i\left(\dfrac{c-1}{\vp_i(v)}\right).v
\]
where $.$ is the pseudo-rational action of $U$ on $V$. These maps make $V$ into a type $A_{n-1}^{(1)}$ geometric crystal.
\end{thm}

We say that the geometric crystal on $V$ is {\em induced} from the unipotent crystal $(V,g)$. For example, the unipotent crystal $(B^-, \Id)$ induces a geometric crystal on $B^-$. A short computation using \eqref{eq_pseudo_action} shows that for $X \in B^-$,
\begin{equation}
\label{eq_U_action_B-}
e_i^c(X) = \wh{x}_i\left( \dfrac{c-1}{\vp_i(X)} \right) \cdot X \cdot \wh{x}_i\left( \dfrac{c^{-1} - 1}{\ve_i(X)} \right),
\end{equation}
where $e_i, \vp_i,$ and $\ve_i$ are the induced geometric crystal maps on $B^-$. Note that for any unipotent crystal $(V,g)$, we have by definition the formal identities
\begin{equation}
\label{eq_commutes_with_g}
\gamma = \gamma g, \quad\quad \vp_i = \vp_i g, \quad\quad \ve_i = \ve_i g, \quad\quad g e_i = e_i g,
\end{equation}
where the geometric crystal maps on the left-hand side come from the induced geometric crystal on $B^-$, and those on the right-hand side come from the induced geometric crystal on $V$.

\subsubsection{Product of unipotent crystals}
\label{sec_unip_prod}

We now define the product of unipotent crystals, following \cite{BKI}. Given $u \in U$ and $b \in B^-$, define $\beta(u,b) = u'$ if $ub = b'u'$, with $b' \in B^-$ and $u' \in U$. If $ub$ does not have such a factorization, then $\beta(u,b)$ is undefined (cf. Definition \ref{defn_B-_action}).

The following result is essentially the combination of Thm. 3.3 and Lem. 3.9 in \cite{BKI}. Although Berenstein and Kazhdan work with rational actions of the unipotent subgroup of a reductive group and we work with pseudo-rational actions of an infinite-dimensional group, the proof is identical.

\begin{thm}
\label{thm_unip_prod}
Suppose $(V,g)$ and $(W,g)$ are unipotent crystals. Define $g : V \times W \rightarrow B^-$ by $g(v,w) = g(v)g(w)$, and equip $V \times W$ with the pseudo-rational $U$-action
\[
u.(v,w) = (u.v, \beta(u,g(v)).w).
\]
Then $(V \times W, g)$ is a unipotent crystal. Furthermore, the geometric crystal induced from $(V \times W, g)$ is the product of the geometric crystals induced from $(V,g)$ and $(W,g)$.
\end{thm}



\subsection{Geometric and unipotent crystals on the Grassmannian}

\subsubsection{The Grassmannian}

Let $\Gr(k,n)$ be the Grassmannian of $k$-dimensional subspaces in $\mathbb{C}^n$. We view the Grassmannian as a projective algebraic variety in its Pl\"{u}cker embedding, and for $J \in {[n] \choose k}$, we write $P_J(N)$ for the $J^{th}$ Pl\"{u}cker coordinate of the subspace $N$. Pl\"ucker coordinates are projective---that is, they are only defined up to a common nonzero scalar multiple. We represent a point $N \in \Gr(k,n)$ as the column span of a (full-rank) $n \times k$ matrix $N'$, so that $P_J(N)$ is the maximal minor of $N'$ using the rows in $J$. When there is no danger of confusion, we treat a subspace and its matrix representatives interchangeably. For example, we may speak of the Pl\"{u}cker coordinates of a full-rank $n \times k$ matrix.

There is a natural (left) action of $\GL_n(\bbC)$ on $\Gr(k,n)$ given by matrix multiplication. We denote the action of $A \in \GL_n(\bbC)$ on $N \in \Gr(k,n)$ by $(A,N) \mapsto A \cdot N$; this is the subspace spanned by the columns of $A \cdot N'$, where $N'$ is an $n \times k$ matrix representative of $N$.

To simplify notation, we make the following convention.

\begin{conv}
\label{conv_pluc}
Let $N'$ be a full-rank $n \times k$ matrix representing a point $N \in \Gr(k,n)$.
\begin{enumerate}
\item
We label Pl\"{u}cker coordinates of $N$ by \underline{sets}, not by ordered lists. That is, if $I \in {[n] \choose k}$, then $P_I(N)$ means the determinant of the $k \times k$ submatrix of $N'$ using the rows indexed by the elements of $I$, taken in the order in which they appear in $N'$. Thus, $P_{\{1,2\}}(N) = P_{\{2,1\}}(N)$. We will often write $P_{12}(N)$ or $P_{1,2}(N)$ instead of $P_{\{1,2\}}(N)$.
\item
If $I \subset [n]$ does not contain exactly $k$ elements, then we set $P_I(N) = 0$.
\item
If $I$ is any set of integers, we set $P_I(N) = P_{I'}(N)$, where $I'$ is the set consisting of the residues of the elements of $I$ modulo $n$, where we take the residues to lie in $[n]$.
\end{enumerate}
\end{conv}

\subsubsection{Main definitions}
\label{sec_geom_unip_Gr}

For $k \in [n-1]$, let $\X{k}$ denote the variety $\Y{k}{n}$.\footnote{Since $n$ is fixed throughout the paper, we suppress the dependence on $n$ in the notation $\X{k}$.} We denote a point of $\X{k}$ by $N|t$, where $N \in \Gr(k,n)$ and $t \in \Cx$. In \cite{F1}, we defined a decorated geometric crystal on $\X{k}$, and a unipotent crystal which induces the geometric crystal. We recall those definitions here, starting with the unipotent crystal.

For $A \in \GL_n(\mathbb{C}(\lp))$ and $z \in \bbC$, let $A|_{\lp = z}$ denote the $n \times n$ matrix obtained by evaluating the loop parameter $\lp$ at $z$. This is defined as long as $z$ is not a pole of any entry of $A$; the resulting matrix is invertible if $z$ is not a root of the determinant of $A$. Define a $U$-action $U \times \X{k} \rightarrow \X{k}$ by
\begin{equation}
\label{eq_U_action_Gr}
u.(N|t) = (u|_{\lp = (-1)^{k-1} t} \cdot N)|t.
\end{equation}
Note that $u.(N|t)$ is always defined, since every element of $U$ has Laurent polynomial entries and determinant 1. This action makes $\X{k}$ into a $U$-variety.

\begin{defn}
\label{defn_g}
Define a rational map $g : \X{k} \rightarrow B^-$ by $g(N|t) = A$, where $A$ is the folded matrix defined by
\[
A_{ij} = c_{ij} \dfrac{P_{[j-k+1,j-1] \cup \{i\}}(N)}{P_{[j-k,j-1]}(N)}, \quad\quad\quad
c_{ij}  = 
\begin{cases} 1 & \text{ if } j \leq k \\
t & \text{ if } j > k \text{ and } i \geq j \\
\lp  & \text{ if } j > k \text { and } i < j.
\end{cases}
\]
This map is defined if and only if each of the \emph{cyclic Pl\"ucker coordinates} $P_{[j-k,j-1]}(N)$ is nonzero.
\end{defn}

See \eqref{eq_ex_2_5} for an example of the matrix $g(N|t)$.

\begin{prop}[{\cite[Thm. 6.10]{F1}}]
\label{prop_is_unip_cryst}
The pair $(\X{k}, g)$ is a unipotent crystal.
\end{prop}

By Theorem \ref{thm_induces}, the unipotent crystal $(\X{k}, g)$ induces a geometric crystal on $\X{k}$. Unraveling the definitions, we obtain the following formulas for the geometric crystal structure on $\X{k}$.

\begin{itemize}
\item The map $\gamma : \X{k} \rightarrow (\Cx)^n$ is given by $\gamma(N|t) = (\gamma_1, \ldots, \gamma_n)$, where
\begin{equation*}
\gamma_i =
\begin{cases}
\dfrac{P_{[i-k+1,i]}(N)}{P_{[i-k,i-1]}(N)} \text{ if } 1 \leq i \leq k \bigskip \\
t \dfrac{P_{[i-k+1,i]}(N)}{P_{[i-k,i-1]}(N)} \text{ if } k+1 \leq i \leq n.
\end{cases}
\end{equation*}

\item For $i \in \Zn$, the functions $\vp_i, \ve_i : \X{k} \rightarrow \Cx$ are given by
\[
\vp_i(N|t) =
t^{-\delta_{i,0}} \dfrac{P_{[i-k+1,i-1] \cup \{i+1\}}(N)}{P_{[i-k+1,i]}(N)},
\]
\[
\ve_i(N|t) = t^{-\delta_{i,k}} \dfrac{P_{[i-k+1,i-1] \cup \{i+1\}}(N) P_{[i-k+1,i]}(N)}{P_{[i-k,i-1]}(N) P_{[i-k+2,i+1]}(N)}.
\]

\item For $i \in \Zn$, the rational action $e_i : \Cx \times \X{k} \rightarrow \X{k}$ is given by $e_i^c(N|t) = N'|t$, where
\[
N' = \begin{cases}
x_i\left(\dfrac{c-1}{\vp_i(N|t)}\right) \cdot N & \text{ if } i \neq 0 \smallskip \\
x_0\left(\dfrac{(-1)^{k-1}}{t} \cdot \dfrac{c-1}{\vp_0(N|t)}\right) \cdot N & \text{ if } i = 0.
\end{cases}
\]
Here $x_i(a) = Id + aE_{i,i+1}$ for $i \in [n-1]$, and $x_0(a) = Id + aE_{n1}$, where $Id$ is the $n \times n$ identity matrix, and $E_{ij}$ is an $n \times n$ matrix unit.
\end{itemize}

We now recall the definition of the decoration on $\X{k}$. Say that an $n$-periodic matrix $X$ is {\em $m$-shifted unipotent} if $X_{ij} = 0$ when $i-j > m$, and $X_{ij} = 1$ when $i-j = m$. If $X$ is $m$-shifted unipotent, define
\[
\chi(X) = \sum_{j = 1}^n X_{j+m-1,j}.
\]
It is easy to see that if $X$ is $m$-shifted unipotent and $Y$ is $m'$-shifted unipotent, then $XY$ is $(m+m')$-shifted unipotent, and
\begin{equation}
\label{eq_chi_additive}
\chi(XY) = \chi(X) + \chi(Y).
\end{equation}

If $N|t \in \X{k}$, then $g(N|t)$ is $(n-k)$-shifted unipotent. For example, the matrix $g(N|t)$ for $N \in \Gr(2,5)$ is shown above in \eqref{eq_ex_2_5}. This matrix is 3-shifted unipotent, and
\[
\chi(g(N|t)) = \frac{P_{35}(N)}{P_{45}(N)} + \frac{P_{14}(N)}{P_{15}(N)} + t \frac{P_{25}(N)}{P_{12}(N)} + \frac{P_{13}(N)}{P_{23}(N)} + \frac{P_{24}(N)}{P_{34}(N)}.
\]

\begin{defn}
\label{defn_dec}
Define $f : \X{k} \rightarrow \bbC$ by
\begin{equation*}
f(N|t) = \chi(g(N|t)) = \sum_{i \neq k} \frac{P_{\{i-k\} \cup [i-k+2,i]}(N)}{P_{[i-k+1,i]}(N)} + t \frac{P_{[2,k] \cup \{n\}}(N)}{P_{[1,k]}(N)}.
\end{equation*}
\end{defn}

\begin{lem}[{\cite[\S 6.3]{F1}}]
\label{lem_is_dec}
The function $f$ satisfies \eqref{eq_f_prop}, so it makes $\X{k}$ into a decorated geometric crystal.
\end{lem}

\subsubsection{Products}
\label{sec_product_notation}

For $k_1, \ldots, k_d \in [n-1]$, we use the notation $\X{k_1, \ldots, k_d}$ (or $\X{\mb{k}}$, where $\mb{k} = (k_1, \ldots, k_d)$) for the product $\X{k_1} \times \cdots \times \X{k_d}$. By Theorem \ref{thm_unip_prod}, the pair $(\X{k_1, \ldots, k_d},g)$ is a unipotent crystal, where $g : \X{k_1, \ldots, k_d} \rightarrow B^-$ is given by
\begin{equation}
\label{eq_g_prod}
g(x_1, \ldots, x_d) = g(x_1) \cdots g(x_d).
\end{equation}
This unipotent crystal induces a geometric crystal on $\X{k_1, \ldots, k_d}$ by Theorem \ref{thm_induces}. Furthermore, by Definition/Proposition \ref{defn_prop_prod_geom}, the function $f : \X{k_1, \ldots, k_d} \rightarrow \bbC$ defined by $f(x_1, \ldots, x_d) = f(x_1) + \ldots + f(x_d)$ is a decoration on $\X{k_1, \ldots, k_d}$. Note that Definition \ref{defn_dec} and \eqref{eq_chi_additive}, \eqref{eq_g_prod} imply that
\begin{equation}
\label{eq_prod_dec}
f(x_1, \ldots, x_d) = \chi(g(x_1, \ldots, x_d)).
\end{equation}

\subsubsection{Properties of the matrix $g(N|t)$}
\label{sec_g_props}

\begin{prop}[{\cite[Prop. 6.11]{F1}}]
\label{prop_g_props}
Suppose $N|t \in \X{k}$. Let $A = g(N|t)$, viewed as a folded matrix.
\begin{enumerate}
\item \label{itm:pi_g}
The first $k$ columns of $A$ span the subspace $N$.
\item \label{itm:rank_k}
The matrix $A|_{\lp = (-1)^{k-1}t}$ has rank $k$.
\item \label{itm:det_g}
The determinant of $A$ is $(t + (-1)^k \lp)^{n-k}$.
\end{enumerate}
\end{prop}

Combining Proposition \ref{prop_g_props} with some simple linear algebra, we obtain two statements that play an important role in the study of the geometric $R$-matrix in \S \ref{sec_geom_R}.

\begin{cor}
\label{cor_linear_alg}
Suppose $N|t \in \X{k}$ and $B \in M_n(\bbC[\lp,\lp^{-1}])$.
\begin{enumerate}
\item
\label{itm:swallows}
The first $k$ columns of $(g(N|t) \cdot B)|_{\lp = (-1)^{k-1}t}$ are contained in the subspace $N$.
\item
\label{itm:preserves_dim}
If $B|_{\lp = (-1)^{k-1}t}$ is invertible, then the matrix $(B \cdot g(N|t))|_{\lp = (-1)^{k-1}t}$ has rank $k$. Furthermore, the first $k$ columns of this matrix have full rank, and they span the subspace $B|_{\lp = (-1)^{k-1}t} \cdot N$.
\end{enumerate}
\end{cor}

\begin{proof}
By parts \eqref{itm:pi_g} and \eqref{itm:rank_k} of Proposition \ref{prop_g_props}, the column span of the matrix $g(N|t)|_{\lp = (-1)^{k-1}t}$ is the subspace $N$. Multiplication of this matrix by $B|_{\lp = (-1)^{k-1}t}$ on the right is equivalent to performing a sequence of (possibly degenerate) column operations, so all columns of the resulting matrix are contained in $N$, proving \eqref{itm:swallows}.

Part \eqref{itm:preserves_dim} follows from parts \eqref{itm:pi_g} and \eqref{itm:rank_k} of Proposition \ref{prop_g_props}, and the fact that invertible linear transformations preserve dimension.
\end{proof}


\subsection{Symmetries}
\label{sec_symm}

\subsubsection{$\Zn$ symmetry}
\label{sec_cyclic_symm}

\begin{defn}
\label{defn_PR}
Define the {\em cyclic shift map} $\PR : \X{k} \rightarrow \X{k}$ by $\PR(N|t) = N'|t$, where $N'$ is obtained from $N$ by shifting the rows down by 1 (mod $n$), and multiplying the new first row by $(-1)^{k-1} t$. We write $\PR_t$ to denote the map $N \mapsto N'$.
\end{defn}

For example, when $n=4$ and $k=2$, we have
\[
\left(\begin{array}{cc}
z_{11} & z_{12} \\
z_{21} & z_{22} \\
z_{31} & z_{32} \\
z_{41} & z_{42}
\end{array}\right)
\overset{\begin{array}{l}
\PR_t
\end{array}}
{\mapsto}
\left(\begin{array}{cc}
-t \cdot z_{41} & -t \cdot z_{42} \\
z_{11} & z_{12} \\
z_{21} & z_{22} \\
z_{31} & z_{32}
\end{array}\right).
\]
It is easy to see that $\PR$ is well-defined (that is, the definition does not depend on the choice of matrix representative for the subspace $N$), and that it has order $n$. Note that the Pl\"{u}cker coordinates of $N' = \PR_t(N)$ are given by
\begin{equation}
\label{eq_PR_pluc}
P_J(N') =
\begin{cases}
P_{J - 1}(N) & \text{ if } 1 \not \in J \\
t \cdot P_{J - 1}(N) & \text{ if } 1 \in J
\end{cases}
\end{equation}
where $J - 1$ is the subset obtained from $J$ by subtracting 1 from each element (mod $n$). Extend $\PR$ to a map $\X{k_1, \ldots, k_d} \rightarrow \X{k_1, \ldots, k_d}$ by
\begin{equation*}
\PR(x_1, \ldots, x_d) = (\PR(x_1), \ldots, \PR(x_d)).
\end{equation*}

Define the {\em shift map} $\sh$ on an $n$-periodic matrix $X$ by
\begin{equation*}
\sh(X)_{ij} = X_{i-1,j-1}.
\end{equation*}
This map is easily seen to be an automorphism of order $n$. It is the ``unipotent analogue'' of $\PR$ by the following result.

\begin{lem}
\label{lem_PR}
We have the identity $g \circ \PR = \sh \circ \, g$ of rational maps $\X{k_1, \ldots, k_d} \rightarrow B^-$.
\end{lem}

\begin{proof}
The $d=1$ case is \cite[Lem. 6.9(3)]{F1}. The general case follows from the $d=1$ case and the fact that $\sh$ is an automorphism.
\end{proof}

\begin{cor}
\label{cor_PR_minors}
Suppose $N|t \in \X{k}$. Let $A = g(N|t)$ and $A' = g(\PR(N|t))$, and view these as folded matrices. Then for $I,J \in {[n] \choose r}$, we have
\begin{equation}
\label{eq_PR_minors}
\Delta_{I,J}(A') = \begin{cases}
\Delta_{I-1,J-1}(A) & \text{ if } 1 \in I \cap J \text{ or } 1 \not \in I \cup J \\
(-1)^{r-1} \lp \cdot \Delta_{I-1,J-1}(A) & \text{ if } 1 \in I \setminus J \\
(-1)^{r-1} \lp^{-1} \cdot \Delta_{I-1,J-1}(A) & \text{ if } 1 \in J \setminus I.
\end{cases}
\end{equation}
\end{cor}

\begin{proof}
By Lemma \ref{lem_PR}, we have $A' = \sh(A)$. Observe that the submatrix $\sh(A)_{I,J}$ is obtained from the submatrix $A_{I-1,J-1}$ by the following two steps:
\begin{itemize}
\item If $1 \in I$, multiply the last row by $\lp$ and interchange it with the other $r-1$ rows.
\item If $1 \in J$, multiply the last column by $\lp^{-1}$ and interchange it with the other $r-1$ columns.
\end{itemize}
This implies \eqref{eq_PR_minors}.
\end{proof}

\subsubsection{Geometric Sch\"{u}tzenberger involution}
\label{sec_schutz}

For $z \in \Cx$, define $\pi^k_z : M_n(\bbC[\lp, \lp^{-1}]) \rightarrow \X{k}$ by
\begin{equation}
\label{eq_pi_k_z}
\pi^k_z(A) = N|z
\end{equation}
where $N$ is the subspace spanned by the first $k$ columns of the $n \times n$ matrix $A_z = A|_{\lp = (-1)^{k-1} z}$. This is a rational map: it is undefined if the first $k$ columns of $A_z$ do not have full rank.

Define the {\em flip map} $\fl$ on an $n \times n$ matrix $A$ by
\begin{equation*}
\fl(A)_{ij} = A_{n-j+1,n-i+1}.
\end{equation*}
In words, $\fl$ reflects the matrix over the anti-diagonal. It is easy to see that $\fl$ is an anti-automorphism.

\begin{defn}
\label{defn_S}
Define the {\em geometric Sch\"{u}tzenberger involution} $S : \X{k} \rightarrow \X{k}$ by
\[
S(N|t) = \pi^k_t \circ \fl \circ \, g(N|t).
\]
Continuing the notation used for $\PR$, we write $S_t$ to denote the map $N \mapsto N'$, where $N'|t = S(N|t)$. The map $S$ is an involution by \cite[Cor. 7.4(1)]{F1}.
\end{defn}

The Pl\"{u}cker coordinates of $S_t(N)$ will appear frequently enough in later sections that we introduce the following notation for them:
\begin{equation}
\label{eq_Q_def}
Q^J_t(N) := P_{w_0(J)}(S_t(N)),
\end{equation}
where $w_0(J)$ is the subset obtained by replacing each $j \in J$ with $n-j+1$. By Proposition \ref{prop_g_props}\eqref{itm:pi_g} (resp., the definition of $S$), the Pl\"{u}cker coordinates $P_J(N)$ (resp., $Q^J_t(N)$) are the maximal minors of the first $k$ columns (resp., last $k$ rows) of $g(N|t)$. Since the bottom left $k \times k$ submatrix of $g(N|t)$ is upper uni-triangular, we have
\begin{equation}
\label{eq_P_Q}
\dfrac{P_J(N)}{P_{[n-k+1,n]}(N)} = \Delta_{J,[k]}(g(N|t)) \quad \text{ and } \quad \dfrac{Q^J_t(N)}{Q^{[k]}_t(N)} = \Delta_{[n-k+1,n],J}(g(N|t)).
\end{equation}

Extend $S$ to a map $\X{k_1, \ldots, k_d} \rightarrow \X{k_d, \ldots, k_1}$ by
\begin{equation}
\label{eq_S_prod}
S(x_1, \ldots, x_d) = (S(x_d), \ldots, S(x_1)).
\end{equation}
Note that the order of the factors is reversed.

\begin{lem}
\label{lem_S}
We have the identity $g \circ S = \fl \circ \, g$ of rational maps $\X{k_1, \ldots, k_d} \rightarrow B^-$.
\end{lem}

\begin{proof}
The $d = 1$ case is \cite[Prop. 7.3]{F1}. The general case follows from the $d=1$ case and the fact that $\fl$ is an anti-automorphism.
\end{proof}

Note that as an immediate consequence of Lemma \ref{lem_S}, we have
\begin{equation}
\label{eq_S_minors}
\Delta_{I,J}(g(S(N|t))) = \Delta_{w_0(J), w_0(I)}(g(N|t)).
\end{equation}
Lemma \ref{lem_S} also allows us to express the entries of the matrix $g(N|t)$ in terms of the Pl\"{u}cker coordinates $Q^J_t(N)$.

\begin{lem}
\label{lem_g_in_Q}
We have
\[
g(N|t)_{ij} = c'_{ij} \dfrac{Q^{[i+1,i+k-1] \cup \{j\}}_t}{Q^{[i+1,i+k]}_t}, \quad\quad\quad
c'_{ij} = 
\begin{cases} 1 & \text{ if } i > n-k \\
t & \text{ if } i \leq n-k \text{ and } i \geq j \\
\lp  & \text{ if } i \leq n-k \text { and } i < j.
\end{cases}
\]
\end{lem}

\begin{proof}
By Lemma \ref{lem_S}, we have $g(N|t) = \fl \circ \, g \circ S(N|t)$, so
\begin{align*}
g(N|t)_{ij} = (g \circ S(N|t))_{n-j+1,n-i+1} &= c_{n-j+1,n-i+1} \dfrac{P_{[n-i-k+2,n-i] \cup \{n-j+1\}}(S_t(N))}{P_{[n-i-k+1,n-i]}(S_t(N))} \\
&= c_{n-j+1,n-i+1} \dfrac{Q_t^{[i+1,i+k-1] \cup \{j\}}(N)}{Q_t^{[i+1,i+k]}(N)}.
\end{align*}
Clearly $c_{n-j+1,n-i+1} = c'_{ij}$, so we are done.
\end{proof}

\subsubsection{The dual Grassmannian}
\label{sec_duality}

Given a subspace $N \subset \mathbb{C}^n$, let $N^\perp$ be the orthogonal complement of $N$ with respect to the non-degenerate bilinear form given by $\langle v_i, v_j \rangle = (-1)^{i+1} \delta_{i,j}$, where $v_1, \ldots, v_n$ is the standard basis. Note that if $N \in \Gr(k,n)$, then $N^\perp$ is in the ``dual Grassmannian'' $\Gr(n-k,n)$.

Let $T_{w_0}: \Gr(k,n) \rightarrow \Gr(k,n)$ be the automorphism induced by reversing the standard basis of $\bbC^n$. Explicitly, if $N'$ is an $n \times k$ matrix representative for a subspace $N$, then $T_{w_0}(N)$ is the subspace represented by $Q_{w_0} \cdot N'$, where $Q_{w_0}$ is the permutation matrix corresponding to the longest element of $S_n$. Note that $\Delta_{J,[k]}(T_{w_0}(N)) = (-1)^{k(k-1)/2} \Delta_{w_0(J),[k]}(N)$, so $P_J(T_{w_0}(N)) = P_{w_0(J)}(N)$ (as projective coordinates), where, as above, $w_0(J)$ is the subset obtained by replacing each $j \in J$ with $n-j+1$.

Define $\mu : \X{k} \rightarrow \X{n-k}$ by
\[
\mu(N|t) = T_{w_0}(N^\perp)|t.
\]
Slightly abusing notation, we also write $\mu$ for the map $N \mapsto T_{w_0}(N^\perp)$. For $J \subset [n]$, let $\ov{J}$ be the complement $[n] \setminus J$, and let $J^* = w_0(\ov{J})$. By \cite[Lem. 7.5]{F1}, we have
\begin{equation}
\label{eq_dual_Plucker}
P_J(\mu(N)) = P_{J^*}(N).
\end{equation}
Extend $\mu$ to a map $\X{k_1, \ldots, k_d} \rightarrow \X{n-k_1, \ldots, n-k_d}$ by acting component-wise.

\begin{defn}
\label{defn_D}
Define the {\em duality map} $D : \X{k_1, \ldots, k_d} \rightarrow \X{n-k_d, \ldots, n-k_1}$ by
\[
D = S \circ \mu.
\]
For $N \in \Gr(k,n)$ and $t \in \Cx$, let $D_t(N) = S_t(\mu(N))$. The map $D$ is an involution by \cite[Cor. 7.9(1)]{F1}.
\end{defn}

Define $\inv : \GL_n(\bbC(\lp)) \rightarrow \GL_n(\bbC(\lp))$ on a folded matrix $A$ by
\[
\inv(A)_{ij} = (-1)^{i+j} \det(A_\lp) (A^{-1}_\lp)_{ij} = \Delta_{[n] \setminus \{j\}, [n] \setminus \{i\}}(A_\lp),
\]
where $A_\lp = A|_{\lp = (-1)^n \lp}$. It is easy to see that $\inv$ is an anti-automorphism.

\begin{lem}
\label{lem_D}
Suppose $(N_1|t_1, \ldots, N_d|t_d) \in \X{k_1, \ldots, k_d}$. If $g(N_j|t_j)$ is defined for all $j$, then
\[
\beta \cdot g \circ D(N_1|t_1, \ldots, N_d|t_d) = \inv \circ \, g(N_1|t_1, \ldots, N_d|t_d),
\]
where $\beta = \prod_{j=1}^d (t_j + (-1)^{k_j+n} \lp)^{n-k_j-1}$.
\end{lem}

\begin{proof}
The $d = 1$ case is \cite[Prop. 7.8]{F1}. The general case follows from the $d=1$ case and the fact that $\inv$ is an anti-automorphism.
\end{proof}

\begin{cor}
\label{cor_D_minors}
Suppose $N|t \in \X{k}$. Let $A = g(N|t)$ and $A' = g(D(N|t))$, viewed as folded matrices. Then for $I,J \in {[n] \choose r}$, we have
\begin{equation*}
\Delta_{I,J}(A') = (t + (-1)^{n-k} \lp)^{r-(n-k)} \Delta_{\ov{J}, \ov{I}}(A|_{\lp = (-1)^n \lp}).
\end{equation*}
\end{cor}

\begin{proof}
First note that Jacobi's formula for complementary minors of inverse matrices (see, e.g., \cite{Jacobi}) implies that if $C \in \GL_n(\bbC(\lp))$ and $I,J \in {[n] \choose r}$, then
\begin{equation}
\label{eq_Jacobi}
\Delta_{I,J}(\inv(C)) = \det(C_\lp)^{r-1} \Delta_{\ov{J}, \ov{I}}(C_\lp).
\end{equation}

Set $\alpha = t + (-1)^{n-k} \lp$. By Lemma \ref{lem_D}, we have
\[
A' = \dfrac{1}{\alpha^{n-k-1}} \inv(A).
\]
By Proposition \ref{prop_g_props}\eqref{itm:det_g}, the determinant of $A_\lp$ is $\alpha^{n-k}$, so using \eqref{eq_Jacobi}, we have
\[
\Delta_{I,J}(A') = \dfrac{1}{\alpha^{(n-k-1)r}} \alpha^{(n-k)(r-1)} \Delta_{\ov{J}, \ov{I}}(A_\lp) = \alpha^{r-(n-k)} \Delta_{\ov{J}, \ov{I}}(A_\lp).
\qedhere
\]
\end{proof}

\section{Positivity and tropicalization}
\label{sec_trop}

\subsection{The Gelfand--Tsetlin parametrization of $\X{k}$}
\label{sec_GT_param}

Recall from \S \ref{sec_k_rect} that a $k$-rectangle is an array of $k(n-k)+1$ nonnegative integers satisfying certain inequalities; $k$-rectangles parametrize the set of rectangular SSYTs with $k$ rows. By replacing integers with nonzero complex numbers, we obtain a ``rational version'' of $k$-rectangles, as follows. Let
\[
\bT{k} = (\Cx)^{R_k} \times \Cx
\]
where $R_k$ is the indexing set defined by \eqref{eq_R_k}. Denote a point of $\bT{k}$ by $(X_{ij}, t)$, where $(i,j)$ runs over $R_k$. We call $(X_{ij},t)$ a {\em rational $k$-rectangle}. Set
\begin{equation}
\label{eq_x_ij}
x_{ij} = X_{ij}/X_{i,j-1}
\end{equation}
for $1 \leq i \leq k$ and $i \leq j \leq i+n-k$, where we use the convention $X_{i,i-1} = 1$ and $X_{i,i+n-k} = t$. The quantity $x_{ij}$ is the rational analogue of the number of $j$'s in the $i^{th}$ row of a tableau (cf. \eqref{eq_row_coords}). Note that there are no inequality conditions on rational $k$-rectangles.

We now recall from \cite{F1} a parametrization of the variety $\X{k} = \Y{k}{n}$ by the set of rational $(n-k)$-rectangles. Given $a,b \in [n]$ and $z_a, \ldots, z_b \in \Cx$, define
\begin{equation}
\label{eq_M}
M_{[a,b]}(z_a, \ldots, z_b) = \sum_{i \in [a,b]} z_i E_{ii} + \sum_{i \not \in [a,b]} E_{ii} + \sum_{i \in [a+1,b]} E_{i,i-1}
\end{equation}
where $E_{ij}$ is an $n \times n$ matrix unit. For example, if $n = 5$, then
\begin{equation}
\label{eq_chev_gens}
M_{[2,4]}(z_2,z_3,z_4) = \left(
\begin{array}{ccccc}
1 &  &  &  &  \\
 & z_2 &  &  &  \\
 & 1 & z_3 &  &  \\
 &  & 1 & z_4 &  \\
 &  &  &  & 1
\end{array}
\right).
\end{equation}

\begin{defn}
\label{defn_Phi}
\
\begin{enumerate}
\item
Define $\Phi_{n-k} : \bT{n-k} \rightarrow \GL_n(\bbC)$ by
\[
\Phi_{n-k}(X_{ij}, t) = \prod_{i = n-k}^1 M_{[i,i+k]}(x_{ii}, x_{i,i+1}, \ldots, x_{i,i+k}),
\]
where $x_{ij}$ is defined by \eqref{eq_x_ij}, and the terms in the product are arranged from left to right in decreasing order of $i$. We call $\Phi_{n-k}(X_{ij},t)$ a {\em tableau matrix}.

\item
Define $\Theta_k : \bT{n-k} \rightarrow \X{k}$ by $\Theta_k(X_{ij},t) = N|t$, where $N$ is the ``projection'' of the tableau matrix $\Phi_{n-k}(X_{ij},t)$ onto the subspace spanned by its first $k$ columns. This may be written compactly as
\[
\Theta_k(X_{ij},t) = \pi^k_t \circ \Phi_{n-k}(X_{ij}, t)
\]
where $\pi^k_t$ is defined by \eqref{eq_pi_k_z}. We call $\Theta_k$ the {\em Gelfand--Tsetlin parametrization} of $\X{k}$.
\end{enumerate}
\end{defn}

\begin{ex}
\label{ex_2_5}
Suppose $n=5$ and $k=2$. For $(X_{ij},t) \in \bT{3}$, we have
\begin{equation}
\label{eq_2_5}
\Phi_3(X_{ij}, t) = \left(
\begin{array}{ccccc}
x_{11} & 0 & 0 & 0 & 0 \\
x_{22} & x_{12}x_{22} & 0 & 0 & 0 \\
x_{33} & (x_{12} + x_{23})x_{33} & x_{13}x_{23}x_{33} & 0 & 0 \\
1 & x_{12} + x_{23} + x_{34} & x_{13}(x_{23} + x_{34}) & x_{24}x_{34} & 0 \\
0 & 1 & x_{13} & x_{24} & x_{35}
\end{array}
\right)
\end{equation}
where $x_{ij}$ is defined by \eqref{eq_x_ij}. We have $\Theta_2(X_{ij},t) = N|t$, where $N$ is spanned by the first two columns of this tableau matrix.
\end{ex}

\begin{lem}[{\cite[Cor. 4.15(1)]{F1}}]
\label{lem_Theta_k_is_positive}
If $N|t = \Theta_k(X_{ij},t)$, then each Pl\"{u}cker coordinate $P_J(N)$ is given (up to a common scalar) by a nonzero Laurent polynomial in the $X_{ij}$, with non-negative integer coefficients.
\end{lem}

One consequence of Lemma \ref{lem_Theta_k_is_positive} is that if $N|t = \Theta_k(X_{ij},t)$, then $N$ does not depend on $t$. In light of this, we let $\ov{\Theta}_k : (\Cx)^{R_{n-k}} \rightarrow \Gr(k,n)$ denote the map $(X_{ij}) \mapsto N$, so that $\Theta_k = \ov{\Theta}_k \times \Id$.

\begin{remark}
The tableau matrix $\Phi_{n-k}(X_{ij},t)$ does depend on $t$, and is an interesting object of study in its own right. For example, the following result shows that the unipotent crystal map $g : \X{k} \rightarrow B^-$ is a ``deformation'' or ``cyclic extension'' of the map $\Phi_{n-k}$. See also Remark \ref{rmk_tab_mat_tab_prod}.
\end{remark}

\begin{lem}[{\cite[Prop. 6.11(4)]{F1}}]
\label{lem_g_0}
Suppose $(X_{ij},t) \in \bT{n-k}$. If $N|t = \Theta_k(X_{ij},t)$, then we have
\[
g(N|t)|_{\lp = 0} = \Phi_{n-k}(X_{ij},t).
\]
\end{lem}

We end this section by giving an explicit formula for $\ov{\Theta}_k^{-1}$. The formula makes use of a distinguished class of Pl\"{u}cker coordinates.

\begin{defn}
\label{defn_basic}
For $1 \leq i \leq n-k+1$ and $i-1 \leq j \leq i+k-1$, define the $k$-subset
\[
I_{i,j} = [i,j] \cup [n-k+j-i+2,n].
\]
We call a subset of this form a {\em basic $k$-subset} (or simply a {\em basic subset} if $k$ is understood), and we refer to $P_{I_{i,j}}$ as a {\em basic Pl\"{u}cker coordinate}. We call a subset of the form $w_0(I_{i,j})$ a {\em reflected basic $k$-subset}, and we call $P_{w_0(I_{i,j})}$ a {\em reflected basic Pl\"{u}cker coordinate}.
\end{defn}

Note that there is some redundancy in the definition of basic $k$-subsets: if $i = n-k+1$ or $j = i-1$, then $I_{i,j} = [n-k+1,n]$.

\begin{lem}[{\cite[Prop. 4.3]{F1}}]
\label{lem_GT_pl}
The map $\ov{\Theta}_k$ is an open embedding of $(\Cx)^{R_{n-k}}$ into $\Gr(k,n)$. The (rational) inverse is given by $N \mapsto X_{ij}$, where
\begin{equation}
\label{eq_GT_pl}
X_{ij} = \dfrac{P_{I_{i,j}}(N)}{P_{I_{i+1,j}}(N)}
\end{equation}
for $1 \leq i \leq n-k$ and $i \leq j \leq i+k-1$.
\end{lem}

As a corollary, we obtain an important property of basic (and reflected basic) Pl\"ucker coordinates; this property is a crucial tool in the proof of the positivity of the geometric $R$-matrix in \S \ref{sec_posit}.

\begin{lem}
\label{lem_basic=basis}
Every Pl\"{u}cker coordinate on $\Gr(k,n)$ can be expressed as a Laurent polynomial in the basic (resp., reflected basic) Pl\"{u}cker coordinates, with non-negative integer coefficients.
\end{lem}

\begin{proof}
For basic Pl\"ucker coordinates, this follows from combining Lemmas \ref{lem_Theta_k_is_positive} and \ref{lem_GT_pl}. The reflected version follows from the fact that $P_J \mapsto P_{w_0(J)}$ defines an automorphism of the homogeneous coordinate ring of the Grassmannian (this is the map induced by reversing the standard basis of the underlying $n$-dimensional vector space).
\end{proof}

\begin{remark}
\label{rmk_cluster_alg}
Lemma \ref{lem_basic=basis} is a special case of the (positive) Laurent phenomenon in the theory of cluster algebras. Indeed, the $k(n-k)+1$ basic (resp., reflected basic) Pl\"{u}cker coordinates are a cluster in the homogeneous coordinate ring of $\Gr(k,n)$ (see \cite[Fig. 18]{MarSco}).
\end{remark}


\subsection{Positivity}
\label{sec_posit_defn}

We say that a rational function $h \in \bbC(z_1, \ldots, z_d)$ is {\em positive} if it can be expressed as a ratio of two nonzero polynomials in $z_1, \ldots, z_d$ whose coefficients are positive integers. We call such an expression a {\em positive expression}. We say that $h$ is {\em non-negative} if it is either positive or identically zero. For example, $h = z_1^2 - z_1z_2 + z_2^2$ is positive because it has the positive expression $h = \dfrac{z_1^3 + z_2^3}{z_1 + z_2}$. (We remark that the term ``subtraction-free'' is often used in place of ``positive'' in the literature.)

We say that a rational map $h = (h_1, \ldots, h_{d_2}) : (\Cx)^{d_1} \rightarrow (\Cx)^{d_2}$ is a {\em positive map of tori} (or simply {\em positive}) if each $h_i$ is given by a positive element of $\bbC(z_1, \ldots, z_{d_1}$).

We now introduce a notion of positivity for rational maps between varieties more complicated than $(\Cx)^d$. Our definition is a stripped-down version of the definition in \cite{BKII}.

\begin{defn}
\label{defn_posit}
A {\em positive variety} is a pair $(X, \Theta_X)$, where $X$ is an irreducible complex algebraic variety, and $\Theta_X : (\Cx)^d \rightarrow X$ is a birational isomorphism. We say that $\Theta_X$ is a {\em parametrization} of $X$. When there is no danger of confusion, we refer to a positive variety by the name of its underlying variety.

Suppose $(X, \Theta_X)$ and $(Y, \Theta_Y)$ are positive varieties. A rational map $h: X \rightarrow Y$ is a {\em morphism of positive varieties} (or simply {\em positive}) if the rational map
\[
\Theta h := \Theta_Y^{-1} \circ f \circ \Theta_X : (\Cx)^{d_1} \rightarrow (\Cx)^{d_2}
\]
is a positive map of tori.
\end{defn}

\begin{remark}
If $h : X \rightarrow Y$ and $g : Y \rightarrow Z$ are rational maps, then the composition $g \circ h$ is undefined if the image of $h$ is disjoint from the domain of $g$. When we say that a composition of rational maps is positive, we implicitly guarantee that it is defined. For example, in the previous definition, $h$ is not positive if the image of $h$ is disjoint from the domain of $\Theta_Y^{-1}$. One nice feature of positive rational maps is that their composition is always defined, by the following result.
\end{remark}

\begin{lem}
\label{lem_posit_comp}
The composition of positive rational maps is positive.
\end{lem}

\begin{proof}
Let $(X, \Theta_X), (Y, \Theta_Y), (Z, \Theta_Z)$ be positive varieties, and suppose $h : X \rightarrow Y$ and $g : Y \rightarrow Z$ are positive rational maps. This means that
\[
\Theta h : (\Cx)^{d_1} \rightarrow (\Cx)^{d_2} \quad\quad \text{ and } \quad\quad \Theta g : (\Cx)^{d_2} \rightarrow (\Cx)^{d_3}
\]
are positive maps of tori. It is clear that $\Theta h$, being positive, is defined on all positive real points $(\bbR_{>0})^{d_1}$, and it maps these points into $(\bbR_{>0})^{d_2}$; similarly, $\Theta g$ is defined on $(\bbR_{>0})^{d_2}$, so $\Theta g \circ \Theta h = \Theta (g \circ h)$ is defined. Clearly this map is also a positive map of tori, so $g \circ h$ is positive.
\end{proof}

If $(X, \Theta_X)$ and $(Y, \Theta_Y)$ are positive varieties, then $(X \times Y, \Theta_X \times \Theta_Y)$ is a positive variety, and if $Z$ is another positive variety, then a rational map $(h_1, h_2) : Z \rightarrow X \times Y$ is positive if and only if $h_1$ and $h_2$ are positive.

The most basic example of a positive variety is the \emph{$d$-dimensional torus} $((\Cx)^d, \Id)$. A more interesting example is the pair $(\Gr(k,n), \ov{\Theta}_k)$ (note that $\ov{\Theta}_k$ is a birational isomorphism by Lemma \ref{lem_GT_pl}). All positive varieties considered in this paper are products of tori and Grassmannians. We now prove several necessary and sufficient conditions for rational maps to and from these varieties to be positive.

Say that a rational function $h : \prod_j \Gr(k_j,n) \times (\Cx)^d \rightarrow \bbC$ is {\em Pl\"{u}cker-positive} if it can be expressed as a ratio $a/b$, where $a$ and $b$ are nonzero polynomials with positive integer coefficients in the Pl\"ucker coordinates of the various Grassmannians, and the coordinates $z_1, \ldots, z_d$ of $(\Cx)^d$. We call such an expression $a/b$ a {\em Pl\"{u}cker-positive expression}. For example, the rational function $h = \dfrac{P_{13}P_{24} - P_{12}P_{34}}{P_{12}P_{34}}$ is Pl\"{u}cker-positive because it can be expressed as $h = \dfrac{P_{14}P_{23}}{P_{12}P_{34}}$ by a three-term Pl\"{u}cker relation.

It is clear that Pl\"ucker-positivity is equivalent to positivity for rational functions on $(\Cx)^d$. In fact, the same is true for rational functions on $\prod_j \Gr(k_j,n) \times (\Cx)^d$.

\begin{lem}
\label{lem_from_Gr}
A rational function $h : \prod_j \Gr(k_j,n) \times (\Cx)^d \rightarrow \bbC$ is positive (i.e., $\Theta h := h \circ (\prod_j \ov{\Theta}_{k_j} \times \Id)$ is positive) if and only if it is Pl\"{u}cker-positive.
\end{lem}

\begin{proof}
We assume that $h$ is a rational function on $\Gr(k,n)$ to simplify notation (the argument in the general case is the same). Let $(X_{ij})$ denote the coordinates on $(\Cx)^{R_{n-k}}$.

Suppose $h$ is Pl\"{u}cker-positive. By Lemma \ref{lem_Theta_k_is_positive}, each Pl\"{u}cker coordinate of the subspace $\ov{\Theta}_k(X_{ij})$ is given by a positive rational function in the $X_{ij}$. By choosing a Pl\"{u}cker-positive expression for $h$ and replacing the Pl\"{u}cker coordinates with these positive expressions in the $X_{ij}$, we obtain a positive expression for $\Theta h$, so $h$ is positive.

Now suppose $h$ is positive. Choose a positive expression for $\Theta h$ in terms of the $X_{ij}$, and replace each $X_{ij}$ with the ratio of Pl\"{u}cker coordinates in \eqref{eq_GT_pl}. This gives a Pl\"{u}cker-positive expression for $h \circ \ov{\Theta}_k \circ \ov{\Theta}_k^{-1} = h$.
\end{proof}

\begin{lem}
\label{lem_to_Gr}
Let $(X, \Theta_X)$ be a positive variety, and let $h : X \rightarrow \Gr(k,n)$ be a rational map. The following are equivalent:
\begin{enumerate}
\item $h$ is positive (i.e., $\ov{\Theta}_k^{-1} \circ h \circ \Theta_X$ is positive);
\item The rational functions $(P_J/P_I) \circ h : X \rightarrow \Cx$ are positive for all basic $k$-subsets $I,J$;
\item $(P_J/P_I) \circ h$ is positive for all reflected basic $k$-subsets $I,J$;
\item $(P_J/P_I) \circ h$ is positive for all $k$-subsets $I,J$.
\end{enumerate}
We say that a map satisfying these equivalent conditions is {\em Pl\"{u}cker-positive}.
\end{lem}

\begin{proof}
Conditions (2)-(4) are equivalent by Lemma \ref{lem_basic=basis}. We now show the equivalence of (1) and (2). Suppose $(X'_{ij}) = \ov{\Theta}_k^{-1}(N)$. Lemma \ref{lem_GT_pl} shows that
\[
\prod_{s=i}^j X'_{sj} = \dfrac{P_{I_{i,j}}(N)}{P_{I_{j+1,j}}(N)} = \dfrac{P_{I_{i,j}}(N)}{P_{[n-k+1,n]}(N)},
\]
so positivity of $\ov{\Theta}_k^{-1} \circ h \circ \Theta_X$ implies positivity of $(P_{I_{i,j}}/P_{[n-k+1,n]}) \circ h \circ \Theta_X$ for all basic subsets $I_{i,j}$. Conversely, each $X'_{ij}$ is a ratio of basic Pl\"ucker coordinates of $N$ (again by Lemma \ref{lem_GT_pl}), so if $(P_J/P_I) \circ h \circ \Theta_X$ is positive for all basic subsets $I$ and $J$, then $\ov{\Theta}_k^{-1} \circ h \circ \Theta_X$ is positive.
\end{proof}

Lemmas \ref{lem_from_Gr} and \ref{lem_to_Gr} show that for the varieties we consider, Pl\"ucker-positivity is equivalent to positivity. Thus, we will use the terms ``Pl\"ucker-positive'' and ``positive'' interchangeably from now on.

\begin{lem}
\label{lem_all_maps_posit}
Each of the rational maps $\gamma, \vp_i, \ve_i, e_i, f, \PR, S, D$ on $\X{k}$ (or $\Cx \times \X{k}$ in the case of $e_i$) is positive.
\end{lem}

\begin{proof}
The maps $\gamma, \vp_i, \ve_i, f$ are positive by the formulas used to define them, and $\PR$ is positive by \eqref{eq_PR_pluc}. The geometric crystal operator $e_i$ is positive by \cite[Lem. 5.6]{F1}. The geometric Sch\"{u}tzenberger involution $S$ is positive by \cite[Lem. 7.2]{F1}, and since $\mu$ is positive by \eqref{eq_dual_Plucker}, $D = S \circ \mu$ is positive by Lemma \ref{lem_posit_comp}.
\end{proof}


\subsection{Tropicalization of the geometric crystal structure}
\label{sec_trop_geom_cryst}

Tropicalization is a procedure for turning positive\footnote{For a more general notion of tropicalization that removes the positivity assumption, see \cite[\S 4]{BKII}.} rational maps $(\Cx)^{d_1} \rightarrow (\Cx)^{d_2}$ into piecewise-linear maps $\bbZ^{d_1} \rightarrow \bbZ^{d_2}$ by replacing the operations $+,\cdot,\div$ with the operations $\min, +, -$, and ignoring constants. More formally, if
\[
p = \sum c_{m_1, \ldots, m_d} z_1^{m_1} \cdots z_d^{m_d}
\]
is a nonzero polynomial in $z_1, \ldots, z_d$ with \underline{positive} integer coefficients, set
\[
\Trop(p) = \min_{(m_1, \ldots, m_d)} \{m_1 z_1 + \ldots + m_d z_d\}.
\]
Given a positive rational function $h \in \bbC(z_1, \ldots, z_d)$, define its {\em tropicalization} to be the piecewise-linear function from $\bbZ^d$ to $\bbZ$ given by
\[
\Trop(h) = \Trop(p) - \Trop(q),
\]
where $h = p/q$ is some expression of $h$ as a ratio of polynomials with positive integer coefficients (this definition doesn't depend on the choice of $p$ and $q$ by, e.g., \cite[Lem. 2.1.6]{BFZ}). For example,
\[
\Trop\left(\dfrac{z_1^2z_2 + z_3}{z_2^5 + 8z_1z_3 + 4}\right) = \min(2z_1 + z_2, z_3) - \min(5z_2, z_1+z_3,0).
\]

More generally, given a positive map of tori $h = (h_1, \ldots, h_{d_2}) : (\Cx)^{d_1} \rightarrow (\Cx)^{d_2}$, define $\Trop(h)$ to be the piecewise-linear map $(\Trop(h_1), \ldots, \Trop(h_{d_2})) : \bbZ^{d_1} \rightarrow \bbZ^{d_2}$. If $h,g : (\Cx)^{d_1} \rightarrow (\Cx)^{d_2}$ are positive, then we have
\[
\Trop(h + g) = \min(\Trop(h), \Trop(g)) \quad\quad \text{ and } \quad\quad \Trop(hg^{\pm 1}) = \Trop(h) \pm \Trop(g).
\]
Furthermore, tropicalization respects composition of positive maps.

\begin{defn}
\label{defn_trop}
Suppose $(X,\Theta_X)$ and $(Y, \Theta_Y)$ are positive varieties. If $h : X \rightarrow Y$ is a positive rational map, define its {\em tropicalization} $\wh{h}$ by
\[
\wh{h} = \Trop(\Theta h) = \Trop(\Theta_Y^{-1} \circ h \circ \Theta_X).
\]
\end{defn}

By Lemma \ref{lem_all_maps_posit}, the maps associated to the geometric crystal $\X{n-k} = \Y{n-k}{n}$ are positive, so we may tropicalize them and obtain piecewise-linear maps on $\tT{k}$ (see \S \ref{sec_k_rect}). It was shown in \cite{F1} that these piecewise-linear maps, when restricted to the set of $k$-rectangles inside $\tT{k}$, give formulas for the combinatorial crystal structure on $k$-row rectangular tableaux. Here we generalize to the case of products of the $\X{n-k}$.

For $\mb{k} = (k_1, \ldots, k_d) \in [n-1]^d$, let $n-\mb{k} = (n-k_1, \ldots, n-k_d)$. Recall that $\X{\mb{k}}$ denotes the product $\X{k_1} \times \cdots \times \X{k_d}$; define $\bT{\mb{k}}, \tT{\mb{k}}$ analogously. Note that $\X{n-\mb{k}}$ is a positive variety with respect to the parametrization $\Theta_{n-\mb{k}} = (\Theta_{n-k_1}, \ldots, \Theta_{n-k_d}) : \bT{\mb{k}} \rightarrow \X{n-\mb{k}}$.

Let $\gamma, \vp_i, \ve_i, e_i, f$ be the maps giving the decorated geometric crystal structure on $\X{n-\mb{k}}$. It follows from Lemma \ref{lem_all_maps_posit} and the formulas in Definition/Proposition \ref{defn_prop_prod_geom} that each of these maps is positive, so we may tropicalize them to obtain a piecewise-linear map $\wh{\gamma} : \tT{\mb{k}} \rightarrow \bbZ^n$, piecewise-linear functions $\wh{\vp}_i, \wh{\ve}_i, \wh{f} : \tT{\mb{k}} \rightarrow \bbZ$, and a piecewise-linear action $\wh{e}_i : \bbZ \times \tT{\mb{k}} \rightarrow \tT{\mb{k}}$.

\begin{thm}
\label{thm_recover_crystals}
Suppose $\mb{b} = (b_1, \ldots, b_d) \in \tT{\mb{k}}$. Then $\wh{f}(\mb{b}) \geq 0$ if and only if each $b_j$ is a $k_j$-rectangle. If $\wh{f}(\mb{b}) \geq 0$, then
\begin{enumerate}
\item $\wh{\gamma}(\mb{b}) = \tw{\gamma}(\mb{b})$;
\item $\wh{\vp}_i(\mb{b}) = -\tw{\vp}_i(\mb{b})$ and $\wh{\ve}_i(\mb{b}) = -\tw{\ve}_i(\mb{b})$;
\item $\tw{e}_i(\mb{b})$ is defined if and only if $\wh{f}(\wh{e}_i(1, \mb{b})) \geq 0$; in this case, $\wh{e}_i(1,\mb{b}) = \tw{e}_i(\mb{b})$.
\end{enumerate}
\end{thm}

\begin{proof}
The $d=1$ case is the combination of Prop. 5.3 and Thm. 5.7 in \cite{F1}. By \cite[Prop. 6.7]{BKII}, the general case follows from the $d = 1$ case.
\end{proof}

\begin{remark}
It was shown in \cite{F1} that the tropicalizations of the symmetries $\PR, S,$ and $D$, when restricted to the set of $k$-rectangles, are equal, respectively, to promotion, the Sch\"{u}tzenberger involution, and the map which replaces each column of a rectangular tableau with its complement in $[n]$, and reverses the order of the columns.
\end{remark}

\section{The geometric $R$-matrix}
\label{sec_geom_R}

\subsection{Definition of $R$}

Fix $\ell, k \in [n-1]$. Consider the unipotent crystals $(\X{\ell}, g)$ and $(\X{k},g)$ defined in \S \ref{sec_geom_unip_Gr}; recall that their product is the unipotent crystal $(\X{\ell} \times \X{k}, g)$, where $g(u,v) = g(u)g(v)$ for $(u,v) \in \X{\ell} \times \X{k}$. Recall the geometric Sch\"{u}tzenberger involution $S$ and the ``evaluation-projection'' $\pi^k_z$ from \S \ref{sec_schutz}. We encourage the reader to review Corollary \ref{cor_linear_alg}, which plays a crucial role in what follows.

Define a rational map $\Psi_{k,\ell} : \X{\ell} \times \X{k} \rightarrow \X{k}$ by
\[
\Psi_{k,\ell}(M|s, N|t) = \pi^k_t \circ \, g(M|s, N|t).
\]

\begin{defn}
\label{defn_R}
The {\em geometric $R$-matrix} is the rational map $R : \X{\ell} \times \X{k} \rightarrow \X{k} \times \X{\ell}$ defined by
\[
R = (\Psi_{k,\ell}, \; S \circ \Psi_{\ell,k} \circ S).
\]
\end{defn}

More explicitly, if $R(M|s, N|t) = (N'|t, M'|s)$, then by Corollary \ref{cor_linear_alg}\eqref{itm:preserves_dim} and Lemma \ref{lem_S}, we have
\begin{equation}
\label{eq_R_alt}
N' = g(M|s)|_{\lp = (-1)^{k-1}t} \cdot N \quad\quad \text{ and } \quad\quad S_s(M') = \fl (g(N|t))|_{\lp = (-1)^{\ell-1} s} \cdot S_s(M).
\end{equation}

\begin{remark}
The formulas \eqref{eq_R_alt} show that $N'$ is the image of $N$ under a linear map that depends on $M, s,$ and $t$, and $S_s(M')$ is the image of $S_s(M)$ under a linear map that depends on $N, s,$ and $t$. We would very much like to have a geometric interpretation of these linear maps.
\end{remark}

The two crucial results about $R$ are the following.

\begin{thm}
\label{thm_R_posit}
The geometric $R$-matrix is positive.
\end{thm}

\begin{thm}
\label{thm_hard}
We have the identity $g \circ R = g$ of rational maps $\X{\ell} \times \X{k} \rightarrow B^-$. That is, if $R(u,v) = (v',u')$ and $g(v'), g(u')$ are defined, then
\[
g(u)g(v) = g(v')g(u').
\]
\end{thm}


Theorems \ref{thm_R_posit} and \ref{thm_hard} are proved in \S \ref{sec_posit} and \S \ref{sec_pf1}, respectively. Most of the important properties of $R$ are direct consequences of Theorem \ref{thm_hard}. Here is an example.

\begin{lem}
\label{lem_R^2=Id}
We have the identity $R^2 = \Id$ of rational maps from $\X{\ell} \times \X{k}$ to itself.
\end{lem}

\begin{proof}
Suppose $(M|s, N|t) \in \X{\ell} \times \X{k}$, and
\[
(M|s, N|t) \; \xrightarrow{R} \; (N'|t, M'|s) \; \xrightarrow{R} \; (M''|s, N''|t).
\]
By Theorem \ref{thm_hard}, we have
\begin{equation}
\label{eq_M''}
M''|s = \pi^\ell_s(g(N'|t)g(M'|s)) = \pi^\ell_s(g(M|s)g(N|t)).
\end{equation}
Corollary \ref{cor_linear_alg}\eqref{itm:swallows} ensures that the first $\ell$ columns of $g(M|s)g(N|t)|_{\lp = (-1)^{\ell-1}s}$ are contained in the subspace $M$. On the other hand, \eqref{eq_M''} shows that these columns span the subspace $M''$, so we conclude that $M'' = M$.

Let $p_1, p_2$ be the projections of $\X{\ell} \times \X{k}$ onto the first and second factors, respectively. We have shown that $p_1 R^2 = p_1$. Clearly $p_2 = Sp_1S$ and $R$ commutes with $S$, so we also have
\[
p_2R^2 = Sp_1SR^2 = Sp_1R^2S = p_2.
\qedhere
\]
\end{proof}

Recall the notation $Q^J_t(N) = P_{w_0(J)}(S_t(N))$.

\begin{cor}
\label{cor_P_Q_R}
Suppose $M|s \in \X{\ell}, N|t \in \X{k},$ and $(N'|t, M'|s) = R(M|s, N|t)$. Let $B = g(M|s)g(N|t), B_s = B|_{\lp = (-1)^{\ell-1}s},$ and $B_t = B|_{\lp = (-1)^{k-1}t}$. For $k$-subsets $I$ and $\ell$-subsets $J$, we have
\begin{equation}
\label{eq_P_Q_N'_M'}
\dfrac{P_I(N')}{P_{[n-k+1,n]}(N')} = \dfrac{\Delta_{I,[k]}(B_t)}{\Delta_{[n-k+1,n],[k]}(B_t)}, \quad\quad \dfrac{Q_s^J(M')}{Q_s^{[\ell]}(M')} = \dfrac{\Delta_{[n-\ell+1,n], J}(B_s)}{\Delta_{[n-\ell+1,n],[\ell]}(B_s)},
\end{equation}
\begin{equation}
\label{eq_P_Q_M_N}
\dfrac{P_J(M)}{P_{[n-\ell+1,n]}(M)} = \dfrac{\Delta_{J,[\ell]}(B_s)}{\Delta_{[n-\ell+1,n],[\ell]}(B_s)}, \quad\quad \dfrac{Q_t^I(N)}{Q_t^{[k]}(N)} = \dfrac{\Delta_{[n-k+1,n], I}(B_t)}{\Delta_{[n-k+1,n],[k]}(B_t)}.
\end{equation}
\end{cor}

\begin{proof}
The equalities \eqref{eq_P_Q_N'_M'} follow from the definition of $R$ and Lemma \ref{lem_S}. The equalities \eqref{eq_P_Q_M_N} follow from Lemma \ref{lem_R^2=Id}, Theorem \ref{thm_hard}, and \eqref{eq_P_Q_N'_M'}.
\end{proof}


\subsection{Properties of $R$}
\label{sec_R_props}

Recall the notation for products introduced in \S \ref{sec_product_notation}. Suppose $\mb{k} = (k_1, \ldots, k_d) \in [n-1]^d$. For $i = 1, \ldots, d-1$, let $\sigma_i(\mb{k}) = (k_1, \ldots, k_{i+1},k_i, \ldots, k_d)$, and let
\[
R_i : \X{\mb{k}} \rightarrow \X{\sigma_i(\mb{k})}
\]
be the map which acts as the geometric $R$-matrix on factors $i$ and $i+1$, and as the identity on the other factors.

Say that a point $N|t \in \X{k}$ is {\em positive} if $t > 0$, and $P_J(N) > 0$ for all $J$. Let $\mathbb{U}_{\mb{k}}$ (and $\mathbb{U}_{k_1, \ldots, k_d}$) denote the subset of $\X{\mb{k}}$ consisting of $(N_1|t_1, \ldots, N_d|t_d)$ such that each $N_i|t_i$ is positive, and the $t_i$ are distinct. Note that $g$ is defined on $\mathbb{U}_{\mb{k}}$, and since the geometric $R$-matrix is positive and involutive, each $R_i$ is a bijection from $\mathbb{U}_{\mb{k}}$ to $\mathbb{U}_{\sigma_i(\mb{k})}$.

\begin{lem}
\label{lem_recover_N_1}
If $(N_1|t_1, \ldots, N_d|t_d) \in \mathbb{U}_{\mb{k}}$, then
\[
\pi_{t_1}^{k_1} \circ \, g(N_1|t_1, \ldots, N_d|t_d) = N_1|t_1.
\]
In other words, the first $k_1$ columns of the matrix $g(N_1|t_1, \ldots, N_d|t_d)|_{\lp = (-1)^{k_1-1}t_1}$ span the subspace $N_1$.
\end{lem}

\begin{proof}
Let $B = g(N_1|t_1, \ldots, N_d|t_d)$, and $B_{t_1} = B|_{\lp = (-1)^{k_1-1}t_1}$. By Corollary \ref{cor_linear_alg}\eqref{itm:swallows}, the first $k_1$ columns of $B_{t_1}$ are contained in the subspace $N_1$. Thus, it suffices to show that the first $k_1$ columns of $B_{t_1}$ have full rank whenever $(N_1|t_1, \ldots, N_d|t_d) \in \mathbb{U}_{\mb{k}}$.

Let
\[
(N_2'|t_2, \ldots, N_d'|t_d, N_1'|t_1) = R_{d-1} \circ \cdots \circ R_1(N_1|t_1, \ldots, N_d|t_d).
\]
By repeated applications of Theorem \ref{thm_hard}, we have $B = g(N_2'|t_2, \ldots, N_d'|t_d, N_1'|t_1)$. Since the absolute values of the $t_i$ are distinct, the matrix $g(N_2'|t_2, \ldots, N_d'|t_d)|_{\lp = (-1)^{k_1-1}t_1}$ is invertible by Proposition \ref{prop_g_props}{\eqref{itm:det_g}}, so the first $k_1$ columns of $B_{t_1}$ have full rank by Corollary \ref{cor_linear_alg}\eqref{itm:preserves_dim}.
\end{proof}


\begin{cor}
\label{cor_N_i_unique}
Suppose $(M_1|t_1, \ldots, M_d|t_d), (N_1|t_1, \ldots, N_d|t_d) \in \mathbb{U}_{\mb{k}}$. If
\begin{equation}
\label{eq_N_i_unique}
g(M_1|t_1, \ldots, M_d|t_d) = g(N_1|t_1, \ldots, N_d|t_d),
\end{equation}
then $M_i = N_i$ for all $i$.
\end{cor}

\begin{proof}
Lemma \ref{lem_recover_N_1} shows that $M_1 = N_1$. By Proposition \ref{prop_g_props}\eqref{itm:det_g}, the matrix $g(M_1|t_1)$ is invertible (in the ring $M_n(\bbC(\lp))$), so we may multiply both sides of \eqref{eq_N_i_unique} by $g(M_1|t_1)^{-1}$ and reduce to a smaller value of $d$.
\end{proof}

\begin{remark}
Corollary \ref{cor_N_i_unique} does not hold for arbitrary points in $\X{\mb{k}}$, even in the case $n = 2, \mb{k} = (1,1)$.
\end{remark}

\begin{thm}
\label{thm_master}
\
\begin{enumerate}
\item
\label{itm:R=isomorphism}
$R : \X{k_1} \times \X{k_2} \rightarrow \X{k_2} \times \X{k_1}$ is an isomorphism of geometric crystals.

\item
\label{itm:R_commutes_symms}
$R : \X{k_1} \times \X{k_2} \rightarrow \X{k_2} \times \X{k_1}$ commutes with the symmetries $\PR, S,$ and $D$.

\item
\label{itm:YB}
$R$ satisfies the Yang--Baxter relation. That is, we have the equality
\begin{equation}
\label{eq_YB}
R_1R_2R_1 = R_2R_1R_2
\end{equation}
of rational maps $\X{k_1,k_2,k_3} \rightarrow \X{k_3,k_2,k_1}$.

\end{enumerate}
\end{thm}

\begin{proof}
First we prove \eqref{itm:R=isomorphism}. By Lemma \ref{lem_R^2=Id}, $R$ is invertible, with inverse $R : \X{k_2} \times \X{k_1} \rightarrow \X{k_1} \times \X{k_2}$. Let $\rho$ be one of the maps $\gamma, \vp_i,$ or $\ve_i$. By \eqref{eq_commutes_with_g} and Theorem \ref{thm_hard}, we have
\[
\rho R = \rho g R = \rho g = \rho.
\]

It remains to show that $R$ commutes with $e_i$. Again by \eqref{eq_commutes_with_g} and Theorem \ref{thm_hard}, we have
\begin{equation}
\label{eq_g_R_e_again}
g R e_i = g e_i = e_i g = e_i g R = g e_i R.
\end{equation}
Suppose $\mb{x} = (N_1|t_1,N_2|t_2) \in \mathbb{U}_{k_1,k_2}$, and $c > 0$. Let
\begin{align*}
\mb{x}' &= (N_2'|t_2, N_1'|t_1) = Re_i^c(\mb{x}), \\
\mb{x}'' &= (N_2''|t_2, N_1''|t_1) = e_i^c R(\mb{x}).
\end{align*}
Since $R$ and $e_i$ are positive maps, we have $\mb{x}', \mb{x}'' \in \mathbb{U}_{k_2,k_1}$, and by \eqref{eq_g_R_e_again}, we have $g(\mb{x}') = g(\mb{x}'')$. Thus, $Re_i^c(\mb{x}) = e_i^cR(\mb{x})$ by Corollary \ref{cor_N_i_unique}. Since the set of points
\[
\{(c, \mb{x}) \, | \, c > 0 \text{ and } \mb{x} \in \mathbb{U}_{k_1,k_2}\}
\]
is dense in $\Cx \times \X{k_1,k_2}$, we conclude that $R e_i = e_i R$.

The proof of \eqref{itm:R_commutes_symms} is formally the same as the proof of the identity $Re_i = e_iR$, with Lemmas \ref{lem_PR}, \ref{lem_S}, and \ref{lem_D} playing the role of \eqref{eq_commutes_with_g} (we also use the positivity of the three symmetries).

The proof of \eqref{itm:YB} is similar. Suppose $\mb{x} = (N_1|t_1, N_2|t_2, N_3|t_3) \in \mathbb{U}_{k_1,k_2,k_3}$, and set
\begin{align*}
\mb{x}' &= (N_3'|t_3, N_2'|t_2, N_1'|t_1) = R_1R_2R_1(\mb{x}), \\
\mb{x}'' &= (N_3''|t_3, N_2''|t_2, N_1''|t_1) = R_2R_1R_2(\mb{x}).
\end{align*}
Theorem \ref{thm_hard} implies that $g(\mb{x}') = g(\mb{x}) = g(\mb{x}'')$, and since $\mb{x}', \mb{x}'' \in \mathbb{U}_{k_3,k_2,k_1}$, we have $\mb{x}' = \mb{x}''$ by Corollary \ref{cor_N_i_unique}. Since $\mathbb{U}_{\mb{k}}$ is dense in $\X{\mb{k}}$, this proves \eqref{eq_YB}.
\end{proof}

\begin{remark}
\label{rmk_many_factors}
Lemma \ref{lem_R^2=Id} and Theorem \ref{thm_master}\eqref{itm:YB} show that the geometric $R$-matrix satisfies the relations of the simple transpositions in the symmetric group. That is, if we repeatedly apply geometric $R$-matrices to consecutive factors in a product $\X{k_1, \ldots, k_d}$, the result depends only on the final permutation of factors. This means that for any permutation $\sigma \in S_d$, there is a well-defined map $R_\sigma : \X{\mb{k}} \rightarrow \X{\sigma(\mb{k})}$.

There is an efficient way to pick off the first and last factors of the image of a point in $\mathbb{U}_{\mb{k}}$ under $R_{\sigma}$. Indeed, if $\mb{x} = (x_1, \ldots, x_d) = (N_1|t_1, \ldots, N_d|t_d) \in \mathbb{U}_{\mb{k}}$ and $R_\sigma(\mb{x}) = (x'_{\sigma(1)}, \ldots, x'_{\sigma(d)})$, then we have
\[
x'_{\sigma(1)} = \pi^{k_{\sigma(1)}}_{t_{\sigma(1)}} \circ g(\mb{x}), \quad\quad x'_{\sigma(d)} = S \circ \pi^{k_{\sigma(d)}}_{t_{\sigma(d)}} \circ \fl \circ \, g(\mb{x})
\]
by Theorem \ref{thm_hard}, Lemma \ref{lem_recover_N_1}, Lemma \ref{lem_S}, and the fact that $S$ commutes with $R$.
\end{remark}


\subsection{Tropicalization of $R$}
\label{sec_trop_R}

By Theorem \ref{thm_hard}, the map $\Theta R : \bT{k_1} \times \bT{k_2} \rightarrow \bT{k_2} \times \bT{k_1}$ is positive, so we may define
\[
\wh{R} = \Trop(\Theta R) : \tT{k_1} \times \tT{k_2} \rightarrow \tT{k_2} \times \tT{k_1}.
\]

\begin{thm}
\label{thm_recover_comb_R}
If $a$ is a $k_1$-rectangle and $b$ is a $k_2$-rectangle, then $\wh{R}(a \otimes b) = \tw{R}(a \otimes b)$, where $\tw{R}$ is the combinatorial $R$-matrix.
\end{thm}

\begin{proof}
By \eqref{eq_prod_dec} and Theorem \ref{thm_hard}, we have $fR = f$, where $f$ is the decoration on $\X{n-k_1} \times \X{n-k_2}$. This, together with Theorems \ref{thm_recover_crystals} and \ref{thm_master}\eqref{itm:R=isomorphism}, implies that for any $L_1, L_2 \geq 0$, $\wh{R}$ restricts to an affine crystal isomorphism $B^{k_1, L_1} \otimes B^{k_2, L_2} \rightarrow B^{k_2, L_2} \otimes B^{k_1, L_1}$. The combinatorial $R$-matrix is the unique such isomorphism.
\end{proof}

In \S \ref{sec_1_2}, we illustrate Theorem \ref{thm_recover_comb_R} with a small example.

\begin{remark}
Theorem \ref{thm_recover_comb_R} allows us to deduce the Yang--Baxter relation for the combinatorial $R$-matrix from the Yang--Baxter relation for the geometric $R$-matrix, thereby giving a new proof of the former.
\end{remark}

\begin{remark}
\label{rmk_tab_mat_tab_prod}
We have used the crystal-theoretic characterization of the combinatorial $R$-matrix to prove that $R$ tropicalizes to $\tw{R}$. Here we outline an alternative proof based on the combinatorial characterization of $\tw{R}$ in terms of the tableau product (Proposition \ref{prop_comb_R}). The idea is that
\[
\textit{the product of tableau matrices tropicalizes to the product of tableaux,}
\]
where the tableau matrix $\Phi_k(X_{ij},t)$ is the $n \times n$ matrix defined in \S \ref{sec_GT_param}. To be a bit more precise, let $a$ be a $k_1$-rectangle corresponding to the tableau $T$, and $b$ a $k_2$-rectangle corresponding to the tableau $U$. Let $(C_{ij})$ be the Gelfand--Tsetlin pattern corresponding to the tableau $T * U$. Theorem 3.9 in \cite{F2abs} (an extended abstract of this paper) states that the product of tableau matrices $\Phi_{k_1}(x) \Phi_{k_2}(y)$ uniquely determines positive rational functions $Z_{ij}(x,y)$ which tropicalize to formulas for $C_{ij}$ in terms of the entries of $a$ and $b$. (In fact, the $Z_{ij}$ are ratios of left-justified minors of the product matrix.) If $\Theta R(x,y) = (y',x')$, then by Theorem \ref{thm_hard} and Lemma \ref{lem_g_0} we have
\[
\Phi_{k_1}(x)\Phi_{k_2}(y) = \Phi_{k_2}(y')\Phi_{k_1}(x'),
\]
so the rectangular tableaux $U'$ and $T'$ obtained by tropicalizing $y'$ and $x'$ satisfy $U' * T' = T * U$.

The special case of \cite[Thm. 3.9]{F2abs} with $k_1 = k_2 = 1$ was proved by Noumi and Yamada in their work on a geometric lift of the RSK correspondence \cite[\S 2.1-2]{NouYam}. The general case can be proved by iterating the one-row case. The technical details take up a lot of space, however, and since we do not yet have an application for this result, we have chosen to omit the proof.
\end{remark}

\section{The geometric coenergy function}
\label{sec_coenergy}

Recall that a $\bbZ$-valued function on a tensor product of two Kirillov--Reshetikhin crystals is a coenergy function if it is invariant under the crystal operators $\te_1, \ldots, \te_{n-1}$, and it interacts with $\te_0$ in a prescribed way (Definition \ref{defn_comb_coenergy}). In this section, we ``lift'' the combinatorial definition to define a notion of geometric coenergy function. We show that a certain minor of the product matrix $g(M|s)g(N|t)$ defines a geometric coenergy function on $\X{k_1} \times \X{k_2}$, and that this function tropicalizes to the coenergy function $\tw{E}$ defined in \S \ref{sec_comb_coenergy}.

\begin{defn}
\label{defn_geom_coenergy}
A rational function $H : \X{k_1} \times \X{k_2} \rightarrow \bbC$ is a {\em geometric coenergy function} if $H \circ e_i^c = H$ for $i = 1, \ldots, n-1$, and $H$ interacts with $e_0^c$ as follows: if $(u,v) \in \X{k_1} \times \X{k_2}$ and $R(u,v) = (v',u')$, then
\begin{equation}
\label{eq_geom_coenergy}
H(e_0^c(u,v)) = H(u,v)\left(\dfrac{\ve_0(u) + c^{-1}\vp_0(v)}{\ve_0(u) + \vp_0(v)}\right)\left(\dfrac{c\ve_0(v') + \vp_0(u')}{\ve_0(v') + \vp_0(u')}\right).
\end{equation}
\end{defn}

\begin{lem}
\label{lem_trop_coenergy}
If $H$ is a positive geometric coenergy function on $\X{k_1} \times \X{k_2}$, then the piecewise-linear function $\wh{H} := \Trop(\Theta H)$, when restricted to $B^{n-k_1, L_1} \otimes B^{n-k_2,L_2} \subset \tT{n-k_1} \times \tT{n-k_2}$, is a coenergy function.
\end{lem}

\begin{proof}
Clearly $\wh{H} \circ \, \tw{e}_i = \wh{H}$ for $i = 1, \ldots, n-1$. If $a \otimes b \in B^{n-k_1,L_1} \otimes B^{n-k_2,L_2}$ and $\tw{R}(a \otimes b) = (b' \otimes a')$, then by tropicalizing \eqref{eq_geom_coenergy} and using Theorems \ref{thm_recover_crystals} and \ref{thm_recover_comb_R} (plus the identity $\max(c,d) = -\min(-c,-d)$), we obtain
\begin{align*}
\wh{H}(\tw{e}_0(a \otimes b)) &= \wh{H}(a \otimes b) + \max(\tep_0(a), \tph_0(b)) - \max(\tep_0(a), \tph_0(b) + 1) \\
&\quad\quad\quad\quad\quad \: + \max(\tep_0(b'), \tph_0(a')) - \max(\tep_0(b') - 1, \tph_0(a')) \\
&= \wh{H}(a \otimes b) + \begin{cases}
0 & \text{ if } \tep_0(a) > \tph_0(b) \\
-1 & \text{ if } \tep_0(a) \leq \tph_0(b) \\
\end{cases}
+ \begin{cases}
1 & \text{ if } \tep_0(b') > \tph_0(a') \\
0 & \text{ if } \tep_0(b') \leq \tph_0(a')
\end{cases}.
\end{align*}
This shows that $\wh{H}$ satisfies \eqref{eq_coenergy}, as needed.
\end{proof}

\begin{defn}
Define $E : \X{k_1} \times \X{k_2} \rightarrow \bbC$ by
\[
E(u,v) = \Delta_{[n-k+1,n], [k]}(g(u)g(v)),
\]
where $k = \min(k_1,k_2)$.
\end{defn}

Note that for $(u,v) \in \X{k_1} \times \X{k_2}$, the last $k_1$ rows of $g(u)$ and the first $k_2$ columns of $g(v)$ are independent of $\lp$, so by the Cauchy--Binet formula, $E$ is indeed complex-valued. In fact, the Cauchy--Binet formula gives a simple expression for $E$ in terms of Pl\"{u}cker coordinates. Recall the notation $Q^J_s(M) := P_{w_0(J)}(S_s(M))$.

\begin{lem}
\label{lem_E_formula}
If $(M|s, N|t) \in \X{k_1} \times \X{k_2}$, then
\begin{equation}
\label{eq_E_formula}
E(M|s, N|t) = \begin{cases}
\ds \sum_{I \in {[k_1-k_2+1,n] \choose k_2}} \dfrac{Q^{I'}_s(M)}{Q^{[k_1]}_s(M)} \dfrac{P_I(N)}{P_{[n-k_2+1,n]}(N)} & \text{ if } k_1 \geq k_2 \\
\ds \sum_{I \in {[n-k_2+k_1] \choose k_1}} \dfrac{Q^I_s(M)}{Q^{[k_1]}_s(M)} \dfrac{P_{I''}(N)}{P_{[n-k_2+1,n]}(N)} & \text{ if } k_1 \leq k_2
\end{cases}
\end{equation}
where $I' = [k_1-k_2] \cup I$, and $I'' = I \cup [n-k_2+k_1+1,n]$.
\end{lem}

\begin{proof}
We assume $k_1 \geq k_2$ (the case $k_1 \leq k_2$ is similar). By Cauchy--Binet, we have
\[
E(M|s, N|t) = \sum_{|I| = k_2} \Delta_{[n-k_2+1,n], I}(g(M|s)) \Delta_{I, [k_2]}(g(N|t)).
\]
The bottom-left $k_1 \times k_1$ submatrix of $g(M|s)$ is upper uni-triangular, so
\[
\Delta_{[n-k_2+1,n], I}(g(M|s)) = \begin{cases}
\Delta_{[n-k_1+1,n], [k_1-k_2] \cup I}(g(M|s)) & \text{ if } I \subset [k_1-k_2+1,n] \\
0 & \text{ otherwise.}
\end{cases}
\]
This, together with \eqref{eq_P_Q}, proves the $k_1 \geq k_2$ case of \eqref{eq_E_formula}.
\end{proof}

\begin{prop}
\label{prop_E_is_geom_coenergy}
$E : \X{k_1} \times \X{k_2} \rightarrow \bbC$ is a geometric coenergy function.
\end{prop}

\begin{proof}
Suppose $u = M|s \in \X{k_1}$ and $v = N|t \in \X{k_2}$. Set $B = g(u)g(v)$, $R(u,v) = (v',u')$, and $k = \min(k_1, k_2)$. Since $e_i^c$ commutes with the unipotent crystal map $g$, we have
\[
E(e_i^c(u,v)) = \Delta_{[n-k+1,n],[k]}(e_i^c(B)).
\]
By \eqref{eq_U_action_B-}, the folded matrix $e_i^c(B)$ is obtained from the folded matrix $B$ by adding a multiple of row $i+1$ to row $i$, and a multiple of column $i$ to column $i+1$ (mod $n$). If $i \in [n-1]$, then these row and column operations do not change the determinant of a bottom-left justified submatrix, so $E$ is invariant under $e_i^c$.

Now consider $e_0^c$. By \eqref{eq_U_action_B-}, we have
\begin{equation}
\label{eq_E_0}
E(e_0^c(u,v)) = \Delta_{[n-k+1,n],[k]}\left(x_0\left(\lp^{-1}\dfrac{c-1}{\vp_0(u,v)}\right) \cdot B \cdot x_0\left(\lp^{-1}\dfrac{c^{-1}-1}{\ve_0(u,v)}\right)\right)
\end{equation}
where $x_0(z)$ is the $n \times n$ matrix with $1$'s on the diagonal and $z$ in position $(n,1)$. Suppose $k = k_2$. The left-hand side of \eqref{eq_E_0} does not depend on $\lp$, so we may substitute $\lp = (-1)^{k-1} t$ into the right-hand side and obtain
\[
E(e_0^c(u,v)) = \Delta_{[n-k+1,n],[k]}\left(x_0\left(\dfrac{(-1)^{k-1}}{t}\dfrac{c-1}{\vp_0(u,v)}\right) \cdot B_t \cdot x_0\left(\dfrac{(-1)^{k-1}}{t}\dfrac{c^{-1}-1}{\ve_0(u,v)}\right)\right),
\]
where $B_t = B|_{\lp = (-1)^{k-1}t}$. By multi-linearity of the determinant (or by Cauchy--Binet), we have
\begin{multline}
\label{eq_four_minors}
E(e_0^c(u,v)) = \Delta_{[n-k+1,n],[k]}(B_t) + \dfrac{1}{t}\dfrac{c-1}{\vp_0(u,v)}\Delta_{\{1\} \cup [n-k+1,n-1], [k]}(B_t) \\
+ \dfrac{1}{t}\dfrac{c^{-1}-1}{\ve_0(u,v)} \Delta_{[n-k+1,n],[2,k] \cup \{n\}}(B_t) + \dfrac{1}{t^2}\dfrac{(c-1)(c^{-1}-1)}{\vp_0(u,v) \ve_0(u,v)} \Delta_{\{1\} \cup [n-k+1,n-1], [2,k] \cup \{n\}}(B_t).
\end{multline}

Restrict to the open set\footnote{We are trying to prove the identity of rational maps \eqref{eq_geom_coenergy}, so we may restrict to open subsets.} where $(-1)^{k_1+k_2}s \neq t$, so that $g(M|s)|_{\lp = (-1)^{k-1}t}$ is invertible by Proposition \ref{prop_g_props}\eqref{itm:det_g}, and $B_t$ has rank $k$ by Corollary \ref{cor_linear_alg}\eqref{itm:preserves_dim}. In a rank $k$ matrix, any set of $k$ columns which are linearly independent span the same subspace, so we have
\begin{equation}
\label{eq_fourth_minor_ratio}
\dfrac{\Delta_{\{1\} \cup [n-k+1,n-1],[2,k] \cup \{n\}}(B_t)}{\Delta_{[n-k+1,n],[2,k] \cup \{n\}}(B_t)} = \dfrac{\Delta_{\{1\} \cup [n-k+1,n-1],[k]}(B_t)}{\Delta_{[n-k+1,n],[k]}(B_t)}
\end{equation}
on the open set where both denominators are nonzero. Using \eqref{eq_fourth_minor_ratio} and the fact that $E(u,v) = \Delta_{[n-k+1,n], [k]}(B_t)$, we may rewrite \eqref{eq_four_minors} as
\begin{equation}
\label{eq_E_factored}
E(e_0^c(u,v)) = E(u,v)(1 + z_1)(1 + z_2)
\end{equation}
where
\begin{equation}
\label{eq_z_1_z_2}
z_1 = \dfrac{1}{t}\dfrac{c-1}{\vp_0(u,v)} \dfrac{\Delta_{\{1\} \cup [n-k+1,n-1], [k]}(B_t)}{\Delta_{[n-k+1,n], [k]}(B_t)}, \quad\quad z_2 = \dfrac{1}{t}\dfrac{c^{-1}-1}{\ve_0(u,v)} \dfrac{\Delta_{[n-k+1,n], [2,k] \cup \{n\}}(B_t)}{\Delta_{[n-k+1,n], [k]}(B_t)}.
\end{equation}

Now we compute
\begin{align*}
\vp_0(u,v) = \vp_0(v',u') = \vp_0(v')\dfrac{\ve_0(v') + \vp_0(u')}{\ve_0(v')} = \dfrac{1}{t}\dfrac{\Delta_{[n-k+1,n-1] \cup \{1\}, [k]}(B_t)}{\Delta_{[n-k+1,n], [k]}(B_t)}\dfrac{\ve_0(v') + \vp_0(u')}{\ve_0(v')},
\end{align*}
where the first equality comes from Theorem \ref{thm_master}\eqref{itm:R=isomorphism}, the second equality comes from Definition/Proposition \ref{defn_prop_prod_geom}, and the final equality is the formula for $\vp_0$ on $\X{k}$, together with \eqref{eq_P_Q_N'_M'}. Similarly, using Definition/Proposition \ref{defn_prop_prod_geom}, the identity $\ve_i = \vp_{n-i} \circ \, S$ (this is \cite[Cor. 7.4(3)]{F1}), the formula for $\vp_0$ on $\X{k}$, and \eqref{eq_P_Q_M_N}, we compute
\begin{align*}
\ve_0(u,v) = \ve_0(v) \dfrac{\ve_0(u) + \vp_0(v)}{\vp_0(v)} &= \vp_0(S(v)) \dfrac{\ve_0(u) + \vp_0(v)}{\vp_0(v)} \\
&= \dfrac{1}{t} \dfrac{\Delta_{[n-k+1,n], [2,k] \cup \{n\}}(B_t)}{\Delta_{[n-k+1,n], [k]}(B_t)} \dfrac{\ve_0(u) + \vp_0(v)}{\vp_0(v)}.
\end{align*}
Substituting these expressions into \eqref{eq_z_1_z_2}, we get
\begin{equation}
\label{eq_z_1_z_2_final}
z_1 = \dfrac{(c-1) \ve_0(v')}{\ve_0(v') + \vp_0(u')}, \quad\quad z_2 = \dfrac{(c^{-1}-1)\vp_0(v)}{\ve_0(u) + \vp_0(v)},
\end{equation}
and then \eqref{eq_geom_coenergy} is obtained by substituting \eqref{eq_z_1_z_2_final} into \eqref{eq_E_factored}.

The $k = k_1$ case is dealt with similarly, using the substitution $\lp = (-1)^{k_1-1} s$ instead of $\lp = (-1)^{k_2-1}t$.
\end{proof}

The function $E$ is positive by Lemma \ref{lem_E_formula} and the positivity of $S$, so we may define $\wh{E} = \Trop(\Theta E) : \tT{n-k_1} \times \tT{n-k_2} \rightarrow \bbZ$. Recall the coenergy function $\tw{E}$ introduced in \S \ref{sec_comb_coenergy}.

\begin{thm}
\label{thm_recover_E}
The restriction of $\wh{E}$ to $B^{n-k_1, L_1} \otimes B^{n-k_2, L_2}$ is equal to $\tw{E}$.
\end{thm}

\begin{proof}
By Proposition \ref{prop_E_is_geom_coenergy} and Lemma \ref{lem_trop_coenergy}, $\wh{E}$ is a coenergy function on $B^{n-k_1,L_1} \otimes B^{n-k_2, L_2}$. By Proposition \ref{prop_coenergy}(1), the coenergy function on such crystals is unique up to a global additive constant, so it suffices to show that $\wh{E}$ and $\tw{E}$ agree on a single element of $B^{n-k_1,L_1} \otimes B^{n-k_2,L_2}$.

Assume $k_1 \leq k_2$ (the other case is basically the same). If $(x,y) \in \bT{n-k_1} \times \bT{n-k_2}$, then by Lemma \ref{lem_E_formula} and \eqref{eq_P_Q}, we have
\begin{equation*}
\Theta E(x,y) = \sum_{I \in {[n-k_2+k_1] \choose k_1}} q_I(x) p_I(y),
\end{equation*}
where $q_I(x) = \Delta_{[n-k_1+1,n], I}(g(\Theta_{k_1}(x)))$ and $p_I(y) = \Delta_{I \cup [n-k_2+k_1+1,n], [k_2]}(g(\Theta_{k_2}(y)))$. Since the last $k_1$ rows of $g(\Theta_{k_1}(x))$ and the first $k_2$ columns of $g(\Theta_{k_2}(y))$ are independent of $\lp$, Lemma \ref{lem_g_0} implies that
\[
q_I(x) = \Delta_{[n-k_1+1,n], I}(\Phi_{n-k_1}(x)) \quad\quad \text{ and } \quad\quad p_I(x) = \Delta_{I \cup [n-k_2+k_1+1,n], [k_2]}(\Phi_{n-k_2}(y)).
\]
The rational functions $q_I$ and $p_I$ are positive by Lemma \ref{lem_Phi_minors}; letting $\wh{q}_I, \wh{p}_I$ denote their respective tropicalizations, we have
\begin{equation}
\label{eq_trop_E_formula}
\wh{E}(a \otimes b) = \min_{I \in {[n-k_2+k_1] \choose k_1}} (\wh{q}_I(a) + \wh{p}_I(b)).
\end{equation}

The bottom $k_2 \times k_2$ submatrix of $\Phi_{n-k_2}(y)$ is upper uni-triangular, so $p_{[n-k_2+1,n-k_2+k_1]}(y)$ is identically equal to $1$, and thus
\begin{equation}
\label{eq_p_I=0}
\wh{p}_{[n-k_2+1,n-k_2+k_1]}(b) = 0
\end{equation}
for all $b$. Let $a_0$ be the classical highest weight element of $B^{n-k_1, L_1}$. By Lemma \ref{lem_highest_wt_zero}, we have $\wh{q}_I(a_0) = 0$ for all $I \in {[n] \choose k_1}$. Together with \eqref{eq_trop_E_formula} and \eqref{eq_p_I=0}, this implies that $\wh{E}(a_0 \otimes b) = 0$ for all $b \in B^{n-k_2, L_2}$. By \eqref{eq_E_highest_wt}, $\tw{E}(a_0 \otimes b) = 0$ for all $b \in B^{n-k_2, L_2}$, so we are done.
\end{proof}

\section{One-row tableaux}
\label{sec_one}

In this section, we give a more explicit description of the geometric $R$-matrix and geometric coenergy function on $\bT{\ell} \times \bT{k}$ in the case $\ell = 1$. When $k = 1$ as well, we recover the one-row birational $R$-matrix discussed in the Introduction. In \S \ref{sec_1_2}, we verify that the tropicalizations $\wh{R}$ and $\wh{E}$ agree with the combinatorial maps $\tw{R}$ and $\tw{E}$ in a small example.

\subsection{Formulas for $\Theta R$ and $\Theta E$ in the case $\ell = 1$}

Let $X = (X_{11}, X_{12}, \ldots, X_{1,n-1},s) \in \bb{T}_1$ be a rational $1$-rectangle, and define $x_1, \ldots, x_n$ by $x_j = X_{1j}/X_{1,j-1}$ (where $X_{10} := 1$ and $X_{1n} := s$). Let $Y = (Y_{ij}, t) \in \bT{k}$ be a rational $k$-rectangle. Suppose
\[
\Theta R(X, Y) = ((Y'_{ij}, t), (X'_{1j},s)),
\]
and define $x'_j$ from $(X'_{1j}, s)$ as above. We will work through the various definitions from earlier sections to obtain formulas for $Y'_{ij}$, $x'_j$, and $\Theta E(X,Y)$ in terms of the inputs $x_j, Y_{ij},$ and $t$.

Set $N|t = \Theta_{n-k}(Y)$, $N'|t = \Theta_{n-k}(Y')$, $A = g(\Theta_{n-1}(X))g(\Theta_{n-k}(Y))$, and $A_t = A|_{\lp = (-1)^{n-k-1}t}$. For $I \in {[n] \choose n-k}$, define
\begin{equation*}
\tau_I = \tau_I(X,Y) = \Delta_{I,[n-k]}(A_t) \dfrac{P_{[k+1,n]}(N)}{P_I(N)}.
\end{equation*}
By \eqref{eq_P_Q_N'_M'} and Lemma \ref{lem_GT_pl} (applied to both $Y_{ij}$ and $Y'_{ij}$), we have
\begin{equation}
\label{eq_formula_Y'_ij}
Y'_{ij} = Y_{ij} \dfrac{\tau_{[i,j] \cup [k+j-i+2,n]}}{\tau_{[i+1,j] \cup [k+j-i+1,n]}},
\end{equation}
so we are led to the study of the quantities $\tau_I$. By the Cauchy--Binet formula and Proposition \ref{prop_g_props}\eqref{itm:pi_g},
\begin{equation}
\label{eq_minors_A_t}
\tau_I = \sum_J \Delta_{I,J}(C|_{\lp = (-1)^{n-k-1}t}) \dfrac{P_J(N)}{P_I(N)},
\end{equation}
where $C = g(\Theta_{n-1}(X))$. In \cite[Lem. 4.13]{F1}, we showed how to express $P_J(N)$ in terms of the $Y_{ij}$ by summing over a tableau-style object, so we regard these Pl\"{u}cker coordinates as well-understood. It remains to explicitly compute the minors of the matrix $C$. Unraveling the definitions, one sees that $C$ is the matrix \eqref{eq_g_one_row}; that is, it has $x_1, \ldots, x_n$ on the main diagonal, 1's just beneath the main diagonal, $\lp$ in the top-right corner, and zeroes elsewhere.

\begin{lem}
\label{lem_one_row_minors}
Let $C = g(\Theta_{n-1}(X))$, and let $I = \{i_1 < \ldots < i_r\}$ be an $r$-subset of $[n]$, with $r \leq n-1$. For $\epsilon = (\epsilon_1, \ldots, \epsilon_r) \in \{0,1\}^r$, let
\[
I - \epsilon = \{i_1 - \epsilon_1, \ldots, i_r - \epsilon_r\} \subset [n],
\]
where if $i_1 = \epsilon_1 = 1$, we take $i_1 - \epsilon_1 = n$. If $I - \epsilon$ has $r$ elements, then we have
\begin{equation}
\label{eq_I_I-epsilon}
\Delta_{I,I-\epsilon}(C) = \begin{cases}
\ds (-1)^{r-1} \lp \prod_{s \, | \, \epsilon_s = 0} x_{i_s} & \text{ if } i_1 = \epsilon_1 = 1 \\
\ds \prod_{s \, | \, \epsilon_s = 0} x_{i_s} & \text{ otherwise}.
\end{cases}
\end{equation}
If $J \in {[n] \choose r}$ is not of the form $I -\epsilon$ for $\epsilon \in \{0,1\}^r$, then $\Delta_{I,J}(C) = 0$.
\end{lem}

Note that by expanding along the last column of $C$, we have $\det(C) = (-1)^{n+1} \lp + \prod_{j=1}^n x_j$, so the restriction $r \leq n-1$ is necessary.

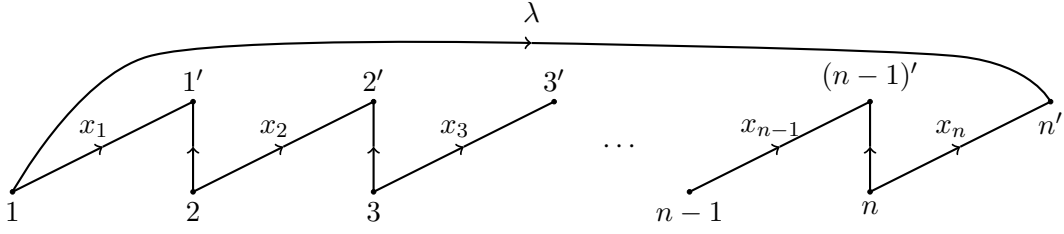
\begin{figure}
\begin{center}
\begin{tikzpicture}[scale=0.6]

\draw[thick,->] (0,2) -- (2,3);
\draw[thick] (2,3) -- (4,4);
\draw[thick,->] (4,2) -- (6,3);
\draw[thick] (6,3) -- (8,4);
\draw[thick,->] (8,2) -- (10,3);
\draw[thick] (10,3) -- (12,4);
\draw[thick,->] (15,2) -- (17,3);
\draw[thick] (17,3) -- (19,4);
\draw[thick,->] (19,2) -- (21,3);
\draw[thick] (21,3) -- (23,4);

\draw[thick,->] (4,2) -- (4,3);
\draw[thick] (4,3) -- (4,4);
\draw[thick,->] (8,2) -- (8,3);
\draw[thick] (8,3) -- (8,4);
\draw[thick,->] (19,2) -- (19,3);
\draw[thick] (19,3) -- (19,4);

\draw (1.8,3.4) node {$x_1$};
\draw (5.8,3.4) node {$x_2$};
\draw (9.8,3.4) node {$x_3$};
\draw (16.8,3.4) node {$x_{n-1}$};
\draw (20.8,3.4) node {$x_n$};

\draw (13.5,3) node {$\cdots$};

\filldraw (0,2) circle [radius=.05] node[below] {$1$};
\filldraw (4,4) circle [radius=.05] node[above] {$1'$};
\filldraw (4,2) circle [radius=.05] node[below] {$2$};
\filldraw (8,4) circle [radius=.05] node[above] {$2'$};
\filldraw (8,2) circle [radius=.05] node[below] {$3$};
\filldraw (12,4) circle [radius=.05] node[above] {$3'$};
\filldraw (15,2) circle [radius=.05] node[below] {$n-1$};
\filldraw (19,4) circle [radius=.05] node[above] {$(n-1)'$};
\filldraw (19,2) circle [radius=.05] node[below] {$n$};
\filldraw (23,4) circle [radius=.05] node[below] {$n'$};

\draw[thick,->] plot [smooth] coordinates {(0,2) (3,5) (11.5,5.3)};
\draw[thick] plot [smooth] coordinates {(11.5,5.3) (21,5) (23,4)};
\draw (11.5,6) node{$\lp$};

\end{tikzpicture}
\end{center}

\caption{A network representation of the matrix $g(\Theta_{n-1}(X))$. Vertical edges have weight 1.}
\label{fig_one_row_network}
\end{figure}

\begin{proof}

We use planar networks and the Lindstr\"{o}m Lemma, as discussed in the Appendix. Observe that $C$ is the matrix associated to the planar network in Figure \ref{fig_one_row_network}. In this network, there are two edges coming out of each source $i$: an edge to sink $i'$, and an edge to sink $(i-1)'$ (mod $n$). Thus, if there is a vertex-disjoint family of paths from the sources in $I$ to the sinks in $J$, then $J = I - \epsilon$ for some $\epsilon \in \{0,1\}^r$; if $J$ is not of this form, then $\Delta_{I,J}(C)$ is zero by the Lindstr\"{o}m Lemma.

We claim that for any $r$-subset $J$, there is at most one vertex-disjoint family of paths from $I$ to $J$. To see this, note that the underlying (undirected) graph of the network is a cycle of length $2n$, and a vertex-disjoint family of paths from $I$ to $J$ is a perfect matching in the subgraph induced by the vertices in $I$ and $J$. Since $r \leq n-1$, the subgraph induced by the vertices in $I$ and $J$ is a forest, and there is at most one perfect matching in any forest.

Suppose $I - \epsilon$ has $r$ elements. In this case, let $p_s$ be the path connecting $i_s$ and $(i_s - \epsilon_s)'$, for $s = 1, \ldots, r$. The family of paths $(p_s)$ is clearly vertex-disjoint, and it has weight
\[
\begin{cases}
\ds \lp \prod_{s \, | \, \epsilon_s = 0} x_{i_s} & \text{ if } i_1 = \epsilon_1 = 1 \\
\ds \prod_{s \, | \, \epsilon_s = 0} x_{i_s} & \text{ otherwise}.
\end{cases}
\]
The permutation associated to this family has sign $(-1)^{r-1}$ if $i_1 = \epsilon_1 = 1$, and is the identity otherwise, so \eqref{eq_I_I-epsilon} follows from the Lindstr\"{o}m Lemma.
\end{proof}

In light of Lemma \ref{lem_one_row_minors}, \eqref{eq_minors_A_t} becomes
\begin{equation}
\label{eq_explicit_minors_A_t}
\tau_I = \sum_{\epsilon} t^{\delta_{i_1,1}\delta_{\epsilon_1,1}} \cdot \prod_{s \, | \, \epsilon_s = 0} x_{i_s} \cdot \dfrac{P_{I-\epsilon}(N)}{P_I(N)},
\end{equation}
where the sum is over $\epsilon \in \{0,1\}^{n-k}$ such that $I - \epsilon$ has $n-k$ elements. For example, if $n = 7$ and $k = 4$, then writing $P_J$ for $P_J(N)$, we have
\[
\tau_{145} = \dfrac{x_1x_4x_5 P_{145} + x_1x_5 P_{135} + x_1 P_{134} + tx_4x_5 P_{457} + tx_5 P_{357} + t P_{347}}{P_{145}}.
\]
Combining \eqref{eq_formula_Y'_ij} and \eqref{eq_explicit_minors_A_t}, we have a reasonably explicit formula for the $Y'_{ij}$.

Now we turn to the $x'_j$. For $j \in \Zn$, define
\begin{equation*}
\kappa_j = \kappa_j(X,Y) = \tau_{[j+k,j+n-1]}(X,Y).
\end{equation*}

\begin{prop}
\label{prop_formula_x'}
We have
\[
x'_j = x_j \dfrac{\kappa_j}{\kappa_{j+1}}.
\]
Furthermore, we have the formula
\begin{equation}
\label{eq_kappa_j}
\kappa_j = \sum_{s = 0}^{n-k} x_{j+k+s}x_{j+k+s+1} \cdots x_{j+n-1} t^{a_{j,s,k}} \dfrac{P_{[j+k-1,j+n-1] \setminus \{j+k+s-1\}}(N)}{P_{[j+k,j+n-1]}(N)}
\end{equation}
where
\[
a_{j,s,k} = \begin{cases}
1 & \text{ if } n+2-s \leq j+k \leq n+1 \\
0 & \text{ otherwise}.
\end{cases}
\]
Each subscript of $\kappa$ and $x$ is interpreted mod $n$.
\end{prop}

\begin{proof}
By Theorem \ref{thm_hard}, we have the matrix equation
\begin{equation}
\label{g_equation}
g(\Theta_{n-1}(X))g(N|t) = g(N'|t)g(\Theta_{n-1}(X')).
\end{equation}
As mentioned above, the diagonal entries of $g(\Theta_{n-1}(X))$ are $x_1, \ldots, x_n$, so by equating the constant coefficients of the diagonal entries of both sides of \eqref{g_equation} and using the definition of $g(N|t)$ (plus Convention \ref{conv_pluc}), we obtain
\begin{equation*}
x_j \dfrac{P_{[j+k+1,j+n]}(N)}{P_{[j+k,j+n-1]}(N)} = x'_j \dfrac{P_{[j+k+1,j+n]}(N')}{P_{[j+k,j+n-1]}(N')} = x'_j \dfrac{\Delta_{[j+k+1,j+n],[n-k]}(A_t)}{\Delta_{[j+k,j+n-1],[n-k]}(A_t)}.
\end{equation*}
This shows that $x'_j = x_j \dfrac{\kappa_j}{\kappa_{j+1}}$.

By \eqref{eq_explicit_minors_A_t}, we have
\begin{equation*}
\kappa_j = \sum_{\epsilon} t^{\delta_{i_1,1}\delta_{\epsilon_1,1}} \cdot \prod_{s \, | \, \epsilon_s = 0} x_{i_s} \dfrac{P_{[j+k,j+n-1] - \epsilon}(N)}{P_{[j+k,j+n-1]}(N)}
\end{equation*}
(to compute $[j+k,j+n-1] - \epsilon$, first identify $[j+k,j+n-1]$ with a subset $\{i_1 < \ldots < i_{n-k}\}$ of $[n]$ by reducing mod $n$, and then subtract $\epsilon_s$ from the $s^{th}$ smallest element of this subset). There are $n-k+1$ choices of $\epsilon$ such that $[j+k,j+n-1] - \epsilon$ has $n-k$ elements, and one may easily verify that each of these choices gives a term from the right-hand side of \eqref{eq_kappa_j}.
\end{proof}

\begin{cor}
\label{cor_E=kappa_1}
The geometric coenergy function $\Theta E : \bT{1} \times \bT{k} \rightarrow \bbC$ is given by
\[
\Theta E(X,Y) = \kappa_1(X,Y) = \sum_{s = 0}^{n-k} x_{k+s+1}x_{k+s+2} \cdots x_n Y_{k,k+s-1},
\]
where $Y_{k,k-1} := 1$.
\end{cor}

\begin{proof}
By definition, $\Theta E(X,Y) = \Delta_{[k+1,n],[n-k]}(A)$, and since this minor is independent of $\lp$, it is equal to $\kappa_1(X,Y)$. The explicit formula for $\kappa_1(X,Y)$ follows from \eqref{eq_kappa_j} and Lemma \ref{lem_GT_pl}.
\end{proof}

\subsection{The case $\ell = k = 1$}

Now we specialize further to the case $k = 1$. Let $Y =  (Y_{11}, \ldots, Y_{1,n-1}, t)$, and define $y_j = Y_{1j}/Y_{1,j-1}$, where $Y_{10} := 1$ and $Y_{1n} := t$. As above, let $N|t = \Theta_{n-1}(Y)$. By Proposition \ref{prop_g_props}\eqref{itm:pi_g} and Lemma \ref{lem_one_row_minors}, we have
\[
\dfrac{P_{[n] \setminus \{a\}}(N)}{P_{[n] \setminus \{b\}}(N)} = \dfrac{\Delta_{[n] \setminus \{a\},[n-1]}(g(\Theta_{n-1}(Y))}{\Delta_{[n] \setminus \{b\}, [n-1]}(g(\Theta_{n-1}(Y))} = \dfrac{y_1 \cdots y_{a-1}}{y_1 \cdots y_{b-1}} = \begin{cases}
y_b \cdots y_{a-1} & \text{ if } b \leq a \\
(y_a \cdots y_{b-1})^{-1} & \text{ if } a \leq b
\end{cases}
\]
for $a, b \in [n]$. Setting $k = 1$ in \eqref{eq_kappa_j} and using $t = y_1 \cdots y_n$, we obtain
\begin{align*}
\kappa_j &= \sum_{s = 0}^{n-j} x_{j+s+1}x_{j+s+2} \cdots x_{j+n-1} \dfrac{P_{[n] \setminus \{j+s\}}(N)}{P_{[n] \setminus \{j\}}(N)} + t \sum_{s = n-j+1}^{n-1} x_{j+s+1} x_{j+s+2} \cdots x_{j+n-1} \dfrac{P_{[n] \setminus \{j+s-n\}}(N)}{P_{[n] \setminus \{j\}}(N)} \\
&= \sum_{s=0}^{n-1} y_j y_{j+1} \cdots y_{j+s-1} x_{j+s+1} x_{j+s+2} \cdots x_{j+n-1},
\end{align*}
where as above, each subscript of $x$ and $y$ is interpreted mod $n$.

\begin{prop}
\label{prop_one_one_case}
The map $\Theta R : ((x_1, \ldots, x_n), (y_1, \ldots, y_n)) \mapsto ((y'_1, \ldots, y'_n), (x'_1, \ldots, x'_n))$ is given by
\[
y'_j = y_j \dfrac{\kappa_{j+1}}{\kappa_j}, \quad\quad x'_j = x_j \dfrac{\kappa_j}{\kappa_{j+1}}, \quad\quad \kappa_j = \sum_{s=0}^{n-1} y_j y_{j+1} \cdots y_{j+s-1} x_{j+s+1} x_{j+s+2} \cdots x_{j+n-1},
\]
where subscripts are interpreted mod $n$.
\end{prop}

\begin{proof}
By the preceding discussion, we have $x'_j = x_j \dfrac{\kappa_j}{\kappa_{j+1}}$. Arguing as in the proof of Proposition \ref{prop_formula_x'}, we have $x_jy_j = y'_j x'_j$, so $y'_j = y_j \dfrac{\kappa_{j+1}}{\kappa_j}$.
\end{proof}

Thus, in the one-row by one-row case, our geometric $R$-matrix agrees with the map described in Proposition \ref{prop_geom_one_row}.

\subsection{The case $n=4, \ell = 1, k= 2$}
\label{sec_1_2}

Suppose $X = (X_{11}, X_{12}, X_{13}, s) \in \bT{1}$, $Y = (Y_{11}, Y_{12}, Y_{22}, Y_{23}, t) \in \bT{2}$, and $(Y',X') = \Theta R(X,Y)$. Define
\[
x_1 = X_{11} \quad\quad x_2 = X_{12}/X_{11} \quad\quad x_3 = X_{13}/X_{12} \quad\quad x_4 = s/X_{13},
\]
and define $x'_j$ analogously. Define also
\[
\begin{array}{ccccc}
y_{11} = Y_{11} && y_{12} = Y_{12}/Y_{11} && y_{13} = t/Y_{13} \\
y_{22} = Y_{22} && y_{23} = Y_{23}/Y_{22} && y_{14} = t/Y_{23}.
\end{array}
\]
Note that $t = y_{11}y_{12}y_{13} = y_{22} y_{23} y_{24}$. Let $N|t = \Theta_2(Y) \in \Y{2}{4}$. Using the definition of $\Theta_k$, one computes that $N$ is the column span of the matrix
\[
\left(
\begin{array}{cc}
y_{11} & 0 \\
y_{22} & y_{12}y_{22} \\
1 & y_{12} + y_{23} \\
0 & 1
\end{array}
\right).
\]
Set $P_J = P_J(N)$. By \eqref{eq_formula_Y'_ij}, Proposition \ref{prop_formula_x'}, and \eqref{eq_explicit_minors_A_t}, we have
\begin{equation}
\label{eq_Y'_2_4}
Y'_{11} = Y_{11} \dfrac{\tau_{14}}{\tau_{34}} \quad\quad Y'_{12} = Y_{12} \dfrac{\tau_{12}}{\tau_{24}} \quad\quad Y'_{22} = Y_{22} \dfrac{\tau_{24}}{\tau_{34}} \quad\quad Y'_{23} = Y_{23} \dfrac{\tau_{23}}{\tau_{34}},
\end{equation}
\begin{equation}
\label{eq_x'_2_4}
x'_1 = x_1 \dfrac{\kappa_1}{\kappa_2} \quad\quad x'_2 = x_2 \dfrac{\kappa_2}{\kappa_3} \quad\quad x'_3 = x_3 \dfrac{\kappa_3}{\kappa_4} \quad\quad x'_4 = x_4 \dfrac{\kappa_4}{\kappa_1},
\end{equation}
where
\begin{align*}
\kappa_1 = \tau_{34} &= \dfrac{x_3x_4 P_{34} + x_4P_{24} + P_{23}}{P_{34}} = x_3x_4 + x_4 y_{22} + y_{22}y_{23} \\
\kappa_2 = \tau_{14} &= \dfrac{x_1x_4 P_{14} + x_1 P_{13} + t P_{34}}{P_{14}} = x_1x_4 + x_1(y_{12} + y_{23}) + y_{12}y_{13} \\
\kappa_3 = \tau_{12} &= \dfrac{x_1x_2 P_{12} + t x_2 P_{24} + t P_{14}}{P_{12}} = x_1x_2 + x_2 y_{13} + \dfrac{y_{11}y_{13}}{y_{22}} \\
\kappa_4 = \tau_{23} &= \dfrac{x_2x_3 P_{23} + x_3 P_{13} + P_{12}}{P_{23}} = x_2x_3 + x_3\dfrac{y_{11}}{y_{22}y_{23}}(y_{12} + y_{23}) + \dfrac{y_{11}y_{12}}{y_{23}} \\
\tau_{24} &= \dfrac{x_2x_4 P_{24} + x_2P_{23} + x_4P_{14} + P_{13}}{P_{24}} = x_2x_4 + x_2 y_{23} + \dfrac{y_{11}}{y_{22}} (x_4 + y_{12} + y_{23}).
\end{align*}

By tropicalizing these formulas, one obtains piecewise-linear formulas for the combinatorial $R$-matrix on $B^1 \otimes B^2$. Specifically, let $A = (A_{11}, A_{12}, A_{13}, L_1)$ be a 1-rectangle, let $B = (B_{11}, B_{12}, B_{22}, B_{23}, L_2)$ be a 2-rectangle, and let $B' \otimes A' = \wh{R}(A \otimes B)$. Define
\[
a_1 = A_{11} \quad\quad a_2 = A_{12} - A_{11} \quad\quad a_3 = A_{13} - A_{12} \quad\quad a_4 = L_1 - A_{13},
\]
\[
\begin{array}{ccccc}
b_{11} = B_{11} && b_{12} = B_{12} - B_{11} && b_{13} = L_2 - B_{12} \\
b_{22} = B_{22} && b_{23} = B_{23} - B_{22} && b_{24} = L_2 - B_{23}, \\
\end{array}
\]
so that $a_j$ is the number of $j$'s in the one-row tableau corresponding to $A$, and $b_{ij}$ is the number of $j$'s in the $i^{th}$ row of the two-row tableau corresponding to $B$. Define $a'_j, b'_{ij}$ analogously.

For $I \in {[4] \choose 2}$, let $\tw{\tau}_I$ be the tropicalization of $\tau_I$, where $x_j, y_{ij}$ is replaced with $a_j, b_{ij}$ in the tropicalization. Let $\tw{\kappa}_j = \tw{\tau}_{\{j+2,j+3\}}$. For example, we have
\begin{align*}
\tw{\kappa}_1 = \tw{\tau}_{34} &= \min(a_3 + a_4, a_4 + b_{22}, b_{22} + b_{23}), \\
\tw{\tau}_{24} &= \min(a_2 + a_4, a_2 + b_{23}, b_{11} - b_{22} + \min(a_4, b_{12}, b_{23})).
\end{align*}
By tropicalizing \eqref{eq_Y'_2_4} and \eqref{eq_x'_2_4}, we have
\[
B'_{11} = B_{11} + \tw{\tau}_{14} - \tw{\tau}_{34} \quad\quad B'_{12} = B_{12} + \tw{\tau}_{12} - \tw{\tau}_{24} \quad\quad B'_{22} = B_{22} + \tw{\tau}_{24} - \tw{\tau}_{34} \quad\quad B'_{23} = B_{23} + \tw{\tau}_{23} - \tw{\tau}_{34},
\]
\[
a'_1 = a_1 + \tw{\kappa}_1 - \tw{\kappa}_2 \quad\quad a'_2 = a_2 + \tw{\kappa}_2 - \tw{\kappa}_3 \quad\quad a'_3 = a_3 + \tw{\kappa}_3 - \tw{\kappa}_4 \quad\quad a'_4 = a_4 + \tw{\kappa}_4 - \tw{\kappa}_1.
\]

\begin{ex}
Let $A,B$ correspond to the tableaux $T,U$ from Example \ref{ex_comb_R}. We have
\[
\begin{array}{cccc}
a_1 & a_2 & a_3 & a_4
\end{array}
=
\begin{array}{cccc}
2 & 0 & 4 & 1
\end{array}
\quad\quad \text{ and } \quad\quad
\begin{array}{ccc}
b_{11} & b_{12} & b_{13} \\
b_{22} & b_{23} & b_{24}
\end{array}
=
\begin{array}{ccc}
3 & 1 & 1 \\
2 & 0 & 3
\end{array}.
\]
We compute
\begin{align*}
\tw{\kappa}_1 = \tw{\tau}_{34} &= \min(5,3,2) = 2 \\
\tw{\kappa}_2 = \tw{\tau}_{14} &= \min(3,2+\min(1,0),2) = 2 \\
\tw{\kappa}_3 = \tw{\tau}_{12} &= \min(2,1,2) = 1 \\
\tw{\kappa}_4 = \tw{\tau}_{23} &= \min(4, 5+\min(1,0), 4) = 4 \\
\tw{\tau}_{24} &= \min(1,0,1+\min(1,1,0)) = 0,
\end{align*}
so
\[
\begin{array}{llr}
B'_{11} &= 3 + 2 - 2 &= 3 \\
B'_{12} &= (3+1) + 1 - 0 &= 5 \\
B'_{22} &= 2 + 0 - 2 &= 0 \\
B'_{23} &= (2+0) + 4 - 2 &= 4
\end{array}
\quad\quad \text{ and } \quad\quad
\begin{array}{ll}
a'_1 &= 2 + 2 - 2 = 2 \\
a'_2 &= 0 + 2 - 1 = 1 \\
a'_3 &= 4 + 1 - 4 = 1 \\
a'_4 &= 1 + 4 - 2 = 3.
\end{array}
\]

\smallskip

\noindent The rectangles $B'$ and $A'$ correspond to the tableaux $U'$ and $T'$ from Example \ref{ex_comb_R}, so we have verified that $\wh{R} = \tw{R}$ in this case. Also, by Corollary \ref{cor_E=kappa_1}, we have $\wh{E}(A \otimes B) = \tw{\kappa}_1(A \otimes B) = 2$, which agrees with the coenergy of $T \otimes U$ computed in Example \ref{ex_comb_coenergy}.
\end{ex}

\section{Proof of the positivity of the geometric $R$-matrix}
\label{sec_posit}

In this section we prove Theorem \ref{thm_R_posit}, which states that the geometric $R$-matrix is positive. We start by reducing this theorem to a statement about the positivity of certain minors of the folded matrix $g(N|t)$.

\subsection{Reduction to Proposition \ref{prop_posit_key}}
\label{sec_reduction_to_posit_key}

Recall the notions of positive varieties and positive rational maps from \S \ref{sec_posit_defn}. Let $X$ be a positive variety, $\lp$ an indeterminate, and $f : X \rightarrow \bbC[\lp]$ a rational map, i.e., a map of the form
\[
f = f_0 + f_1 \lp + \ldots + f_d \lp^d,
\]
where $f_i : X \rightarrow \bbC$ are rational functions.
For an integer $r$, we say that $f$ is {\em $r$-non-negative} if for each $i$, the rational function $(-1)^{(r-1)i} f_i$ is non-negative, and we say that $f$ is {\em $r$-positive} if $f$ is $r$-non-negative and not identically zero. For example, for any positive variety $X$, the constant function $f = 1 - \lp + \lp^2$ is $r$-positive for even $r$, but not for odd $r$.

We will need the following observation, whose proof is immediate.

\begin{lem}
\label{lem_r_pos_to_pos}
If $f : X \rightarrow \bbC[\lp]$ is $r$-non-negative (resp., $r$-positive), then the rational function $\ov{f} : X \times \Cx \rightarrow \bbC$ defined by $\ov{f}(x,z) = f(x)|_{\lp = (-1)^{r-1}z}$ is non-negative (resp., positive).
\end{lem}

For two $r$-subsets $I,J \subset [n]$, let $\Delta_{I,J} : \X{k} \rightarrow \bbC[\lp]$ denote the rational map which sends $N|t$ to the minor $\Delta_{I,J}(g(N|t))$. Say that a subset of $[n]$ is a {\em cyclic interval} if its elements are consecutive mod $n$. Define a {\em cyclic interval of a subset $I \subset [n]$} to be a maximal collection of elements of $I$ which form a cyclic interval.

\begin{prop}
\label{prop_posit_key}
Let $I, J \subset [n]$ be two subsets of size $r$, at least one of which has no more than two cyclic intervals, and let $\Delta_{I,J} : \X{k} \rightarrow \bbC[\lp]$ be the rational map just defined. Then
\begin{enumerate}
\item
if $r \leq k$, $\Delta_{I,J}$ is $r$-non-negative;
\item
if $r > k$, $\Delta_{I,J}$ is equal to $(t + (-1)^k \lp)^{r-k} f_{I,J}$, where $f_{I,J}$ is an $r$-non-negative map $\X{k} \rightarrow \bbC[\lp]$.
\end{enumerate}
\end{prop}

\begin{remark}
We expect that Proposition \ref{prop_posit_key} holds without the restriction on $I$ and $J$. We need this restriction in our proof because we do not know the correct analogue of Definition \ref{defn_C_a_b} and Proposition \ref{prop_r_pos} for a subset $J$ with more than two cyclic intervals.
\end{remark}

Before proving Proposition \ref{prop_posit_key}, we explain how it implies Theorem \ref{thm_R_posit}. Since the geometric Sch\"{u}tzenberger involution $S$ is positive, it suffices to show that the map $\Psi_{k,\ell} : \X{\ell} \times \X{k} \rightarrow \X{k}$ is positive. Suppose $(M|s, N|t) \in \X{\ell} \times \X{k}$, and let $N'|t = \Psi_{k,\ell}(M|s, N|t)$, $A = g(M|s)$, and $A_t = A_{\lp = (-1)^{k-1}t}$. Fix a $k$-subset $I$. By \eqref{eq_R_alt} and the Cauchy--Binet formula, we have
\begin{equation*}
P_I(N') = \sum_J \Delta_{I,J}(A_t) P_J(N).
\end{equation*}
If $I$ has at most two cyclic intervals, then by Proposition \ref{prop_posit_key} and Lemma \ref{lem_r_pos_to_pos}, there are non-negative rational functions $f_{I,J} : \X{\ell} \times \Cx \rightarrow \bbC$ such that
\begin{equation*}
\Delta_{I,J}(A_t) = (s + (-1)^{\ell + k -1} t)^{\max(0,k-\ell)} f_{I,J}(M|s, t).
\end{equation*}
Furthermore, by Proposition \ref{prop_g_props}\eqref{itm:det_g}, we have $\det(A_t) = (s + (-1)^{\ell+k-1}t)^{n-\ell}$, so $A_t$ is invertible for $(M|s, t)$ in an open subset of $\X{\ell} \times \Cx$. This means that at least one of the rational functions $f_{I,J}$ is not identically zero.

If $I$ and $I'$ are $k$-subsets with at most two cyclic intervals, then on an open subset of $\X{\ell} \times \X{k}$, we have
\[
\dfrac{P_I(N')}{P_{I'}(N')} = \dfrac{\ds \sum f_{I,J}(M|s, t) P_J(N)}{\ds \sum f_{I',J}(M|s,t) P_J(N)},
\]
where $f_{I,J}$ (resp., $f_{I',J}$) are non-negative rational functions which are not all zero. In particular, this is true when $I$ and $I'$ are basic $k$-subsets (Definition \ref{defn_basic}), so $\Psi_{k,\ell}$ is positive by Lemma \ref{lem_to_Gr}.


\subsection{Proof of Proposition \ref{prop_posit_key}}
\label{sec_pf_posit_key}

To prove Proposition \ref{prop_posit_key}, we first use the Lindstr\"om Lemma to show that all minors of $g(N|t)$ which do not depend on the indeterminate $\lp$ are non-negative. Then we carefully analyze the structure of the matrix $g(N|t)$ to show that every minor in a sufficiently large class (essentially, the minors indexed by basic or reflected basic subsets) is either a product of two $\lp$-independent minors, or a cyclic shift of a $\lp$-independent minor.

\subsubsection{Non-negativity of minors that do not depend on $\lp$}
\label{sec_lp_ind_minors}

Let $A$ be the folded matrix $g(N|t)$, where $N|t \in \X{k}$. Here we view $A$ as an array of $n^2$ rational maps $A_{ij} : \X{k} \rightarrow \bbC[\lp]$. By the definition of $g$, the maps $A_{ij}$ split up into three categories:

\begin{equation}
\label{eq_three_categories}
A_{ij} \text{ is } \begin{cases}
\text{ a nonzero map to } \bbC & \text{ if } i-n+k \leq j \leq i \\
\text{ a nonzero map to } \bbC \cdot \lp & \text{ if } j \geq i+k \\
0 & \text{ if } j < i-n+k \text{ or } i < j < i+k.
\end{cases}
\end{equation}

\noindent In the second case, we say that $A_{ij}$ {\em depends on $\lp$}; otherwise we say that $A_{ij}$ is {\em independent of $\lp$}. Given subsets $I, J \subset [n]$, say that the submatrix $A_{I,J}$ is {\em independent of $\lp$} if $A_{ij}$ is independent of $\lp$ for all $i \in I, j \in J$. If $A_{I,J}$ is independent of $\lp$, then $\Delta_{I,J}$ is a rational function $\X{k} \rightarrow \bbC$, so $r$-positivity of $\Delta_{I,J}$ is the same thing as (ordinary) positivity of $\Delta_{I,J}$.

\begin{lem}
\label{lem_lp_ind_minors}
Let $I = \{i_1 < \cdots < i_r\}$ and $J = \{j_1 < \cdots < j_r\}$ be two $r$-subsets of $[n]$, with $r \leq k$. If the submatrix $A_{I,J}$ is independent of $\lp$, then the rational map $\Delta_{I,J}$ is positive (equivalently, $r$-positive) if
\begin{equation}
\label{eq_lp_ind_minors}
i_s - n + k \leq j_s \leq i_s \quad\quad \text{ for } \; s = 1, \ldots, r, 
\end{equation}
and zero otherwise.
\end{lem}

\begin{proof}
Let $\Phi^{n-k}_{I,J} : \bT{n-k} \rightarrow \bbC$ denote the map $(X_{ij},t) \mapsto \Delta_{I,J}(\Phi_{n-k}(X_{ij},t))$. Since $A_{I,J}$ is independent of $\lp$, Lemma \ref{lem_g_0} implies that
\begin{equation*}
\Phi^{n-k}_{I,J} = \Delta_{I,J} \circ \Theta_k.
\end{equation*}
By Lemma \ref{lem_Phi_minors} (which is an application of the Lindstr\"om Lemma), $\Phi^{n-k}_{I,J}$ is positive if \eqref{eq_lp_ind_minors} holds, and zero otherwise, so the same is true of $\Delta_{I,J}$ (since by definition, $\Delta_{I,J}$ is positive if and only if $\Delta_{I,J} \circ \Theta_k$ is positive).
\end{proof}


\subsubsection{Exploiting the symmetries}
\label{sec_reductions}

Before turning to the heart of the proof, we use the positivity of the symmetries $\PR$, $S$, and $D$ (Lemma \ref{lem_all_maps_posit}) to make several reductions to Proposition \ref{prop_posit_key}.

Suppose $I$ and $J$ are $r$-subsets, and consider the rational map $\Delta_{I,J} : \X{k} \rightarrow \bbC[\lp]$. By \eqref{eq_S_minors}, we have
\[
\Delta_{I,J} \circ S = \Delta_{w_0(J), w_0(I)},
\]
so since $S$ is positive, $r$-non-negativity (resp., $r$-positivity) of $\Delta_{I,J}$ is equivalent to that of $\Delta_{w_0(J), w_0(I)}$. This allows us to reduce to the case where $J$ has at most two cyclic intervals.

Corollary \ref{cor_D_minors} allows us to reduce to the case $r \leq k$, as follows. Assume Proposition \ref{prop_posit_key} holds for $r \leq k$, and fix $r > k$. For $I, J \in {[n] \choose r}$, Corollary \ref{cor_D_minors} gives the identity
\[
\Delta_{I,J} = (t + (-1)^{k} \lp)^{r-k} (\Delta_{\ov{J},\ov{I}}|_{\lp = (-1)^n \lp} \circ D)
\]
of rational maps $\X{k} \rightarrow \bbC[\lp]$. Suppose $I$ or $J$ (equivalently, $\ov{I}$ or $\ov{J}$) has at most two cyclic intervals. By our assumption, the rational map $\Delta_{\ov{J}, \ov{I}} : \X{n-k} \rightarrow \bbC[\lp]$ is $(n-r)$-non-negative, so $\Delta_{\ov{J}, \ov{I}}|_{\lp = (-1)^n \lp}$ is $r$-non-negative. Since $D$ is positive, Proposition \ref{prop_posit_key} holds for $\Delta_{I,J}$.

Corollary \ref{cor_PR_minors} shows that
\[
\Delta_{I,J} \circ \PR = \begin{cases}
\Delta_{I-1,J-1} & \text{ if } 1 \in I \cap J \text{ or } 1 \not \in I \cup J \\
(-1)^{r-1} \lp \cdot \Delta_{I-1,J-1} & \text{ if } 1 \in I \setminus J \\
(-1)^{r-1} \lp^{-1} \cdot \Delta_{I-1,J-1} & \text{ if } 1 \in J \setminus I.
\end{cases}
\]
This, together with the positivity of $\PR$ and $\PR^{-1}$, implies the following result.

\begin{lem}
\label{lem_PR_reduction}
$\Delta_{I,J}$ is $r$-non-negative (resp., $r$-positive) if and only if $\Delta_{I-1, J-1}$ is $r$-non-negative (resp., $r$-positive).
\end{lem}

Recall that a subset is \emph{reflected basic} if it is an interval of $[n]$, or it consists of two disjoint intervals of $[n]$, one of which contains 1 (Definition \ref{defn_basic}). Every subset with at most two cyclic intervals is a cyclic shift of a reflected basic subset, so combining the observations above, we see that it suffices to prove Proposition \ref{prop_posit_key} in the case where $r \leq k$, and $J$ is a reflected basic subset.


\subsubsection{Reflected basic subsets and zero rows}
\label{sec_zero_rows}

Following Convention \ref{conv_pluc}, we interpret an interval $[c,d] \subset \bbZ$ as a cyclic interval of $[n]$ by reducing each element of $[c,d]$ modulo $n$. As usual, $[c,d]$ is the empty set if $c > d$. For example, if $n \geq 6$, then $[-2,3]$ and $[n-2,n+3]$ both represent the cyclic interval $[1,3] \cup [n-2,n]$, but $[n-2,3]$ is the empty set.

Given a subset $J \subset [n]$, let $Z(J)$ be the rows of the submatrix $A_{[n],J}$ which are identically zero (set $Z(\emptyset) = \emptyset$). We call $Z(J)$ the {\em zero rows} of the columns $J$. It follows from \eqref{eq_three_categories} that the $j^{th}$ column of $A$ has zeroes in rows $[j-k+1,j-1]$. This implies that if $s \geq 1$ and $c \in \bbZ$, then
\begin{equation}
\label{eq_Z_cyc_int}
Z([c,c+s-1]) = [c-k+s,c-1].
\end{equation}

\begin{defn}
\label{defn_C_a_b}
Fix $r \leq k$, $a \in [0,r]$, and $b \in [0,n-r]$, and consider the reflected basic $r$-subset
\[
J_{a,b} = [1,a] \cup [a+b+1,r+b].
\]
(Every reflected basic $r$-subset is of this form.) Let $Z_1$ be the zero rows of columns $[1,a]$, and let $Z_2$ be the zero rows of columns $[a+b+1,r+b]$. We say that a subset $I \in {[n] \choose r}$ {\em satisfies condition $C^r_{a,b}$} (or $C_{a,b}$ if $r$ is understood) if
\[
I \cap Z(J_{a,b}) = \emptyset, \quad\quad |I \cap Z_1| \leq r-a, \quad\quad \text{ and } \quad\quad |I \cap Z_2| \leq a.
\]
\end{defn}

Note that if $r = k$, then $|Z_1| = k-a$ and $|Z_2| = a$ by \eqref{eq_Z_cyc_int}, and $Z(J_{a,b})$ is empty because each row of $A$ has only $k-1$ zeroes. Thus, condition $C^k_{a,b}$ always holds. Note also that by \eqref{eq_Z_cyc_int}, we have
\begin{equation}
\label{eq_Z_i}
Z_1 = \begin{cases}
\emptyset & \text{ if } a = 0 \\
[n+a-k+1,n] & \text{ if } a > 0
\end{cases}
\quad\quad
Z_2 = \begin{cases}
[r+b-k+1,a+b] & \text{ if } a < r \\
\emptyset & \text{ if } a = r
\end{cases}
\end{equation}
and thus
\begin{equation}
\label{eq_Z_intersect}
Z(J_{a,b}) =
\begin{cases}
[r+b-k+1,b] & \text{ if } a = 0 \\
[n+a-k+1,a+b] \cup [n+r+b-k+1,n] & \text{ if } a \in [1,r-1] \\
[n+r-k+1,n] & \text{ if } a = r.
\end{cases}
\end{equation}

\begin{prop}
\label{prop_r_pos}
Fix $r \leq k$. Let $J_{a,b}$ be a reflected basic $r$-subset, and let $Z_1 = Z([1,a]), Z_2 = Z([a+b+1,r+b])$ be the zero rows of the two intervals of $J_{a,b}$. Then for $I \in {[n] \choose r}$, the rational map $\Delta_{I,J_{a,b}}$ is $r$-positive if $I$ satisfies condition $C_{a,b}$, and zero otherwise.
\end{prop}

Thanks to the reductions in \S \ref{sec_reductions}, this result implies Proposition \ref{prop_posit_key}. The proof of Proposition \ref{prop_r_pos} is rather technical. The idea is to use Lemma \ref{lem_lp_ind_minors} and the cyclic shifting map to show that a large class of the minors $\Delta_{I,J_{a,b}}$ are $r$-positive, and then to show that all other minors of the form $\Delta_{I,J_{a,b}}$ are either zero, or can be expressed as positive Laurent polynomials in the minors that are known to be $r$-positive. We carry out the first step with Lemma \ref{lem_r_pos}, and the second step in \S \ref{sec_posit_proof_complete}.

\begin{lem}
\label{lem_r_pos}
Fix $r \leq k$.
\begin{enumerate}
\item The submatrix $A_{I,J_{a,b}}$ is independent of $\lp$ if and only if $I \subset [r+b-k+1,n]$, or $a = r$.
\item If $A_{I,J_{a,b}}$ is independent of $\lp$, then $\Delta_{I,J_{a,b}}$ is positive (equivalently, $r$-positive) if $I$ satisfies condition $C_{a,b}$, and zero otherwise.
\item If there is some $c$ such that $J_{a,b} - c = J_{a',b'}$ and the submatrix $A_{I-c,J_{a',b'}}$ is independent of $\lp$, then Proposition \ref{prop_r_pos} holds for $\Delta_{I,J_{a,b}}$. Here $S - c$ is the subset obtained from $S$ by subtracting $c$ from each element, and interpreting the result mod $n$.
\end{enumerate}
\end{lem}

\begin{proof}
Let $I = \{i_1 < \ldots < i_r\}$ and $J_{a,b} = \{j_1 < \ldots < j_r\}$. By \eqref{eq_three_categories}, $A_{ij}$ is independent of $\lp$ if and only if $j-i < k$, so $A_{I,J^{a,b}}$ is independent of $\lp$ if and only if
\begin{equation}
\label{eq_j_r_i_1}
j_r - i_1 < k.
\end{equation}
If $a = r$, then $J_{a,b} = [r]$, so \eqref{eq_j_r_i_1} holds for every $I$. If $a \neq r$, then $j_r = r+b$, so \eqref{eq_j_r_i_1} holds if and only if $i_1 > r+b-k$. This proves (1).

Now suppose $A_{I,J_{a,b}}$ is independent of $\lp$.
Specializing Lemma \ref{lem_lp_ind_minors} to the case $J = J_{a,b}$, we see that $\Delta_{I,J_{a,b}}$ is positive if $I$ satisfies
\begin{equation}
\label{eq_i_s_bloated}
i_s \in \begin{cases}
[s,s+n-k] & \text{ if } s = 1, \ldots, a \\
[s+b,s+b+n-k] & \text{ if } s = a+1, \ldots, r,
\end{cases}
\end{equation}
and zero otherwise. So to prove (2), we must show that:

\textit{Given the assumption $I \subset [r+b-k+1,n]$ or $a = r$, $I$ satisfies condition $C_{a,b}$ if and only if $I$ satisfies \eqref{eq_i_s_bloated}.}

To see this, first suppose $a \in [r-1]$. In this case, \eqref{eq_i_s_bloated} is equivalent to the three inequalities
\begin{equation}
\label{eq_i_s}
i_a \leq n+a-k, \quad\quad i_{a+1} \geq a+b+1, \quad\quad i_r \leq r+b+n-k.
\end{equation}
Using \eqref{eq_Z_i}, \eqref{eq_Z_intersect} and considering separately the cases $r+b \geq k$ and $r+b < k$, it is straightforward to check that for $I \subset [r+b-k+1,n]$, \eqref{eq_i_s} is equivalent to condition $C_{a,b}$. If $a = r$, \eqref{eq_i_s_bloated} is equivalent to the first inequality of \eqref{eq_i_s}; if $a = 0$, \eqref{eq_i_s_bloated} is equivalent to the last two inequalities of \eqref{eq_i_s}. The verification of the claim in these ``degenerate'' cases is similar.

For (3), suppose $A_{I',J_{a',b'}}$ is independent of $\lp$, where $I' = I-c$ and $J_{a',b'} = J_{a,b}-c$. The cyclic symmetry of the locations of zeroes in the matrix $A$ implies that $I'$ satisfies condition $C_{a',b'}$ if and only if $I$ satisfies condition $C_{a,b}$. Thus, Proposition \ref{prop_r_pos} holds for $\Delta_{I,J_{a,b}}$ by Lemma \ref{lem_PR_reduction} and part (2).
\end{proof}


\subsubsection{Conclusion of the proof}
\label{sec_posit_proof_complete}

We now complete the proof of Proposition \ref{prop_posit_key} (and thus the proof of Theorem \ref{thm_R_posit}) by proving Proposition \ref{prop_r_pos}.

Suppose $I$ does not satisfy condition $C_{a,b}$. If $I \cap Z(J_{a,b}) \neq \emptyset$, then the submatrix $A_{I,J_{a,b}}$ has a row of zeroes, so its determinant vanishes. If $|I \cap Z_1| > r-a$, then the first $a$ columns of $A_{I,J_{a,b}}$ have at least $r-a+1$ zero rows, so again the determinant vanishes. The case $|I \cap Z_2| > a$ is similar.

Now suppose $I$ satisfies condition $C_{a,b}$. If $J_{a,b}$ is contained in a cyclic interval of size $k$, then there is some $c$ so that $J_{a,b} - c = J_{a',b'} \subset [k]$. The first $k$ columns of $A$ are independent of $\lp$, so $\Delta_{I,J_{a,b}}$ is $r$-positive by Lemma \ref{lem_r_pos}(3).

Assume that $J_{a,b}$ is not contained in a cyclic interval of size $k$. The zeroes in each row of $A$ are located in $k-1$ cyclically consecutive columns, so in this case, the submatrix $A_{[n], J_{a,b}}$ does not have a row of zeroes. In other words, $Z(J_{a,b}) = \emptyset$, so the first part of condition $C_{a,b}$ is automatically satisfied. Also, $Z_1$ and $Z_2$ are both non-empty, and $|Z_1| = k-a$, $|Z_2| = k-r+a$, so $I$ satisfies condition $C_{a,b}$ if and only if
\[
|Z_i \setminus I| \geq k-r \quad \text{ for } i = 1,2.
\]
Thus, we need to show that $\Delta_{I,J_{a,b}}$ is $r$-positive whenever
\begin{equation}
\label{eq_Y_I}
Y_1 \in {Z_1 \choose k-r}, \quad\quad Y_2 \in {Z_2 \choose k-r}, \quad \text{ and } \quad I \in {[n] \setminus (Y_1 \cup Y_2) \choose r}.
\end{equation}

Let $Y_1, Y_2,$ and $I$ be as in \eqref{eq_Y_I}. Suppose first that $I \supset Z_1 \setminus Y_1$. In this case, since $Z_1 = [n+a-k+1,n]$, the lower left $(r-a) \times a$ submatrix of $A_{I,J_{a,b}}$ consists entirely of zeroes, so we have
\begin{equation*}
\Delta_{I,J_{a,b}}(A) = \Delta_{I', [1,a]}(A) \Delta_{I'', [a+b+1, r+b]}(A),
\end{equation*}
where $I' = I \setminus (Z_1 \setminus Y_1)$ and $I'' = Z_1 \setminus Y_1$. The first $k$ columns and the last $k$ rows of $A$ are independent of $\lp$, so the submatrices $A_{I', [1,a]}$ and $A_{I'', [a+b+1,r+b]}$ are independent of $\lp$. The $a$-subset $I'$ is disjoint from $Z_1$, so $I'$ satisfies condition $C^a_{a,0}$; similarly, since $Z_1 \cap Z_2 = Z(J_{a,b}) = \emptyset$, the $(r-a)$-subset $I''$ is disjoint from $Z_2$, so $I''$ satisfies condition $C^{r-a}_{0,a+b}$. Thus, $\Delta_{I', [1,a]}$ and $\Delta_{I'', [a+b+1, r+b]}$ are positive rational functions by Lemma \ref{lem_r_pos}(2), so $\Delta_{I,J_{a,b}}$ is a positive (hence $r$-positive) rational function.

If $I \supset Z_2 \setminus Y_2$, set $c = a+b$, so that $J_{a,b} - c = J_{r-a,n-r-b}$. Let $I' = I - c$, $Z'_1 = Z_2 - c$, and $Y'_1 = Y_2 - c$. Clearly $Z'_1$ consists of the zero rows of columns $[1,r-a]$, and $I' \supset Z'_1 \setminus Y'_1$, so $\Delta_{I',J_{r-a,n-r-b}}$ is $r$-positive by the previous paragraph, and $\Delta_{I,J_{a,b}}$ is $r$-positive by Lemma \ref{lem_PR_reduction}.

It remains to consider the case where $I \not \supset Z_i \setminus Y_i$ for $i = 1, 2$. This case is subtler, and we proceed indirectly. Set $S = [n] \setminus (Y_1 \cup Y_2)$. Call a subset of $S$ an {\em $S$-interval} if it is of the form $S \cap [c,d]$ (with $1 \leq c \leq d \leq n$), and it is a maximal subset of $S$ with this property. Let $I$ be a basic $r$-subset of $S$; this means that $I$ consists of one or two $S$-intervals, and if there are two $S$-intervals, one of them contains the largest element of $S$. If $I \subset [r+b-k+1,n]$, then $\Delta_{I,J_{a,b}}$ is $r$-positive by Lemma \ref{lem_r_pos}, so suppose $I \not \subset [r+b-k+1,n]$. We claim that
\begin{equation}
\label{eq_I_I'}
I \subset [1,a+b] \cup [n+a-k+1, n].
\end{equation}

To prove \eqref{eq_I_I'}, recall that $Z_1 = [n+a-k+1,n]$, and $Z_2 = [r+b-k+1,a+b]$. Since $I \not \subset [r+b-k+1,n]$, we must have $r+b-k+1 \geq 2$, so $[r+b-k+1,a+b]$ is an actual interval of $[n]$ (i.e., it doesn't ``wrap around''), and
\[
S = [1,r+b-k] \cup (Z_2 \setminus Y_2) \cup [a+b+1,n+a-k] \cup (Z_1 \setminus Y_1).
\]
By assumption, $I$ is a basic $r$-subset of $S$ which intersects $[1,r+b-k]$ and does not contain all of $Z_2 \setminus Y_2$ or $Z_1 \setminus Y_1$. There are two possibilities: either $I$ is a single $S$-interval contained in $[1, r+b-k] \cup Z_2 \setminus Y_2$, or $I = I_1 \cup I_2$, where $I_1$ is an $S$-interval contained in $[1,r+b-k] \cup Z_2 \setminus Y_2$, and $I_2$ is an $S$-interval contained in $Z_1 \setminus Y_1$. In either case, $I \cap [a+b+1,n+a-k] = \emptyset$, so \eqref{eq_I_I'} holds.

Now let $I' = I - (a+b)$, and note that $J_{a,b} - (a+b) = J_{r-a,n-r-b}$. By \eqref{eq_I_I'}, we have
\[
I' \subset [n-b-k+1,n] = [r+(n-r-b)-k+1,n],
\]
so $\Delta_{I,J_{a,b}}$ is $r$-positive by Lemma \ref{lem_r_pos}.

Now let $J$ be an arbitrary $r$-subset of $S$. Lemma \ref{lem_basic=basis} says that $\Delta_{J,J_{a,b}}$ can be expressed as a positive Laurent polynomial in the minors $\Delta_{J',J_{a,b}}$ with $J'$ a basic subset of $S$. We have shown that these $\Delta_{J',J_{a,b}}$ are $r$-positive; furthermore, it is clear from the proof that each of these $\Delta_{J',J_{a,b}}$ is a monomial with respect to $\lp$. It follows that $\Delta_{J,J_{a,b}}$ is $r$-positive (although not necessarily a monomial).

We conclude that $\Delta_{I,J_{a,b}}$ is $r$-positive whenever $I$ satisfies \eqref{eq_Y_I}, completing the proof.

\section{Proof of the identity $g \circ R = g$}
\label{sec_pf1}

In this section we prove Theorem \ref{thm_hard}. Suppose $u = M|s \in \X{\ell}$ and $v = N|t \in \X{k}$. Let $A = g(u)g(v)$ (viewed as a folded matrix), and let $A_t = A|_{\lp = (-1)^{k-1}t}$, and $A_s = A|_{\lp = (-1)^{\ell-1}s}$. Define $v' = N'|t \in \X{k}$ and $u' = M'|s \in \X{\ell}$ by $R(u,v) = (v',u')$. We must show that
\begin{equation}
\label{eq_hard_to_prove}
g(u)g(v) = g(v')g(u').
\end{equation}

For $I \in {[n] \choose k}$, let $P'_I = P_I(N')$, and for $J \in {[n] \choose \ell}$, let $Q'_J = Q^J_s(M') := P_{w_0(J)}(S_s(M'))$. By \eqref{eq_P_Q_N'_M'}, we have
\begin{equation}
\label{eq_P'_Q'}
\dfrac{P'_I}{P'_{[n-k+1,n]}} = \dfrac{\Delta_{I, [k]}(A_t)}{\Delta_{[n-k+1,n], [k]}(A_t)} \quad\quad \text{ and } \quad\quad \dfrac{Q'_J}{Q'_{[\ell]}} = \dfrac{\Delta_{[n-\ell+1,n],J}(A_s)}{\Delta_{[n-\ell+1,n], [\ell]}(A_s)}.
\end{equation}

The key to the proof of Theorem \ref{thm_hard} is the following identity.

\begin{prop}
\label{prop_key}
For $r = 1, \ldots, n$, we have
\begin{equation}
\label{eq_key}
(t + (-1)^k \lp) \dfrac{Q'_{[\ell-1] \cup \{r\}}}{Q'_{[\ell]}} = \sum_{a = 1}^n (-1)^{n+a} \dfrac{P'_{[n-k,n] \setminus \{a\}}}{P'_{[n-k+1,n]}} A_{ar}.
\end{equation}
\end{prop}
(Note that by Convention \ref{conv_pluc}, the terms on the right-hand side of \eqref{eq_key} with $a < n-k$ are zero, and the left-hand side is zero when $r \leq \ell-1$.)

Before proving this identity, we use it to deduce Theorem \ref{thm_hard}. Consider the folded matrices
\[
B = (t+(-1)^k \lp) \cdot g(u') \quad\quad \text{ and } \quad\quad C = h(v')g(u)g(v),
\]
where $h(v')$ is defined by
\[
h(v')_{ij} = (-1)^{i+j} c''_{ij}(k,t) \dfrac{P'_{[i-k,i] \setminus \{j\}}}{P'_{[i-k+1,i]}}, \quad\quad\quad
c''_{ij}(k,t) =
\begin{cases}
1 & \text{ if } i > k \\
t & \text{ if } i \leq k \text{ and } i \geq j \\
(-1)^n \lp & \text{ if } i \leq k \text { and } i < j.
\end{cases}
\]
The right-hand side of \eqref{eq_key} is equal to $C_{nr}$, and by Lemma \ref{lem_g_in_Q}, the left-hand side of \eqref{eq_key} is equal to $B_{nr}$; thus, Proposition \ref{prop_key} says that $B_{nr} = C_{nr}$ for all $r$. Recall the cyclic shift map $\PR$ and the shift automorphism $\sh$ from \S \ref{sec_cyclic_symm}. Since $\sh \circ \, g = g \circ \PR$ (Lemma \ref{lem_PR}), $\sh \circ \, h = h \circ \PR$ (this is straightforward to check), and $\PR$ commutes with $R$ (Theorem \ref{thm_master}\eqref{itm:R_commutes_symms}), the equality of the last rows of $B$ and $C$ implies that $B = C$.

By \cite[Lem. 7.7]{F1}, the matrix $h(v')$ satisfies
\[
g(v')h(v') = (t+(-1)^k \lp) \cdot Id,
\]
so left-multiplying $B$ and $C$ by $g(v')$ gives the desired equality \eqref{eq_hard_to_prove}.

\begin{proof}[Proof of Proposition \ref{prop_key}]
Let
\[
p_r(\lp) = \sum_{a = n-k}^n (-1)^{n+a} \dfrac{P'_{[n-k,n] \setminus \{a\}}}{P'_{[n-k+1,n]}} A_{ar}
\]
be the right-hand side of \eqref{eq_key}. Let $X$ be the unfolding of $A$. The entries of the matrices $g(u)$ and $g(v)$ are at most linear in $\lp$, so the entries of $A$ are at most quadratic in $\lp$, and
\[
A_{ij} = X_{ij} + \lp X_{n+i,j} + \lp^2 X_{2n+i,j}.
\]
Recall from \S \ref{sec_geom_unip_Gr} that an unfolded matrix $X$ is {\em $m$-shifted unipotent} if $X_{ij} = 0$ when $i-j > m$, and $X_{ij} = 1$ when $i-j=m$. The matrices $g(u)$ and $g(v)$ are $(n-\ell)$- and $(n-k)$-shifted unipotent, respectively, so their product is $(2n-\ell-k)$-shifted unipotent. This implies, in particular, that if $a \geq n-k$, then $A_{ar}$ is either constant or linear in $\lp$, so $p_r(\lp)$ is a polynomial of degree at most one. Proposition \ref{prop_key} is therefore an immediate consequence of the following two claims:
\begin{enumerate}
\item $(-1)^{k-1}t$ is a root of $p_r(\lp)$;
\item the coefficient of $\lp$ in $p_r(\lp)$ is $ \ds (-1)^k \dfrac{Q'_{[\ell-1] \cup \{r\}}}{Q'_{[\ell]}}$.
\end{enumerate}

To prove the first claim, let $D = (A_t)_{[n-k,n], <1, 2, \ldots, k, r>}$ denote the (generalized) submatrix of $A_t$ consisting of the last $k+1$ rows, and columns $1, \ldots, k, r$, in that order. If $r \in [k]$, then clearly $\det(D) = 0$; if $r \not \in [k]$, then $\det(D)$ is still zero because $g(v)|_{\lp = (-1)^{k-1}t}$ has rank $k$ by Proposition \ref{prop_g_props}\eqref{itm:rank_k}. On the other hand, expanding the determinant along column $r$ gives $\det(D) = p_r((-1)^{k-1}t)$, so (1) follows.

It remains to prove the second claim. Since the coefficient of $\lp$ in $A_{ar}$ is $X_{n+a,r}$, claim (2) can be rephrased as the identity
\begin{equation}
\label{eq_claim_2}
\dfrac{Q'_{[\ell-1] \cup \{r\}}}{Q'_{[\ell]}} = \sum_{a = n-k}^n (-1)^{k+n+a} \dfrac{P'_{[n-k,n] \setminus \{a\}}}{P'_{[n-k+1,n]}} X_{n+a,r}.
\end{equation}
If $r \leq \ell-1$, then $X_{n+a,r} = 0$ for $a \geq n-k$ (since $X$ is $(2n-\ell-k)$-shifted unipotent), so \eqref{eq_claim_2} holds trivially in this case.

To prove \eqref{eq_claim_2} for $r \geq \ell$, we start by massaging the (folded) matrices $g(u)$ and $g(v)$ into a more convenient form. By Proposition \ref{prop_g_props}\eqref{itm:rank_k}, the matrix $g(u)|_{\lp = (-1)^{\ell-1}s}$ has rank $\ell$. This means that we may add linear combinations of the last $\ell$ rows of $g(u)$ (which are linearly independent, and do not depend on $\lp$) to the first $n-\ell$ rows to obtain the matrix $g(u)^*$, where
\[
g(u)^*_{ij} = \begin{cases}
g(u)_{ij} & \text{ if } i \geq n-\ell+1 \\
(\lp + (-1)^\ell s) g(u)_{ij} & \text{ if } j-i \geq \ell \\
0 & \text{ otherwise}.
\end{cases}
\]
Similarly, we may add linear combinations of the first $k$ columns of $g(v)$ to the last $n-k$ columns to obtain the matrix $g(v)^*$, where
\[
g(v)^*_{ij} = \begin{cases}
g(v)_{ij} & \text{ if } j \leq k \\
(\lp + (-1)^k t) g(v)_{ij} & \text{ if } j-i \geq k \\
0 & \text{ otherwise}.
\end{cases}
\]
Define $A^* = g(u)^* g(v)^*$. See Figure \ref{fig_*_matrices} for an example of the matrices $g(u)^*$, $g(v)^*$, and $A^*$.

\begin{figure}
\begin{equation*}
\left(
\begin{array}{cccc}
* & 0 & \lp & \lp * \\
* & * & 0 & \lp \\
1 & * & * & 0 \\
0 & 1 & * & *
\end{array}
\right)
\left(
\begin{array}{cccc}
* & 0 & \lp & \lp * \\
* & * & 0 & \lp \\
1 & * & * & 0 \\
0 & 1 & * & *
\end{array}
\right)
=
\left(
\begin{array}{cccc}
X_{11} + \lp & \lp X_{52} & \lp X_{53} & \lp X_{54} \\
X_{21} & X_{22} + \lp & \lp X_{63} & \lp X_{64} \\
X_{31} & X_{32} & X_{33} + \lp & \lp X_{74} \\
X_{41} & X_{42} & X_{43} & X_{44} + \lp
\end{array}
\right)
\end{equation*}

\begin{equation*}
\left(
\begin{array}{cccc}
0 & 0 & \lp + s & (\lp +s)* \\
0 & 0 & 0 & \lp + s \\
1 & * & * & 0 \\
0 & 1 & * & *
\end{array}
\right)
\left(
\begin{array}{cccc}
* & 0 & \lp + t & (\lp+t) * \\
* & * & 0 & \lp +t \\
1 & * & 0 & 0 \\
0 & 1 & 0 & 0
\end{array}
\right)
=
\left(
\begin{array}{cc|cc}
\lp+s & (\lp+s) X_{52} & 0 & 0 \\
0 & \lp+s & 0 & 0 \\ \hline
X_{31} & X_{32} & \lp+t & (\lp+t) X_{74} \\
X_{41} & X_{42} & 0 & \lp+t
\end{array}
\right)
\end{equation*}

\caption{Suppose $n = 4, \ell = k = 2$, and $u = M|s, v = N|t \in \X{2}$. The first line shows the product $g(u)g(v) = A$, where the $*$'s are ratios of Pl\"{u}cker coordinates of $M$ or $N$, possibly scaled by $s$ or $t$, and $X$ is the unfolding of $A$. The second line shows the product $g(u)^*g(v)^* = A^*$, with the blocks of $A^*$ indicated.}
\label{fig_*_matrices}
\end{figure}

Given two subsets $I,J \subset [n]$ of the same cardinality, say that $(I,J)$ is a {\em good pair} if $I$ contains or is contained in the interval $[n-\ell+1,n]$, and $J$ contains or is contained in the interval $[1,k]$. The construction of $g(u)^*$ and $g(v)^*$, together with the Cauchy--Binet formula, implies that
\begin{equation}
\label{eq_good_pair}
\Delta_{I,J}(A^*) = \Delta_{I,J}(A) \quad \text{ if } (I,J) \text{ is a good pair.}
\end{equation}

Set
\[
\alpha_s = \lp + (-1)^\ell s, \quad\quad\quad \alpha_t = \lp + (-1)^k t.
\]
The reader may easily verify that the entries of $A^*$ are given by
\begin{equation}
\label{eq_A*}
(A^*)_{ij} = \begin{cases}
X_{ij} & \text{ if } i \geq n-\ell+1 \text{ and } j \leq k \\
\alpha_s X_{n+i,j} & \text{ if } i \leq n-\ell \text{ and } j \leq k \\
\alpha_t X_{n+i,j} & \text{ if }i \geq n-\ell+1 \text{ and } j \geq k+1 \\
\alpha_s \alpha_t X_{2n+i,j} & \text{ if }i \leq n-\ell \text{ and } j \geq k+1.
\end{cases}
\end{equation}
In other words, $A^*$ has the block form

\begin{center}
\begin{tikzpicture}[scale=1.2]

\draw (2,0) -- (2,4);
\draw (0,2) -- (4,2);
\draw (0,0) rectangle (4,4);
\draw (1,3) node[scale=2]{$\alpha_s E$};
\draw (1,1) node[scale=2]{$G$};
\draw (3,3) node[scale=2]{$\alpha_s \alpha_t F$};
\draw (3,1) node[scale=2]{$\alpha_t H$};
\draw[|-|] (-.7,3.8) --node[left]{$n-\ell$} (-0.7,2.2);
\draw[|-|] (-0.7,1.8) --node[left]{$\ell$} (-0.7,0.2);
\draw[|-|] (0.2,-0.7) --node[below]{$k$} (1.8,-0.7);
\draw[|-|] (2.2,-0.7) --node[below]{$n-k$} (3.8,-0.7);

\end{tikzpicture}
\end{center}

\noindent where $G$ is filled with the constant terms of the corresponding entries of $A$, $E$ and $H$ with the coefficients of $\lp$ in the corresponding entries of $A$, and $F$ with the coefficients of $\lp^2$ in the corresponding entries of $A$ (examples are shown in Figures \ref{fig_*_matrices} and \ref{fig_ell>k}). Furthermore, the matrices $E,F,G,H$ satisfy:
\begin{itemize}
\item $H$ and $\fl(E)$ (see \S \ref{sec_schutz}) are upper uni-triangular;
\item all entries in the first column and last row of $F$ are zero.
\end{itemize}
Taken together, these properties imply that row $n-\ell$ (resp., column $k+1$) of $A^*$ has only one nonzero entry, namely, $(A^*)_{n-\ell,k} = \alpha_s$ (resp., $(A^*)_{n-\ell+1,k+1} = \alpha_t$).

The argument now splits into two cases.

\smallskip
\noindent \textbf{Case 1: $\ell \neq k$.}
First note that the argument used to deduce Theorem \ref{thm_hard} from Proposition \ref{prop_key} can be reversed to deduce the latter from the former, so these two results are in fact equivalent. Thanks to Lemma \ref{lem_S} and the fact that the geometric Sch\"{u}tzenberger involution commutes with the geometric $R$-matrix, \eqref{eq_hard_to_prove} holds for $(u,v) \in \X{\ell} \times \X{k}$ if and only if it holds for $(S(v), S(u)) \in \X{k} \times \X{\ell}$. Thus, the $\ell,k$ case of Theorem \ref{thm_hard} is equivalent to the $k,\ell$ case, so the same is true of Proposition \ref{prop_key}, and we may assume $\ell > k$ here.

\begin{figure}

\begin{equation*}
\left(
\begin{array}{cc|ccccc}
\alpha_s X_{81} & \alpha_s X_{82} & 0 & 0 & 0 & 0 & \alpha_s \alpha_t \\
\alpha_s & \alpha_s X_{92} & 0 & 0 & 0 & 0 & 0 \\
0 & \alpha_s & 0 & 0 & 0 & 0 & 0 \\ \hline
X_{41} & X_{42} & \alpha_t & \alpha_t X_{11,4} & \alpha_t X_{11,5} & \alpha_t X_{11,6} & \alpha_t X_{11,7} \\
X_{51} & X_{52} & 0 & \alpha_t & \alpha_t X_{12,5} & \alpha_t X_{12,6} & \alpha_t X_{12,7} \\
X_{61} & X_{62} & 0 & 0 & \alpha_t & \alpha_t X_{13,6} & \alpha_t X_{13,7} \\
X_{71} & X_{72} & 0 & 0 & 0 & \alpha_t & \alpha_t X_{14,7}
\end{array}
\right)
\end{equation*}

\caption{The matrix $A^*$ in the case $n = 7, \ell = 4, k = 2$, with blocks indicated.}
\label{fig_ell>k}

\end{figure}

Using \eqref{eq_good_pair}, \eqref{eq_A*}, and the fact that the matrix $H$ in the block form of $A^*$ is upper uni-triangular (it may be helpful to refer to Figure \ref{fig_ell>k}), we compute
\begin{align*}
\Delta_{[n-\ell+1,n], [\ell-1] \cup \{r\}}(A) &= \Delta_{[n-\ell+1,n], [\ell-1] \cup \{r\}}(A^*) \\
&= (-1)^{k(\ell-k)} \alpha_t^{\ell-k-1} \sum_{a = n-k}^n (-1)^{a-n+k} A^*_{ar} \Delta_{[n-k,n] \setminus \{a\}, [k]}(A^*) \\
&= (-1)^{k(\ell-k)} \alpha_t^{\ell-k} \sum_{a = n-k}^n (-1)^{a-n+k} X_{n+a,r} \Delta_{[n-k,n] \setminus \{a\}, [k]}(A).
\end{align*}
In particular, we have
\[
\Delta_{[n-\ell+1,n], [\ell]}(A) = (-1)^{k(\ell-k)} \alpha_t^{\ell-k} \Delta_{[n-k+1,n], [k]}(A),
\]
and thus
\begin{equation}
\label{eq_Deltas}
\frac{\Delta_{[n-\ell+1,n], [\ell-1] \cup \{r\}}(A)}{\Delta_{[n-\ell+1,n], [\ell]}(A)} = \sum_{a = n-k}^n (-1)^{a-n+k} X_{n+a,r} \frac{\Delta_{[n-k,n] \setminus \{a\}, [k]}(A)}{\Delta_{[n-k+1,n], [k]}(A)}.
\end{equation}

By \eqref{eq_P'_Q'}, the two sides of \eqref{eq_claim_2} are obtained by evaluating the two sides of \eqref{eq_Deltas} at $\lp = (-1)^{\ell-1}s$ and $\lp = (-1)^{k-1}t$, respectively. Note, however, that the entries of the submatrix $A_{[n-\ell+1,n],[k]}$ do not depend on the value of $\lp$, so since $k < \ell$, the right-hand side of \eqref{eq_Deltas} is independent of $\lp$. This means that the left-hand side is also independent of $\lp$, and \eqref{eq_claim_2} follows.

\smallskip
\noindent \textbf{Case 2: $\ell = k$.}

Since $X_{n+a,k}$ is 1 when $a = n-k$ and 0 when $a \geq n-k$, \eqref{eq_claim_2} clearly holds when $r = k$. Fix $r \geq k+1$. Let $z$ denote the coefficient of $\lp$ in $\Delta_{[n-k,n],[k] \cup \{r\}}(A)$. We will deduce \eqref{eq_claim_2} by computing $z$ in two ways. On the one hand, we use \eqref{eq_good_pair}, \eqref{eq_A*}, and the fact that the only nonzero entry in the $(n-k)^{th}$ row of $A^*$ is $A^*_{n-k,k} = \alpha_s$ to compute
\begin{align*}
\Delta_{[n-k,n],[k] \cup \{r\}}(A) &= \Delta_{[n-k,n],[k] \cup \{r\}}(A^*) \\
&= \alpha_s \sum_{a = n-k+1}^n (-1)^{a-n+k-1} A^*_{ar} \Delta_{[n-k+1,n] \backslash \{a\}, [k-1]}(A^*) \\
&= \alpha_s \alpha_t \sum_{a = n-k+1}^n (-1)^{a-n+k-1} X_{n+a,r} \Delta_{[n-k+1,n] \backslash \{a\}, [k-1]}(A),
\end{align*}
which shows that
\begin{equation}
\label{eq_z_from_A*}
z = -(s+t) \sum_{a = n-k+1}^n (-1)^{a-n} X_{n+a,r} \Delta_{[n-k+1,n] \backslash \{a\}, [k-1]}(A).
\end{equation}

On the other hand, the coefficient of $\lp$ in the determinant $\Delta_{[n-k,n],[k] \cup \{r\}}(A)$ can be computed from the unfolded matrix $X$ by taking the alternating sum of determinants of submatrices of columns $[1,k] \cup \{r\}$ of $X$, where $k$ of the rows come from $[n-k,n]$ and one row comes from $[2n-k,2n]$; that is,
\begin{equation}
\label{eq_z_from_X}
z = \sum_{a = n-k}^n (-1)^{n-a} \Delta_{([n-k,n] \backslash \{a\}) \cup \{n+a\},[k] \cup \{r\}}(X).
\end{equation}

Consider the term in this sum with $a=n-k$. Expanding the determinant along row $2n-k$ (it may be helpful to refer to the first line of Figure \ref{fig_*_matrices}), we obtain
\begin{equation*}
\Delta_{[n-k+1,n] \cup \{2n-k\},[k] \cup \{r\}}(X) = X_{2n-k,r} \Delta_{[n-k+1,n], [k]}(X) - \Delta_{[n-k+1,n],[k-1] \cup \{r\}}(X).
\end{equation*}
Observe that the entries of the bottom-left $k \times k$ submatrix of $A$ do not depend on $\lp$, and the entries $A_{ar}$ for $a \geq n-k+1$ are polynomials in $\lp$ of degree at most one. This means that
\begin{equation}
\label{eq_bottom_left_minor}
\Delta_{[n-k+1,n],[k]}(X) = \Delta_{[n-k+1,n],[k]}(A) = \Delta_{[n-k+1,n],[k]}(A_s) = \Delta_{[n-k+1,n],[k]}(A_t)
\end{equation}
and
\begin{multline*}
\Delta_{[n-k+1,n],[k-1] \cup \{r\}}(X) = \Delta_{[n-k+1,n],[k-1] \cup \{r\}}(A) \\
- \lp \sum_{a = n-k+1}^n (-1)^{n-a} X_{n+a,r} \Delta_{[n-k+1,n] \backslash \{a\}, [k-1]}(A).
\end{multline*}
Since the left-hand side of this equation is independent of $\lp$, we may substitute $\lp = (-1)^{k-1}s$ into the right-hand side to obtain
\begin{multline*}
\Delta_{[n-k+1,n],[k-1] \cup \{r\}}(X) = \Delta_{[n-k+1], [k-1] \cup \{r\}}(A_s) \\
- (-1)^{k-1}s \sum_{a = n-k+1}^n (-1)^{n-a} X_{n+a,r} \Delta_{[n-k+1,n] \backslash \{a\}, [k-1]}(A)
\end{multline*}
(in the second line, we again use the fact that the bottom-left $k \times k$ submatrix of $A$ is independent of $\lp$.) By similar reasoning,
\begin{equation*}
\Delta_{([n-k,n] \backslash \{a\}) \cup \{n+a\},[k] \cup \{r\}}(X) = X_{n+a,r}\left(\Delta_{[n-k,n] \backslash \{a\},[k]}(A_t) - t\Delta_{[n-k+1,n] \backslash \{a\}, [k-1]}(A)\right)
\end{equation*}
for $a = n-k+1, \ldots, n$.

Putting all of this together, we may rewrite \eqref{eq_z_from_X} as
\begin{multline}
\label{eq_z_from_X_again}
z = (-1)^k X_{2n-k,r}\Delta_{[n-k+1,n],[k]}(A) - (-1)^k\Delta_{[n-k+1,n],[k-1] \cup \{r\}}(A_s) \\
+ \sum_{a=n-k+1}^n (-1)^{n-a} X_{n+a,r} \left(\Delta_{[n-k,n] \backslash \{a\},[k]}(A_t) -(s+t) \Delta_{[n-k+1,n] \backslash \{a\}, [k-1]}(A)\right).
\end{multline}
Equating the expressions for $z$ in \eqref{eq_z_from_A*} and \eqref{eq_z_from_X_again}, we obtain
\begin{equation*}
\Delta_{[n-k+1,n],[k-1] \cup \{r\}}(A_s) = X_{2n-k,r}\Delta_{[n-k+1,n],[k]}(A) + \sum_{a=n-k+1}^n (-1)^{n-a+k} X_{n+a,r} \Delta_{[n-k,n] \backslash \{a\}, [k]}(A_t).
\end{equation*}
Divide both sides of this equation by $\Delta_{[n-k+1,n],[k]}(A)$ and use \eqref{eq_bottom_left_minor} and \eqref{eq_P'_Q'} to obtain \eqref{eq_claim_2}. This completes the proof.
\end{proof}

\appendix

\section{Planar networks and the Lindstr\"{o}m Lemma}
\label{app_A}

By {\em planar network}, we mean a finite, directed, edge-weighted graph embedded in a disc, with no oriented cycles. The edge weights are nonzero complex numbers (or indeterminates which take values in $\Cx$). We assume there are $r$ distinguished source vertices, labeled $1, \ldots, r$, and $s$ distinguished sink vertices, labeled $1', \ldots, s'$. To each such network $\Gamma$, we associate an $r \times s$ matrix $M(\Gamma)$, as follows. Define the weight of a path to be the product of the weights of the edges in the path. The $(i,j)$-entry of $M(\Gamma)$ is the sum of the weights of all paths from source $i$ to sink $j'$, that is,
\[
M(\Gamma)_{ij} = \sum_{p \, : \, i \rightarrow j'} \wt(p).
\]
We say that $M(\Gamma)$ is the matrix associated to $\Gamma$, and that $\Gamma$ is a network representation of $M$. For an example of a network and its associated matrix, see Figure \ref{fig_Phi_2_5}.

The {\em gluing} of networks is compatible with matrix multiplication, in the sense that if a planar network $\Gamma$ is obtained by identifying the sinks of a planar network $\Gamma_1$ with the sources of a planar network $\Gamma_2$, then
\begin{equation}
\label{eq_gluing}
M(\Gamma) = M(\Gamma_1) \cdot M(\Gamma_2).
\end{equation}

Let $I = \{i_1 < \ldots < i_m\} \subset [r]$ and $J = \{j_1 < \ldots < j_m\} \subset [s]$ be two subsets of cardinality $m$. A {\em family of paths from $I$ to $J$} is a collection of $m$ paths $p_1, \ldots, p_m$, such that $p_a$ starts at source $i_a$ and ends at sink $j'_{\sigma(a)}$, for some permutation $\sigma \in S_m$. We denote such a family by $\mathcal{F} = (p_a; \sigma)$, and we define the weight of the family by $\wt(\mathcal{F}) = \prod_{a = 1}^m \wt(p_a)$. If no two of the paths share a vertex, we say that the family is {\em vertex-disjoint}.

We refer to the following result as the Lindstr\"om Lemma.

\begin{prop}[Lindstr\"{o}m \cite{Lind}]
\label{prop_Lind}
Let $\Gamma$ be a planar network with $r$ sources and $s$ sinks, and let $I \subset [r], J \subset [s]$ be two subsets of the same cardinality. Then the minor of $M(\Gamma)$ using rows $I$ and columns $J$ is given by
\[
\Delta_{I,J}(M(\Gamma)) = \sum_{\mathcal{F} = (p_a; \sigma) \, : \, I \rightarrow J} \sgn(\sigma) \wt(\mathcal{F}),
\]
where the sum is over \underline{vertex-disjoint} families of paths from $I$ to $J$.
\end{prop}

For example, let $\Gamma$ be the network in Figure \ref{fig_Phi_2_5}. There are three vertex-disjoint families of paths from $\{3,4\}$ to $\{2',3'\}$. The weights of these families are $x_{12}x_{13}, x_{12}x_{24},$ and $x_{23}x_{24}$, and in all three cases $\sigma$ is the identity permutation. From the matrix, one computes
\[
\Delta_{34, 23}(M(\Gamma)) = x_{12}x_{13} + x_{12}x_{24} + x_{23}x_{24},
\]
in agreement with the Lindstr\"{o}m Lemma.

The networks considered below will have the property that every vertex-disjoint family of paths is of the form $(p_a; \Id)$, so the Lindstr\"om Lemma expresses every minor of the associated matrix as a polynomial in the edge weights with non-negative integer coefficients (note, however, that this is not the case for the network used in the proof of Lemma \ref{lem_one_row_minors}).

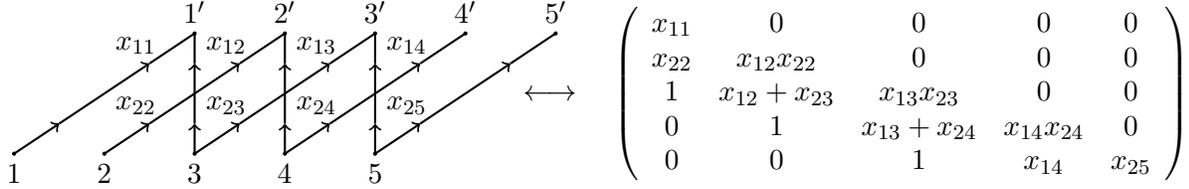
\begin{figure}

\begin{center}
\begin{tikzpicture}[scale=0.4]

\pgfmathtruncatemacro{\n}{5};
\pgfmathtruncatemacro{\k}{3};
\pgfmathtruncatemacro{\nminusk}{\n-\k};
\pgfmathtruncatemacro{\rows}{\nminusk};
\pgfmathtruncatemacro{\xscale}{3};
\pgfmathtruncatemacro{\yscale}{2};
\foreach \x in {1,...,\n}
{
\pgfmathtruncatemacro{\xx}{\xscale*\x};
\foreach \y in {1,...,\rows}
{
\pgfmathtruncatemacro{\yy}{\yscale*\y};
\pgfmathtruncatemacro{\xy}{\xscale*\y};
\pgfmathtruncatemacro{\yx}{\yscale*\x};

\ifthenelse{\x > \nminusk}{
\draw[thick,->] (\xx,\yy) -- (\xx,\yy + 0.5*\yscale);
\draw[thick] (\xx,\yy+0.5*\yscale) -- (\xx,\yy+\yscale);
}{}

\pgfmathtruncatemacro{\aa}{\nminusk - 1}
\pgfmathtruncatemacro{\yweight}{-\y+1+\x}
\pgfmathtruncatemacro{\xweight}{\nminusk - \y + 1}
\ifthenelse{\x > \aa}{
\draw (\xx+0.35*\xscale,\yy+0.80*\yscale) node {$x_{\xweight \yweight}$};
}{}

\draw[thick,->] (\xx + \xy - \xscale,\yy) -- (\xx+\xy-0.5*\xscale,\yy+0.5*\yscale);
\draw[thick] (\xx + \xy-0.5*\xscale,\yy+0.5*\yscale) -- (\xx+\xy, \yy+\yscale);
}

\filldraw (\xx,\yscale) circle [radius=.05] node[below] {$\x$};
\filldraw (\xx+ \xscale*\rows, \yscale*\rows+\yscale) circle [radius=.05] node[above] {$\x'$};
}

\draw (31,4) node {$
\longleftrightarrow \quad \left(
\begin{array}{ccccc}
x_{11} & 0 & 0 & 0 & 0  \\
x_{22} & x_{12}x_{22} & 0 & 0 & 0  \\
1 & x_{12} + x_{23} & x_{13}x_{23} & 0 & 0  \\
0 &1 & x_{13} + x_{24} & x_{14}x_{24} & 0 \\
0 & 0 & 1 & x_{14} & x_{25}
\end{array}
\right)
$};

\end{tikzpicture}
\end{center}

\caption{A planar network and its associated matrix. Unlabeled edges have weight 1.}
\label{fig_Phi_2_5}

\end{figure}

Suppose $(X_{ij},t) \in \bT{k}$, and let $x_{ij} = X_{ij}/X_{i,j-1}$ as in \eqref{eq_x_ij} (so $X_{i,i-1} := 1$ and $X_{i,i+k} := t$). Let $\Gamma_{k,n} = \Gamma_{k,n}(X_{ij},t)$ be the planar network on the vertex set $\bbZ^2$ with
\begin{itemize}
\item $n$ sinks labeled $1', \ldots, n'$, with the $j^{th}$ sink located at $(0,j)$;
\item $n$ sources labeled $1, \ldots, n$, with the $j^{th}$ source located at $(k,j-k)$;
\item a vertical\footnote{We interpret the coordinates using the convention for matrix indices.} arrow pointing from $(i,j)$ to $(i-1,j)$ for $i = 1, \ldots, k$ and $j = 1, \ldots, n-k$. The weight of this edge is 1;
\item a diagonal arrow pointing from $(i,j-i)$ to $(i-1,j-i+1)$ for $i = 1, \ldots, k$ and $j = 1, \ldots, n$. The weight of this edge is $x_{ij}$ if $0 \leq j-i \leq n-k$, and 1 otherwise.
\end{itemize}
For example, the network $\Gamma_{2,5}$ is shown in Figure \ref{fig_Phi_2_5}.

\begin{lem}
\label{lem_network_represents}
The matrix associated to $\Gamma_{k,n}(X_{ij},t)$ is the tableau matrix $\Phi_k(X_{ij},t)$ defined in \S \ref{sec_GT_param}.
\end{lem}

\begin{proof}
By definition,
\begin{equation}
\label{eq_Phi_defn_again}
\Phi_k(X_{ij},t) = \prod_{i = k}^1 M_{[i,i+n-k]}(x_{ii}, x_{i,i+1}, \ldots, x_{i,i+n-k}).
\end{equation}
It is easy to see that the $i^{th}$ ``row'' from the top of $\Gamma_{k,n}$ (i.e., the part of the network where the first coordinate is between $i-1$ and $i$) is a network representation of the $i^{th}$ factor in the right-hand side of \eqref{eq_Phi_defn_again}. The full network $\Gamma_{k,n}$ is obtained by gluing these ``rows'' together, so the result follows from \eqref{eq_gluing}.
\end{proof}

Let $\Phi^k_{I,J} : \bT{k} \rightarrow \bbC$ be the rational function $(X_{ij},t) \mapsto \Delta_{I,J}(\Phi_k(X_{ij},t))$.

\begin{lem}
\label{lem_Phi_minors}
Let $I = \{i_1 < \cdots < i_r\}$ and $J = \{j_1 < \cdots < j_r\}$ be two $r$-subsets of $[n]$, with $r \leq n-k$. Then the rational function $\Phi^k_{I,J}$ is positive if
\begin{equation}
\label{eq_Phi_minors_ineqs}
i_s - k \leq j_s \leq i_s \quad\quad \text{ for } \; s = 1, \ldots, r, 
\end{equation}
and zero otherwise.
\end{lem}

\begin{proof}
By Lemma \ref{lem_network_represents} and the Lindstr\"{o}m Lemma, $\Delta_{I,J}(\Phi_k(X_{ij},t))$ is equal to the sum of the weights of the vertex-disjoint families of paths in $\Gamma_{k,n}$ from the sources in $I$ to the sinks in $J$. Since the edge weights $x_{ij}$ are ratios of $X_{ij}, t$, the function $\Phi^k_{I,J}$ is positive if there is at least one vertex-disjoint family of paths from $I$ to $J$, and zero if there are no such families. Due to the ordering of the sources and sinks, a non-intersecting family of paths from $I$ to $J$ must have paths from $i_s$ to $j'_s$ for each $s$. There is a path from $i_s$ to $j'_s$ if and only if $i_s - k \leq j_s \leq i_s$, so \eqref{eq_Phi_minors_ineqs} is a necessary condition for $\Phi^k_{I,J}$ to be nonzero.

Suppose $I$ and $J$ satisfy \eqref{eq_Phi_minors_ineqs}. We show that $\Phi^k_{I,J}$ is positive by constructing an explicit vertex-disjoint family of paths $p_1, \ldots, p_r$ from $I$ to $J$. If $j_s = i_s$, then $p_s$ is the unique path from $i_s$ to $i'_s$. If $j_s < i_s$, set $a_s = \max(s,i_s-k)$, and let $p_s$ be the unique path from $i_s$ to $j'_s$ whose vertical steps are on the line containing the sink $a'_s$. (Note that since there are no vertical steps on the lines containing the sinks $(n-k+1)', \ldots, n'$, the assumption $s \leq r \leq n-k$ is necessary to guarantee the existence of this path.) It is easy to verify that these paths are vertex-disjoint.
\end{proof}

Suppose $I = [k+1,n]$, and $J \in {[n] \choose n-k}$ is arbitrary. It is clear that \eqref{eq_Phi_minors_ineqs} holds in this case, so the rational function $\Phi^k_{[k+1,n],J}$ is positive, and we may tropicalize it to obtain a piecewise-linear function $\wh{\Phi}^k_{[k+1,n],J} : \tT{k} \rightarrow \bbZ$.

\begin{lem}
\label{lem_highest_wt_zero}
Fix $L \geq 0$, and let $b_0$ be the classical highest weight element of the KR crystal $B^{k,L}$. For all $J \in {[n] \choose n-k}$, we have $\wh{\Phi}^k_{[k+1,n],J}(b_0) = 0$.
\end{lem}

\begin{proof}
Suppose $b = (B_{ij},L) \in B^{k,L}$, and let $b_{ij} = B_{ij} - B_{i,j-1}$ be the number of $j$'s in the $i^{th}$ row of the corresponding tableau, as in \eqref{eq_row_coords}. Let $\Gamma_{k,n}(b)$ be the network $\Gamma_{k,n}(X_{ij},t)$, but with weights $b_{ij}$ instead of $x_{ij}$, and 0 instead of 1 on unlabeled edges. By the Lindstr\"{o}m Lemma, we have
\[
\wh{\Phi}^k_{[k+1,n],J}(b) = \min_{\mathcal{F} : [k+1,n] \rightarrow J} \tw{\wt}(\mathcal{F}),
\]
where $\mathcal{F}$ runs over vertex-disjoint families of paths in $\Gamma_{k,n}(b)$ from $[k+1,n]$ to $J$, and $\tw{\wt}(\mathcal{F})$ is the sum of the weights of the edges in the paths.

The element $b_0 \in B^{k,L}$ corresponds to the SSYT of shape $(L^k)$ whose $i^{th}$ row is filled with the number $i$, so we have
\[
(b_0)_{ij} = \begin{cases}
L & \text{ if } i = j \\
0 & \text{ otherwise}.
\end{cases}
\]
Thus, the only edges in $\Gamma_{k,n}(b_0)$ with nonzero weights are to the left of source $k+1$, so all edges in the paths that contribute to $\wh{\Phi}^k_{[k+1,n],J}(b_0)$ have weight zero.
\end{proof}

\bibliographystyle{plain}
\bibliography{R_matrix_paper.refs}

\end{document}